\documentclass[12pt]{article}
\usepackage{amsthm, amsmath, amssymb, enumerate}
\usepackage{fullpage}
\usepackage{amscd}
\usepackage[colorlinks=true]{hyperref}
\usepackage{MnSymbol}

\begin{document}
\title{On the Averaged Colmez Conjecture}
\author{Xinyi Yuan and  Shou-Wu Zhang}
\date{July 23, 2015}

\maketitle

\theoremstyle{plain}

\newtheorem{thm}{Theorem}[section]
\newtheorem{theorem}[thm]{Theorem}
\newtheorem{cor}[thm]{Corollary}
\newtheorem{corollary}[thm]{Corollary}
\newtheorem{lem}[thm]{Lemma}
\newtheorem{lemma}[thm]{Lemma}
\newtheorem{pro}[thm]{Proposition}
\newtheorem{proposition}[thm]{Proposition}
\newtheorem{prop}[thm]{Proposition}
\newtheorem{definition}[thm]{Definition}
\newtheorem{assumption}[thm]{Assumption}

\newtheorem*{thmm}{Theorem}
\newtheorem*{conj}{Conjecture}
\newtheorem*{notation}{Notation}
\newtheorem*{corr}{Corollary}

\theoremstyle{remark} 
\newtheorem{remark}[thm]{Remark}
\newtheorem{example}[thm]{Example}
\newtheorem{remarks}[thm]{Remarks}
\newtheorem{problem}[thm]{Problem}
\newtheorem{exercise}[thm]{Exercise}
\newtheorem{situation}[thm]{Situation}

\numberwithin{equation}{subsection}

\newcommand{\ZZ}{\mathbb{Z}}
\newcommand{\QQ}{\mathbb{Q}}
\newcommand{\RR}{\mathbb{R}}
\newcommand{\HH}{\mathcal{H}}     

\newcommand{\inverse}{^{-1}}          

\newcommand{\gl}{\mathrm{GL}_2}             
\newcommand{\sll}{\mathrm{SL}_2}            
\newcommand{\adele}{\mathbb{A}}             
\newcommand{\finiteadele}{\mathbb{A}_f}     
\newcommand{\af}{\mathbb{A}_f}              
\newcommand{\across}{\mathbb{A}^{\times}}      
\newcommand{\afcross}{\mathbb{A}_f^{\times}}   

\newcommand{\gla}{\mathrm{GL}_2(\mathbb{A})}       
\newcommand{\glaf}{\mathrm{GL}_2(\mathbb{A}_f)}    
\newcommand{\glf}{\mathrm{GL}_2(F)}                
\newcommand{\glfv}{\mathrm{GL}_2(F_v)}             
\newcommand{\glofv}{\mathrm{GL}_2(O_{F_v})}        

\newcommand{\ofv}{O_{F_v}}                         
\newcommand{\oev}{O_{E_v}}                         
\newcommand{\evcross}{E_v^{\times}}                
\newcommand{\fvcross}{F_v^{\times}}                

\newcommand{\adelef}{\mathbb{A}_F}             
\newcommand{\adelee}{\mathbb{A}_E}             
\newcommand{\aecross}{\mathbb{A}_E^{\times}}   

\newcommand{\OCS}{\overline{\mathcal{S}}}
\renewcommand{\Pr}{\mathcal{P}r}

\newcommand{\jv}{\mathfrak{j}_v}
\newcommand{\fg}{\mathfrak{g}}

\newcommand{\vv}{\mathbb{V}}                      
\newcommand{\bb}{\mathbb{B}}                      
\newcommand{\bbf}{\mathbb{B}_f}                   
\newcommand{\bfcross}{\mathbb{B}_f^\times}        
\newcommand{\ba}{\mathbb{B}_{\mathbb{A}}}          
\newcommand{\baf}{\mathbb{B}_{\mathbb{A}_f}}       
\newcommand{\bv}{{\mathbb{B}_v}}                   
\newcommand{\bacross}{\mathbb{B}_{\mathbb{A}}^{\times}}      
\newcommand{\bafcross}{\mathbb{B}_{\mathbb{A}_f}^{\times}}   
\newcommand{\bvcross}{\mathbb{B}_v^{\times}}               

\newcommand{\ad}{_\mathrm{ad}}            
\newcommand{\NT}{\mathrm{NT}}         
\newcommand{\nonsplit}{\mathrm{nonsplit}}         
\newcommand{\Pet}{\mathrm{Pet}}         
\newcommand{\Norm}{\mathrm{Norm}}         

\newcommand{\lb}{\mathcal{L}}   
\newcommand{\DD}{\mathcal{D}}   

\newcommand{\quasilim}{\widetilde\lim_{s\rightarrow 0}}   
\newcommand{\pr}{\mathcal{P}r}   

\newcommand{\CMU}{\mathrm{CM}_U}             
\newcommand{\eend}{\mathrm{End}}             
\newcommand{\eendd}{\mathrm{End}^{\circ}}    

\newcommand{\sumu}{\sum_{u \in \mu_U^2 \bs F\cross}}

\newcommand{\supp}{\mathrm{supp}}
\newcommand{\cross}{^{\times}}
\newcommand{\der}{\frac{d}{ds}|_{s=0}}   

\newcommand{\pair}[1]{\langle {#1} \rangle}
\newcommand{\wpair}[1]{\left\{{#1}\right\}}
\newcommand\wh{\widehat}
\newcommand\Spf{\mathrm{Spf}}
\newcommand{\lra}{{\longrightarrow}}

\newcommand{\Ei}{\mathrm{Ei}} 

\newcommand{\sumyu}{\sum_{(y,u)}}

\newcommand{\matrixx}[4]
{\left( \begin{array}{cc}
  #1 &  #2  \\
  #3 &  #4  \\
 \end{array}\right)}        

\newcommand{\barint}{\mbox{$ave \int$}}  
\def\barint_#1{\mathchoice
            {\mathop{\vrule width 6pt
height 3 pt depth -2.5pt
                    \kern -8.8pt
\intop}\nolimits_{#1}}%
            {\mathop{\vrule width 5pt height
3 pt depth -2.6pt
                    \kern -6.5pt
\intop}\nolimits_{#1}}%
            {\mathop{\vrule width 5pt height
3 pt depth -2.6pt
                    \kern -6pt
\intop}\nolimits_{#1}}%
            {\mathop{\vrule width 5pt height
3 pt depth -2.6pt
          \kern -6pt \intop}\nolimits_{#1}}}

\newcommand{\BA}{{\mathbb {A}}}
\newcommand{\BB}{{\mathbb {B}}}
\newcommand{\BC}{{\mathbb {C}}}
\newcommand{\BD}{{\mathbb {D}}}
\newcommand{\BE}{{\mathbb {E}}}
\newcommand{\BF}{{\mathbb {F}}}
\newcommand{\BG}{{\mathbb {G}}}
\newcommand{\BH}{{\mathbb {H}}}
\newcommand{\BI}{{\mathbb {I}}}
\newcommand{\BJ}{{\mathbb {J}}}
\newcommand{\BK}{{\mathbb {K}}}
\newcommand{\BL}{{\mathbb {L}}}
\newcommand{\BM}{{\mathbb {M}}}
\newcommand{\BN}{{\mathbb {N}}}
\newcommand{\BO}{{\mathbb {O}}}
\newcommand{\BP}{{\mathbb {P}}}
\newcommand{\BQ}{{\mathbb {Q}}}
\newcommand{\BR}{{\mathbb {R}}}
\newcommand{\BS}{{\mathbb {S}}}
\newcommand{\BT}{{\mathbb {T}}}
\newcommand{\BU}{{\mathbb {U}}}
\newcommand{\BV}{{\mathbb {V}}}
\newcommand{\BW}{{\mathbb {W}}}
\newcommand{\BX}{{\mathbb {X}}}
\newcommand{\BY}{{\mathbb {Y}}}
\newcommand{\BZ}{{\mathbb {Z}}}

\newcommand{\CA}{{\mathcal {A}}}
\newcommand{\CB}{{\mathcal {B}}}
\renewcommand{\CD}{{\mathrm{D}}}
\newcommand{\CE}{{\mathcal {E}}}
\newcommand{\CF}{{\mathcal {F}}}
\newcommand{\CG}{{\mathcal {G}}}
\newcommand{\CH}{{\mathcal {H}}}
\newcommand{\CI}{{\mathcal {I}}}
\newcommand{\CJ}{{\mathcal {J}}}
\newcommand{\CK}{{\mathcal {K}}}
\newcommand{\CL}{{\mathcal {L}}}
\newcommand{\CM}{{\mathcal {M}}}
\newcommand{\CN}{{\mathcal {N}}}
\newcommand{\CO}{{\mathcal {O}}}
\newcommand{\CP}{{\mathcal {P}}}
\newcommand{\CQ}{{\mathcal {Q}}}
\newcommand{\CR }{{\mathcal {R}}}
\newcommand{\CS}{{\mathcal {S}}}
\newcommand{\CT}{{\mathcal {T}}}
\newcommand{\CU}{{\mathcal {U}}}
\newcommand{\CV}{{\mathcal {V}}}
\newcommand{\CW}{{\mathcal {W}}}
\newcommand{\CX}{{\mathcal {X}}}
\newcommand{\CY}{{\mathcal {Y}}}
\newcommand{\CZ}{{\mathcal {Z}}}

\newcommand{\RA}{{\mathrm {A}}}
\newcommand{\RB}{{\mathrm {B}}}
\newcommand{\RC}{{\mathrm {C}}}
\newcommand{\RD}{{\mathrm {D}}}
\newcommand{\RE}{{\mathrm {E}}}
\newcommand{\RF}{{\mathrm {F}}}
\newcommand{\RG}{{\mathrm {G}}}
\newcommand{\RH}{{\mathrm {H}}}
\newcommand{\RI}{{\mathrm {I}}}
\newcommand{\RJ}{{\mathrm {J}}}
\newcommand{\RK}{{\mathrm {K}}}
\newcommand{\RL}{{\mathrm {L}}}
\newcommand{\RM}{{\mathrm {M}}}
\newcommand{\RN}{{\mathrm {N}}}
\newcommand{\RO}{{\mathrm {O}}}
\newcommand{\RP}{{\mathrm {P}}}
\newcommand{\RQ}{{\mathrm {Q}}}
\newcommand{\RS}{{\mathrm {S}}}
\newcommand{\RT}{{\mathrm {T}}}
\newcommand{\RU}{{\mathrm {U}}}
\newcommand{\RV}{{\mathrm {V}}}
\newcommand{\RW}{{\mathrm {W}}}
\newcommand{\RX}{{\mathrm {X}}}
\newcommand{\RY}{{\mathrm {Y}}}
\newcommand{\RZ}{{\mathrm {Z}}}

\newcommand{\ga}{{\frak a}}
\newcommand{\gb}{{\frak b}}
\newcommand{\gc}{{\frak c}}
\newcommand{\gd}{{\frak d}}
\newcommand{\gf}{{\frak f}}
\newcommand{\gh}{{\frak h}}
\newcommand{\gi}{{\frak i}}
\newcommand{\gj}{{\frak j}}
\newcommand{\gk}{{\frak k}}
\newcommand{\gm}{{\frak m}}
\newcommand{\gn}{{\frak n}}
\newcommand{\go}{{\frak o}}
\newcommand{\gp}{{\frak p}}
\newcommand{\gq}{{\frak q}}
\newcommand{\gr}{{\frak r}}
\newcommand{\gs}{{\frak s}}
\newcommand{\gt}{{\frak t}}
\newcommand{\gu}{{\frak u}}
\newcommand{\gv}{{\frak v}}
\newcommand{\gw}{{\frak w}}
\newcommand{\gx}{{\frak x}}
\newcommand{\gy}{{\frak y}}
\newcommand{\gz}{{\frak z}}

\newcommand{\ab}{{\mathrm{ab}}}
\newcommand{\Ad}{{\mathrm{Ad}}}
\newcommand{\an}{{\mathrm{an}}}
\newcommand{\Aut}{{\mathrm{Aut}}}

\newcommand{\Br}{{\mathrm{Br}}}
\newcommand{\bs}{\backslash}
\newcommand{\bbs}{\|\cdot\|}

\newcommand{\Ch}{{\mathrm{Ch}}}
\newcommand{\cod}{{\mathrm{cod}}}
\newcommand{\cont}{{\mathrm{cont}}}
\newcommand{\cl}{{\mathrm{cl}}}
\newcommand{\criso}{{\mathrm{criso}}}

\newcommand{\dR}{{\mathrm{dR}}}
\newcommand{\disc}{{\mathrm{disc}}}
\newcommand{\Div}{{\mathrm{Div}}}
\renewcommand{\div}{{\mathrm{div}}}

\newcommand{\Eis}{{\mathrm{Eis}}}
\newcommand{\End}{{\mathrm{End}}}

\newcommand{\Frob}{{\mathrm{Frob}}}

\newcommand{\Gal}{{\mathrm{Gal}}}
\newcommand{\GL}{{\mathrm{GL}}}
\newcommand{\GO}{{\mathrm{GO}}}
\newcommand{\GSO}{{\mathrm{GSO}}}
\newcommand{\GSp}{{\mathrm{GSp}}}
\newcommand{\GSpin}{{\mathrm{GSpin}}}
\newcommand{\GU}{{\mathrm{GU}}}
\newcommand{\BGU}{{\mathbb{GU}}}

\newcommand{\Hom}{{\mathrm{Hom}}}
\newcommand{\Hol}{{\mathrm{Hol}}}
\newcommand{\HC}{{\mathrm{HC}}}

\renewcommand{\Im}{{\mathrm{Im}}}
\newcommand{\Ind}{{\mathrm{Ind}}}
\newcommand{\inv}{{\mathrm{inv}}}
\newcommand{\Isom}{{\mathrm{Isom}}}

\newcommand{\Jac}{{\mathrm{Jac}}}
\newcommand{\JL}{{\mathrm{JL}}}

\newcommand{\Ker}{{\mathrm{Ker}}}
\newcommand{\KS}{{\mathrm{KS}}}

\newcommand{\Lie}{{\mathrm{Lie}}}

\newcommand{\new}{{\mathrm{new}}}
\newcommand{\NS}{{\mathrm{NS}}}

\newcommand{\ord}{{\mathrm{ord}}}
\newcommand{\ol}{\overline}

\newcommand{\rank}{{\mathrm{rank}}}

\newcommand{\PGL}{{\mathrm{PGL}}}
\newcommand{\PSL}{{\mathrm{PSL}}}
\newcommand{\Pic}{\mathrm{Pic}}
\newcommand{\Prep}{\mathrm{Prep}}
\newcommand{\Proj}{\mathrm{Proj}}

\renewcommand{\Re}{{\mathrm{Re}}}
\newcommand{\red}{{\mathrm{red}}}
\newcommand{\reg}{{\mathrm{reg}}}
\newcommand{\Rep}{{\mathrm{Rep}}}
\newcommand{\Res}{{\mathrm{Res}}}

\newcommand{\Sel}{{\mathrm{Sel}}}
\font\cyr=wncyr10  \newcommand{\Sha}{\hbox{\cyr X}}
\newcommand{\SL}{{\mathrm{SL}}}
\newcommand{\SO}{{\mathrm{SO}}}
\newcommand{\Sp}{\mathrm{Sp}}
\newcommand{\Spec}{{\mathrm{Spec}}}
\newcommand{\Sym}{{\mathrm{Sym}}}
\newcommand{\sgn}{{\mathrm{sgn}}}
\newcommand{\Supp}{{\mathrm{Supp}}}

\newcommand{\tor}{{\mathrm{tor}}}
\newcommand{\tr}{{\mathrm{tr}}}

\newcommand{\ur}{{\mathrm{ur}}}

\newcommand{\vol}{{\mathrm{vol}}}

\newcommand{\wt}{\widetilde}
\newcommand{\pp}{\frac{\partial\bar\partial}{\pi i}}
\newcommand{\intn}[1]{\left( {#1} \right)}
\newcommand{\norm}[1]{\|{#1}\|}
\newcommand{\sfrac}[2]{\left( \frac {#1}{#2}\right)}
\newcommand{\ds}{\displaystyle}
\newcommand{\ov}{\overline}
\newcommand{\incl}{\hookrightarrow}
\newcommand{\imp}{\Longrightarrow}
\newcommand{\lto}{\longmapsto}
\newcommand{\iso}{\overset \sim \lra}

\tableofcontents

\section{Introduction}
The Colmez conjecture is a formula expressing the Faltings height of an abelian variety with complex multiplication in terms of some linear combination of logarithmic derivatives of Artin L-functions.
The aim of this paper to prove an averaged version of the conjecture, which was also proposed by Colmez.

\subsection{Statements}

First let us recall the definition of Faltings heights introduced by Faltings \cite{Fa}.
Let $A$ be an abelian variety of dimension $g$ over a number field $K$, and $\CA$  the  relative indentity component of the N\'eron model of $A$ over $O_K$.
Assume that $\CA$ is semi-abelian.  Denote by $\Omega (\CA)=\Lie (\CA)^\vee$ the sheaf of invariant differential $1$-forms on $\CA$.
Let $\bar \omega (\CA)$ be a metrized line bundle over $\Spec\, O_K$, whose finite part is defined as
 $$\omega (\CA):=\det\Omega  (\CA),$$
and whose metric $\|\cdot\|_v$ 
at each  archimedean place $v$ of $K$ is given by
$$\|\alpha \|_v^2:=\frac{1}{(2\pi)^{g}}\int _{A_v (\BC)}|\alpha\wedge \bar \alpha|, \qquad \alpha \in \omega (A_v)=\Gamma (A_v, \Omega^g_{A_v}).$$ 
 Then Faltings \cite[\S3]{Fa}  defines a moduli-theoretic height $h(A)$  by
 $$h(A):=\frac 1{[K:\BQ]}\wh\deg\ \ol\omega (\CA).$$
 Since $\CA$ is semi-abelian, this height is invariant under base change.
 
 Now let us state our main result as conjectured by Colmez.
 Let $E$ be a CM field of degree $[E:\BQ]=2g$,  with the maximal  totally real subfield $F$ and a complex conjugation $c:E\to E$.
 Let $\Phi\subset \Hom (E, \BC)$ be a CM-type, i.e., a subset such that 
 $\Phi\cap \Phi^c =\emptyset$ and
 $\Phi\cup \Phi^c =\Hom (E, \BC)$. Let $A_\Phi$ be a CM abelian variety over $\BC$ of CM type $(O_E,\Phi)$. 
By the theory of complex multiplication, there is a number field $K$ in $\BC$ such that $A_\Phi$ is defined over $K$ and has a smooth and projective integral model $\CA$ over $O_K$. 
Colmez proved that the  height $h(A_\Phi)$  depends only 
on the CM-type $\Phi$. Thus we may denote  this height by $h(\Phi)$.

Colmez gave a conjectural formula expressing the precise value of $h(A_\Phi)$ in terms of linear combinations of logarithmic derivatives of Artin L-functions determined by $\Phi$. See \cite[Thm. 0.3, Conj. 0.4]{Co}.
When $E/\BQ$ is abelian,
the conjecture was proved up to rational multiples of $\log 2$ in the same paper, and later the rational multiples were eliminated by Obus \cite{Ob}.
When $[E:\BQ]=4$, the conjecture was essentially proved by Yang \cite{Ya1, Ya2}.

The goal of this paper is to prove the following averaged formula for general CM fields using techniques in the proof of the  Gross--Zagier formula (\cite{GZ}) and its generalization
(\cite{YZZ}).

\begin{thm}\label{mainthm} 
Let $E/F$ be a CM extension,  $\eta=\eta _{E/F}$ be the corresponding quadratic character of $\BA_F^\times $,
and $d_F$ (resp.  $d_{E/F}$) be the absolute discriminant  of $F$ (resp. the norm of the relative  discriminant of $E/F$). Then 
$$\frac 1{2^g}\sum _\Phi h(\Phi)
=-\frac 12\frac {L'_f(0,\eta)}{L_f(0,\eta)} -\frac 14\log (d_{E/F}d_F),
$$
where $\Phi$ runs through the set of all CM types of $E$, and $L_f(s,\eta)$ is the finite part of the completed L-function $L(s,\eta)$. 
\end{thm}

The averaged formula was explicitly stated in \cite[p. 634]{Co} with some typo. Note that we use a different normalization of the Faltings height.

\begin{remark}
Note that the above theorem can be reformulated as  an arithmetic expression  for $L'(0, \eta)$.
This expression is 
 analogous to the class number formula 
$$L(0, \eta)=2^a\frac Hw$$
where $2^a$, $H$, and $w$ are respectively the ratios of regulators, class numbers, and the number of roots of unity of the fields $E$ and $F$. 
\end{remark}

\begin{remark} 
 When $E$ is imaginary quadratic, the Colmez  conjecture can be 
deduced from the   Chowla--Selberg formula in \cite{CS} in 1967. Our method (and also the method of Yang \cite{Ya1, Ya2})
 thus give a different proof of the Chow--Selberg formula.
 Another very interesting geometric proof of the Chowla--Selberg formula was discovered by Gross
 \cite{Gr2}. He also made a conjecture with Deligne for the periods of motives with CM 
by an abelian field. 
Anderson \cite{An} reformulated the conjecture of  Deligne and Gross  in terms
of the logarithmic derivatives of odd Dirichlet L-functions at $s=0$.
All these  predictions were only up to algebraic numbers. Colmez used the Faltings 
height instead of just the archimedean periods, to make the conjectures precise. 
 \end{remark}

\begin{remark} 
Shortly after we posted  our paper on arXiv, a different  proof of the averaged formula modulo a rational multiple of $\log 2$ has been 
posted on the arXiv by Andreatta, Goren, Howard and Madapusi-Pera in \cite{AGHM}. In a more recent version, they have removed the ambiguity of $\log 2$, and thus their final result is the same as ours. Their proof uses integral models of  high-dimensional Shimura varieties and is based on the method of Yang \cite{Ya1, Ya2}.
 \end{remark}

\begin{remark}
By the recent work of Jacob Tsimerman \cite{Ts}, the theorem implies the Andre--Oort conjecture for Siegel abelian varieties:
{\em Let $X$ be a Shimura variety of abelian type over $\BC$.
Let $Y\subset X$ be a closed subvariety which contains a Zariski dense subset of special points of $X$. 
Then $Y$ is a special subvariety.}
\end{remark}

Theorem \ref{mainthm} is a direct consequence of Theorem \ref{U2O} and Theorem \ref{quaternion main} below. The proof of each of the latter two theorems forms a part of this paper, so this paper is naturally divided into two parts. Theorem \ref{U2O} is proved in Part I;  
Theorem \ref{quaternion main} is proved in Part II.

\subsection{Faltings heights}

Part I (\S2-\S5) of this paper is devoted to reducing Theorem \ref{mainthm} to a Gross--Zagier type  formula  on quaternionic Shimura curves. In the following, for quaternionic Shimura curves, Hodge bundles and CM points, we will use the terminology of \cite[\S1.2, \S1.3, \S3.1]{YZZ}

Fix a CM extension $E/F$ as above. 
Let $\BB$ be a totally definite incoherent quaternion algebra over $\BA:=\BA_F$. 
Assume that there is an embedding $\BA_E\hookrightarrow \bb$ over $\BA$ and fix one throughout this paper.
For each   open compact subgroup $U$ of $\BB_f^\times$,  we have a Shimura curve $X_U$, which is a projective and smooth curve over $F$. 
Let $X$ be the projective limit of $X_U$. Then $X$ has a right action by $\BB_f^\times$ with quotients $X/U=X_U$.

The Shimura curves $X_U$    do not parametrize abelian varieties but   can be embedded into  Shimura curves of PEL types over $\bar F$.
We will  construct   integral  models $\CX_U$ following the work of  Carayol \cite{Ca} and \v Cerednik--Drinfeld \cite{BC}  and define the Hodge bundle $\CL_U$
(Theorem \ref{hodge bundles}).

Assume that $U=\prod U_v$ is  a maximal compact subgroup of $\BB_f^\times$ containing $\wh O_E^\times$.
Then $X_U$ has a canonical integral model $\CX_U$ over $O_F$.
Let $\bar\CL_U$ be the arithmetic Hodge bundle of $\CX_U$, whose hermitian metric at an archimedean place $v$ is given by 
$$\|dz\|_v=2\, \Im(z)$$
with respect to the usual complex uniformizations by coherent quaternion algebras. 
See \S \ref{section integral models} for the constructions of $\CX_U$ and $\bar\CL_U$. 

Let $P_U\in X_U(E^\ab)$ be the image of a point  $P\in X^{E^\times}$.
It has a  height defined by
$$h_{\bar\CL_U}(P_U):=\frac{1}{[F(P_U):F]}\wh\deg(\bar\CL_U|_{\bar P_U}),$$
where $\bar P_U$ denotes the Zariski closure of the image of $P_U$ in $\CX_U$. 
The first part of our paper is to relate this height to 
the average of the Faltings heights of CM abelian varieties.

\begin{thm}\label{U2O} 
Let $d_\BB$ be the norm of the product of finite primes of $O_F$ over which $\BB$ is ramified.
 Assume that there is no finite place of $F$ ramified in both $E$ and $\BB$. Then 
$$\frac 1{2^g}\sum _\Phi h(\Phi)= \frac 12 h_{\ol \CL_U}(P_U)-\frac 14 \log (d_\BB d_F ).$$
\end{thm}
We prove this theorem by several manipulations of heights, which are sketched in the following. 

\subsubsection*{Decomposition of Faltings heights}
Let $K\subset \BC$ be a number field containing the normal closure of $E$ over $\BQ$ such that any CM abelian variety 
by $O_E$ has a smooth model over $O_K$. 
Let $A/K$ be a CM abelian variety of type $(O_E,\Phi)$ and $\CA/O_K$ be the smooth projective integral model.
Then we will   decompose the height $h(\Phi)$ into a sum of 
$g$ terms indexed by  $\tau \in \Phi$, 
$$h(\Phi, \tau)=\frac 12 \wh{\deg}\, \bar \CN (A, \tau)$$ 
where each $\bar \CN (A, \tau)$ is a hermitian line bundle over $\Spec\, O_K$.
We will show that this height depends only on the pair  $(\Phi, \tau)$ in Theorem \ref{hAtau}, and then denote it as 
$h(\Phi, \tau)$.
In Theorem \ref{PhitauPsi},  we obtain
$$h(\Phi)-\sum _{\tau\in \Phi} h(\Phi, \tau)
=-\frac 1{4[E_\Phi: \BQ]}\log (d_\Phi d_{\Phi^c}).$$
Here $E_\Phi$ is the reflex field of $(E, \Phi)$ and $d_\Phi, d_{\Phi^c}$ are certain absolute discriminants of $\Phi, \Phi ^c$. 

Let  $(\Phi_1, \Phi_2)$ be a nearby pair of CM types of $E$ in the sense that  $|\Phi_1\cap \Phi_2|=g-1$.
Let $\tau_i$ be the complement of $\Phi_1\cap \Phi_2$ in $\Phi_i$ for $i=1,2$. 
Define
$$h(\Phi_1, \Phi_2)=\frac 12(h(\Phi_1, \tau_1)+h(\Phi_2, \tau_2))$$
We will show that $h(\Phi_1, \Phi_2)$ does not depend on the choice of $(\Phi_1, \Phi_2)$ and 
that $h(\Phi_1, \Phi_2)$ is equal to $\frac 12 h(A_0, \tau)$ for any abelian variety $A_0$ with an action by $O_E$
and isogenous to $A_{\Phi_1}\times A_{\Phi_2}$, where $\tau=\tau_i|_F$.
See Theorem \ref{specialA}. Thus Theorem \ref{U2O} is reduced to the following equality:
$$
gh(A_0, \tau)=h_{\overline \CL_U}(P_U)-\frac 12 \log (d_\BB).$$
Assume that $A$ is defined over the number field $K$ containing $F(P_U)$ and has good reduction over $O_K$. 
We will prove the above identity by constructing an isomorphism of hermitian line bundles over $\Spec\,O_K$ (cf. Proposition \ref{A-P}):
\begin{equation}\label{equation-AP}\ol \CN(A_0, \tau)\iso 
(\ol \CN _U|_{P_U})\otimes_{O_{F(P_U)} }O_{K},\end{equation}
where $\overline \CN_U:=\overline \CL_U^2(-\gd _\BB)$ is a $\BQ$-bundle over $\CX_U$.

\subsubsection*{Kodaira--Spencer isomorphisms}
We will construct the isomorphism \ref{equation-AP} by applying Kodaira--Spencer maps for families  of abelian varieties, Hodge structures, and $p$-divisible groups parametrized by various Shimura curves. These maps give relations ``$N=\omega ^{\otimes 2}$" between invariant differentials of these objects and differentials of the base curves. 

First of all, let $(\Phi_1, \Phi_2)$ be a nearby pair of  CM types of $E$.
Let $F' $ be the reflex field of $\Phi_1+\Phi_2$. 
Then there is a PEL-type Shimura curve $X'_{U'}$ with minimal level defined over $F'$ parametrizing the quadruples $(A, i, \theta, \kappa)$
of an abelian variety $A$,  an action $i$ of $O_E$ on $A$ of type $\Phi_1+\Phi_2$, a polarization $\theta: A\lra A^t$ inducing complex conjugation on $E$,
and a level structure $\kappa: O_\BB\iso \wh T(A)$. 
On $X'_{U'}$ there is a point $P'_{U'}$ representing  an abelian variety $A_0$ which is isogenous to $A_{\Phi_1}\times A_{\Phi_2}$.
By the Kodaira--Spencer map, there is an isomorphism
$$N(A_0, \tau)\simeq \omega _{X'_{U'}, P_{U'}'}^{\otimes 2}.$$
We will prove an archiemdean Kodaira--Spencer isomorphism (Theorem \ref{N2Linf}) in terms of  hermitian structures  using complex uniformization of $X'$.

There is no natural maps between the Shimura curves $X_U$ and $X'_{U'}$ over the reflex fields, even though they have isomorphic  connected components over $\bar F$.
We will construct another  Shimura curve $X''_{U''}$ with morphisms $X_U\lra X''_{U''}$ and $X'_{U'}\lra X''_{U''}$ so that both point $P_U$ and $P'_{U'}$ have the same 
image $P''_{U''}$.  This gives an isomorphism  over $K$ required in 
(\ref{equation-AP}): 
\begin{equation}\label{K-AP}N(A_0, \tau)\iso N _{P_U}\otimes_{F(P_U)}K.\end{equation}
This isomorphism is in fact an isometry at all archimdean places. 

It remains to show that the isomorphism (\ref{K-AP}) extends to the isomorphism (\ref{equation-AP}). We need only do this by working on every place of $K$.
  For each prime $\wp'$ of $F$, there is a $p$-divisible group $H''$
on certain infinite cover  $X_{1, \wp'}''$ of $X''_{U''}$ defined over $K':=F_{\wp'}^\ur$, the completion of the maximal unramified extension of $F_{\wp'}$. 
This group restricts to  the $p$-divisible group $H':=A[p^\infty]$ on $X'_{1, \wp'}$, an infinite cover of $X'_{U'}$. 
On the other hand, on an infinite cover $X_{1, \wp}$ of $X_U$ over $K:=F_\wp^\ur$, where $\wp:=\wp'|_F$,
 there is a $p$-divisible group $H$ independent of the choice of $E$. The groups $H$ and $H''$  are related   by the Tate module of a $p$-divisible group $I$ on $Y$. See
Proposition \ref{TTT}.

We will  give a description for $\CN _{1, \wp}$ in terms of the deformation of $\CH$  via a  Kodaira--Spencer isomorphism (Theorem \ref{N2Lfin}). 
By Proposition \ref{TTT}, this also gives a description of $\CN_{1, \wp}\otimes O_{F_{\wp'}'^\ur}$ in term of the deformation of $\CH''$ (Corollary \ref{n2n})
which is the required extension of the isomorphism (\ref{equation-AP}) at places over $\wp'$.

\subsection{Quaternionic heights}

Part II (\S6-\S9) of this paper is devoted to the proof the following height formula on quaternionic Shimura curves. 
Let  $U=\prod_v U_v$ be  a maximal open compact subgroup of $\BB_f^\times$ containing the image of $\wh O_{E}^\times=\prod_v O_{E_v}^\times$.

\begin{thm}\label{quaternion main}
Assume that at least two places of $F$ are ramified in $\BB$, and that  
there is no non-archimedean place of $F$ ramified in both $E$ and $\BB$. Then
$$ h_{\bar\CL_U}(P_U)
=-\frac{L_f'(0,\eta)}{L_f(0,\eta)} +\frac 12 \log  \frac{d_\BB}{d_{E/F}}.
$$
Here $d_\BB=\RN (\gd_\BB)$ is the absolute discriminant of $\BB$.
\end{thm}

We prove this theorem by extending our method of proving the Gross--Zagier formula in \cite{YZZ}.
Recall that the Gross--Zagier formula is an identity between the derivative of $L$-series of 
a Hilbert modular form  and the height of a CM point on a modular abelian variety. This formula is proved 
by a comparison  of the analytic  kernel  $\Pr I'(0, g, \phi)$ and a geometric kernel $2Z(g, (1,1), \phi)$ parametrized by certain modified 
Schwartz function $\phi\in \ol\CS (\BB\times \BA^\times )$. More precisely, we have proved that  
the difference 
$$\mathcal D (g, \phi)=\Pr I'(0, g, \phi)-2Z(g, (1,1), \phi), \qquad g\in \GL_2(\BA_F)$$
is perpendicular to the relevant cusp forms.

The cancellation for the ``main terms'' of $\mathcal D (g, \phi)$ eventually imply the Gross--Zagier formula; however, the cancellation of the ``degenerate terms'' imply Theorem \ref{quaternion main}.
To retrieve information of these degenerate terms,  we need to compute this difference for a wider class of Schwartz functions $\phi$ than those considered in \cite{YZZ}. In fact, \cite{YZZ} makes some assumptions on $\phi$ so that the degenerate terms vanish automatically. 
In the following, we sketch some new ingredients of the proof here.

\subsubsection*{Derivative series}
By the reduced norm $q$, the incoherent quaternion algebra $\BB$ is viewed as a quadratic space over $\BA=\BA_F$. 
Then we have a modified space $\ol\CS (\BB\times \BA^\times)$ of Schwartz functions with a Weil representation $r$ by $\GL_2(\BA)\times\BB^\times\times \BB^\times$. 
For each  $\phi\in \ol\CS (\BB\times \BA^\times )$ invariant under an open compact subgroup $U\times U$ of $\BB_f^\times\times \BB_f^\times$, 
we have a finite sum  of products of theta series and  Eisenstein series
$$I(s, g, \phi)_U=\sum _{u\in \mu _U^2\bs F^\times} \sum _{\gamma \in P^1(F)\bs \SL_2(F)}
\delta (\gamma g)^s \sum _{x_1\in E} r(\gamma g)\phi (x_1, u),$$
where $\mu _U=F^\times \cap U$, and $P^1$ is the upper triangular  subgroup of $\SL_2$.

For the decomposition $\BB=E_\BA+E_\BA \mathfrak j$, 
this function is a linear combination of the products $\theta (g, \phi_1)\cdot E (s, g, \phi_2)$ of the theta series $\theta (g, \phi_1)$
for some coherent Schwartz functions  $\phi_1\in \CS(E_\BA)$,
and the Eisenstein series $E(s, g, \phi_2)$ for some incoherent Schwartz functions $\phi_2\in \CS (E_\BA j)$.  This implies that $I(0, g, \phi)=0$.
Let $\Pr I'(0, g, \phi)$ be the holomorphic projection of the derivative at $s=0$ of $I(s, g, \phi)$. 

In Theorem \ref{derivative series}, we give a precise formula for  $\Pr I'(0, g, \phi)$ under some assumptions of Schwartz functions,
which particularly includes the following term:
\begin{equation}\label{1111}
\left(2\frac {L_f'(0, \eta)}{L_f (0, \eta)}+\log |d_{E/F}d_F|\right)\sum _{\mu _U^2\bs F^\times}\sum _{y\in E^\times}\phi (y, u).
\end{equation}
Notice that this term was killed in \cite{YZZ} by some stronger assumption on Schwartz functions.

\subsubsection*{Height series}
For any $\phi\in \ol\CS (\BB\times \BA^\times)$ invariant under $U\times U$, we have a generating series of Hecke operators on the Shimura curve $X_U$:
$$Z(g, \phi)_U=Z_0(g, \phi)+w_U\sum _{a\in F^\times}\sum _{x\in U\bs \BB_f^\times /U}
r(g)\phi (x, aq(x)^{-1})Z(x)_U,$$
where $w_U=|\mu _2\cap U|$, the constant term $Z_0(g, \phi)$ is a linear combination of Hodge classes on $X_U\times X_U$, which can be neglected in this paper,  and every $Z(x)_U$ is a divisor of $X_U\times X_U$ associated to the Hecke operator corresponding to the double coset $UxU$. 
By \cite[Theorem 3.17]{YZZ}, this series is absolutely convergent and defines an automorphic form on $g\in \GL_2(\BA)$ with coefficients in $\Pic (X_U\times X_U)_\BC$.

Let $P=P_U$ be the CM point of $X_U$ as above, and $P_U^\circ\in \Jac (X_U)$ 
be the divisor of degree zero modified by the Hodge classes.
Then we can form a height series 
$$Z(g, \phi)_U=\pair{Z(g, \phi)_U P_U^\circ, \ P_U^\circ}_\NT,$$
where the right-hand side is the Neron--Tate height pairing.

In Theorem \ref{height series},  we give a precise formula for $Z(g, \phi)_U$ under some assumption of Schwartz functions, which particularly includes the following term:
\begin{equation}\label{2222}
-\frac {i_0(P, P)}{[O_E^\times: O_F^\times]}
\sum _{u\in \mu _U^2\bs F^\times}\sum _{y\in E^\times} r(g)\phi (y, u),
\end{equation}
where $i_0(P, P)$ is a modified arithmetic self-intersection number of the Zariski closure $\bar P$ on the integral model $\CX_U$. 
Notice that this terms was killed in \cite{YZZ} by some stronger assumption on Schwartz functions. 

Finally, Theorem \ref{quaternion main} essentially follows from an identity between  (\ref{1111}) and (\ref{2222}). To get this identity, the idea is to use the theory of pseudo-theta series in \S \ref{section pseudo}.
There is already a basic concept of pseudo-theta series in \cite{YZZ}, but here we develop a more general theory to cover the degenerate terms.

\subsubsection*{Pseudo-theta series}

From the explicit formulas in Theorem \ref{derivative series} and Theorem \ref{height series}, the difference $\mathcal D (g, \phi)$ is a finite sum of the so-called {\em pseudo-theta series}:
$$A_{\phi'}^{(S)}(g)=\sum _{u\in \mu^2\bs F^\times}
\sum _{x\in V_1\setminus V_0}\phi_S'(g, x, u)r_V (g)\phi^S (x, u), \qquad g\in \GL_2(\BA),$$
where 
\begin{itemize}
\item $S$ is a finite set of places of $F$ including all archimedean places,  
\item $\mu\subset O_F^\times$ is a subgroup of finite index,
 \item $V_0\subset V_1\subset V$ is a filtration of totally positive definite quadratic spaces of $F$, 
\item $\phi^S\in \CS (V(\BA^S)\times \BA^{S, \times})$ is a Schwartz function outside $S$, and 
\item $\phi'_S$ is
 a locally constant function on  
 $$\prod _{v\in S}  \left(\GL_2(F_v)\times (V_1-V_0)(F_v)\times F_v\right)$$
 with some extra smoothness or boundedness conditions.
 \end{itemize}
 
 Notice that a pseudo-theta series usually is not automorphic. But our key 
 Lemma \ref{pseudo} shows that if a sum of pseudo-theta series is 
 automorphic, then we can replace them by the difference $\theta _{A, 1}-\theta _{A, 0}$ of associated theta series:
 $$\theta_{A, 1}(g)=\sum _{u\in \mu^2\bs F^\times}\sum _{x\in V_1}r_{V_1}(g)
 \phi_S'(1, x, u)r_{V_1}(g)\phi^S(x, u),$$
 $$\theta_{A, 0}(g)=\sum _{u\in \mu^2\bs F^\times}\sum _{x\in V_0}r_{V_0}(g)
 \phi_S'(1, x, u)r_{V_0}(g)\phi^S(x, u).$$
Since the weights of these theta series depend only on the dimensions of $V_i$, there is a vanishing of some sums of theta series grouped in terms of $\dim V_i$.
 
Combining  Lemma \ref{pseudo} for  $\mathcal D(g, \phi)$ with some local computation gives the following identity for the self-intersection of CM points $P$
(Theorem \ref{result of comparison}):
$$\frac 1{[O_E^\times: O_F^\times]}i_0(P, P)
=\frac {L_f'(0, \eta)}{L_f (0, \eta)}+\frac 12 \log (d_{E/F}/d_\BB).$$
This is essentially the desired identity between (\ref{1111}) and (\ref{2222}).
Now Theorem \ref{quaternion main} follows the following arithmetic adjunction formula (Theorem \ref{adjunction all places}):
  $$\frac 1{[O_E^\times: O_F^\times]} i_0(P, P)
  =-h_{\ol \CL_U}(P),$$
which will be  proved  by explicit local computations.

\subsubsection*{Acknowlegment}

The authors are indebted to Wei Zhang in their joint work with him of proving a general Gross--Zagier formula on Shimura curves from 2007 to 2012. Inspired by this joint work, the analytic formula for the quaternionic height was obtained by XY in 2008. Motivated by the averaged Colmez conjecture and the work of Jacob Tsimerman in 2015, the relation between the quaternionic height and the averaged Faltings height was proved by SZ in 2015. The authors would also like to express their deep thanks to Pierre Colmez and Jacob Tsimerman. 

The authors are indebted to Benedict Gross and Don Zagier as the current paper would not exist without the original Gross--Zagier formula. The authors are grateful to Benedict Gross for his encouragement,  support, and his explanation on the very interesting history from the Chowla--Selberg formula to the Colmez conjecture. 

The authors would like to thank Pierre Colmez, Yichao Tian, Tonghai Yang, and Wei Zhang for helpful discussions. 
XY would like to thank the hospitality of AMSS of Chinese Academy of Sciences, and acknowledge the support of the National Science Foundation
under the award DMS-1330987. SZ would like to thank the  AMSS of Chinese Academy of Sciences, and the IAS of Tsinghua University for their hospitality during the preparation of this work, and to the National Science Foundation for its support via awards DMS-1415502 and  DMS-1404369.

\newpage
\part{Faltings heights}

The goal of this part is to prove Theorem \ref{U2O}.
Throughout this part, we fix a quadratic CM extension $E/F$. 

\section{Decomposition of Faltings heights}

In this section, we will first decompose $h(\Phi)$ into a sum of 
 components $h(\Phi, \tau)$ for each  $\tau\in \Phi$.
See Theorem \ref{PhitauPsi}.
This is done by using a hermitian pairing between $\Omega (A_\Phi)$ and $\Omega (A_\Phi^t)$. Then we define the height $h(\Phi_1, \Phi_2)$ 
for a nearby pair $(\Phi_1, \Phi_2)$ of CM types of $E$ (in the sense that $\Phi_1\cap \Phi_2$ has  $g-1$ elements)
as the  average of two heights $h(\Phi_i, \tau_i)$, where $\tau_i$ is the complements of $\Phi_1\cap \Phi_2$ in $\Phi_i$.
We will end this section by showing that $h(\Phi_1, \Phi_2)$ can be computed by any abelian variety isogenous to the product of two CM abelian varieties 
with CM types $\Phi_1$ and $\Phi_2$.

\subsection{Hermitian pairings} \label{section hermitian}
Let $A$ be a complex abelian variety with space $\Omega (A)$ of holomorphic $1$-forms.
Then we  define a metric on the complex line  $\omega(A)=\det \Omega (A)$  by
$$\|\alpha\|^2=\frac{1}{(2\pi)^g}\int _{A(\BC)}|\alpha\wedge \bar\alpha|.$$
In terms of Hodge theory,
this norm is given by the following pairing between $\det H^1(A, \BC)$ and $\det H_1(A, \BZ)$:
$$\|\alpha\|^2=\frac{1}{(2\pi)^g} |\pair{\alpha\wedge \bar\alpha, e_{A}}|,$$
where  $e_{A}$ is a basis of  $\det H_1(A, \BZ)=H_{2g}(A, \BZ)$.

Let  $A^t$ be the   dual abelian variety of $A$.
Then we have  a uniformization
$$A^t(\BC)=H^1(A, \CO_A)/H^1(A, 2\pi i\BZ).$$
This induces the following canonical isomorphisms
$$\Omega (A^t)^\vee=\Lie (A^t)\simeq H^1(A, \CO_A)\simeq H^{0, 1}(A)=\bar\Omega (A).$$ 
 Thus we have  a perfect hermitian pairing:
 $$\Omega (A)\times \Omega (A^t)\lra \BC.$$
 The hermitian pairing is functorial in the sense that if $\phi: B\lra A$ is a morphism of 
abelian varieties, then 
we have
$$(\phi^* \alpha, \beta)=(\alpha, (\phi ^t)^*\beta), \qquad \alpha \in \Omega (A), \quad \beta \in \Omega (B^t).$$
Here $\phi^t: A^t\lra B^t$ denotes the dual morphism. 

Taking determinants, this gives a  hermitian norm  $\|\cdot \|$ on $\omega (A)\otimes \omega(A^t)$.
Using this norm, we obtain the following product formula.

\begin{lem}
For any 
$\alpha \in \det \Omega (A)$ and $\beta \in \det \Omega(A^t)$, 
$$
\|\alpha\|^2\cdot \|\beta\|^2
=\|\alpha \otimes \beta\|^2.$$\end{lem}

\begin{proof} The direct sum of the pairing  $\Omega (A)\otimes \Omega (A^t)\lra \BC$ and its complex conjugate give a perfect hermitian pairing
$$H^1(A, \BC)\otimes H^1(A^t, \BC)\lra \BC.$$
This pairing is dual to the canonical perfect pairing 
$$H_1(A, \BZ)\otimes H_1(A^t, \BZ)\lra 2\pi i\BZ$$
by the above uniformization of $A^t$. Taking determinants and using the Hodge decomposition, 
we obtain isomorphism of lines:
$$\omega (A)\otimes \ol\omega (A)\otimes \omega (A^t)\otimes\ol\omega (A^t)\simeq \BC.$$
This isomorphism is dual to the isomorphism 
$$\det H_{2g}(A, \BZ)\otimes \det H_{2g}(A^t, \BZ)\lra (2\pi i)^{2g}\BZ.$$
Then we have 
\begin{align*}
\|\alpha\|^2\cdot \|\beta\|^2&=(2\pi)^{-2g}|\pair{\alpha\wedge \bar\alpha, e_{A}}|\cdot |\pair{\beta\wedge \bar\beta, e_{A^t}}|\\
=&(2\pi)^{-2g}|\pair{\alpha\otimes \beta\cdot \ol{\alpha\otimes\beta}, \quad e_{A}\otimes  e_{A^t}}|
=\|\alpha \otimes\beta\|^2.\end{align*}
Here in the last step, we use the pairing $(e_A, e_{A^t})=(2\pi i)^{2g}$. 
\end{proof}

Now we assume that $A$ has a multiplication by an order of a number field $E$.
 Then $E$ is either totally real or CM. Let $c$ be the CM involution on $E$ (which is trivial if $E$ is totally real).
Then for each embedding $\tau: E\lra \BC$, we have a projection $E\otimes\BC \lra \BC$,
and a  $\tau$-eigen quotient space
$$W(A, \tau):=\Omega (A)\otimes _{E\otimes \BC, \tau}\BC.$$ 

The action of $E$ on $A$ induces an action of $E$ on $A^t$. More precisely, for any $\gamma\in E$ corresponding to $\gamma:A\to A$, let $\gamma$ act on $A^t$ via $\gamma^t:A^t\to A^t$, where the latter is just the morphism compatible with the pull-back map $\gamma^*:\Pic^0(A)\to \Pic^0(A)$.
Now we define $W(A^t,\tau)$ analogously.
Then there are  decompositions
$$\Omega (A)=\bigoplus_{\tau: E\lra \BC} W(A, \tau), \qquad \Omega (A^t)=\bigoplus_{\tau: E\lra \BC} W(A^t, \tau).$$
The above hermitian pairing between $\Omega (A)$ and $\Omega (A^t)$ is an orthogonal sum of 
hermitian parings between  $W(A, \tau) $ and $W(A^t, \tau c)$ for each complex embedding $\tau$ of $E$.

\subsection{Decomposition of heights}
Now we assume that  $A$ is defined over a number field $K\subset \BC$ with a semi-abelian relative identity component of the  Neron model $\CA$ over $O_K$,
  that $\CA$ has actions  by the ring of integers $O_E$ of a field $E$,
 and that $K$ contains the normal closure of $E$ in $\bar \BQ$.
 Then for each embedding $\tau: E\lra K$, we can define the $\tau$-quotient $O_K$-module
$$\CW(\CA, \tau):=\Omega (\CA)\otimes _{O_K\otimes O_E, \tau}O_K.$$ 
The action of $E$ on $A$ induces an action on $A^t$ as above, so we define $\CW (\CA^t, \tau)$ analogously.
Define a line bundle over $\Spec\, O_K$ by
$$\CN(A, \tau):=\det \CW(\CA, \tau)\otimes \det \CW(\CA^t, \tau c).$$
At each archimedean place $v$ of $K$, there is a norm $\|\cdot\|_v$ on  $\CN (A, \tau)$ defined as above. Thus we have 
a metrized line bundle 
 $\ol\CN(A, \tau):=(\CN, \|\cdot\|)$.
We  define the $\tau$-part of the Faltings height:
$$h(A, \tau)=\frac1{2[K:\BQ]}\wh\deg(\ol \CN (A, \tau) ).$$ 

\begin{thm}\label{hAtau}
Assume that  $A$ has CM by $O_E$ with type $\Phi\subset \Hom (E, K)$. Then $h(A, \tau)$ depends only on the pair $(\Phi, \tau)$.
\end{thm}

\begin{proof}
Let $B$ be another abelian variety with CM  by $O_E$ of type $\Phi$.
After a base change, we can assume that $A$ and $B$ are defined over $K$ and have everywhere good reduction over $O_K$. 
We can also assume that there is a dual pair of $O_E$-isogenies over $K$:
$$f: A\lra B, \qquad f^t: B^t\lra A^t.$$
These isogenies extend to integral models over $O_K$:
 $$f: \CA\lra \CB, \qquad f^t: \CB^t\lra \CA^t.$$
They further induce nonzero morphisms of line bundles:
$$f^*: \CW (\CB, \tau) \lra \CW (\CA, \tau), \qquad f^{t *}: \CW(\CA^t, \tau c)\lra \CW(\CB^t, \tau c).$$
Thus we have a rational map of metrized line bundles:
$$\varphi: \CN (B, \tau)\lra  \CN(A, \tau).$$
Computing the norm of this map gives 
$$h(A, \tau)-h(B, \tau)=
-\frac 1{2[K:\BQ]}\sum _{p\le \infty} \sum_{\sigma: K\to \bar \BQ_p}\log \|\varphi _\sigma\|_p.$$
Theorem \ref{hAtau}  will follow from the identity
$$\prod_{\sigma: K\to \bar \BQ_p}\|\varphi _\sigma\|_p=1$$
for each place $p$ of $\BQ$.
Notice that this identity is compatible with base changes.  
 If $p=\infty$,  by the above functoriality of the hermitian pairing of invariant forms, it is easy to see that $\varphi_\sigma$ is an isometry. 
  
  It remains to study the product  when $p<\infty$. We will use the $p$-divisible groups $\CA[p^\infty]$ and $\CB[p^\infty]$ over $O_K$,
  and analogous space of differential forms. 
    For a place $\sigma$ of $K$ over a prime $p$, and an abelian variety $\CX$ from $\CA, \CA^t, \CB, \CB^t$, we have identities 
    $$\Omega (\CX)_\sigma =\Omega (\CX[p^\infty])_\sigma , \qquad \CW(\CX, \tau)_\sigma =\CW(\CX[p^\infty], \tau)_\sigma.$$
    Thus we may view $\varphi_\sigma$ as a morphism of line bundles induced from $p$-divisible groups:
 $$\varphi_\sigma : \CN (B[p^\infty], \tau)\lra  \CN(A[p^\infty], \tau).$$

Notice that $\Hom_{O_{E, p}} (\CA[p^\infty], \CB[p^\infty])$ is a free  module of rank $1$ over $O_{E, p}$.
   Thus we have an isomorphism 
  of $p$-divisible $\BZ_p\otimes O_E$-modules  over $O_K$: 
  $$\iota: \CA[p^\infty]\lra \CB[p^\infty].$$
  We can use this morphism to identify $\CB[p^\infty]$ with $\CA[p^\infty]$, and $\CB^t[p^\infty]$ with $\CA^t[p^\infty]$. 
  In this way, $f$ is an $O_{E, p}$-endomorphism of $\CA[p^\infty]$. Since the Tate module of this group at the generic fiber is a free $O_{E, p}$-module
  of rank $1$, $f$ is given by multiplication by an element $\alpha \in O_{E, p}$ on $\CA[p^\infty]$. Taking the dual,  $f^t$ is given by $\bar\alpha \in O _{E, p}$ on
  $\CA^t[p^\infty]$.
 Thus $\varphi_\sigma $ is given by the multiplication by $(\alpha/\bar\alpha)_\sigma $ on the group $\CN(A[p^\infty], \tau)$. 
It follows that 
  $$\prod _{\sigma: K\to \bar\BQ_p}\|\varphi_\sigma\|_p=\prod _{\sigma : K\to \bar \BQ_p}\frac {|\alpha _{\sigma\tau}|}{|\alpha _{\sigma \tau c}|}=
  \prod _{\sigma : K\to \bar \BQ_p}\frac {|\alpha _{\sigma\tau}|}{|\alpha _{c_p\sigma \tau }|}=1.$$
  Here $c_p$ is an element $\Gal (\bar \BQ_p/\BQ_p)$ which induces the complex conjugation on $E$ via every embedding $E\lra \bar \BQ_p$.
  \end{proof}

By Theorem \ref{hAtau}, we can denote $h(A, \tau)$  by $h(\Phi, \tau)$ if $A$ has  CM type $(O_E,\Phi)$.
In the following, we want to compute the difference:
$$h(\Phi)-\sum _{\tau\in \Phi} h (\Phi, \tau).$$
Let $E_\Phi$ be the reflex field of $\Phi$ generated by all $\Phi$-traces and $t: E\lra E_\Phi$ be the induced trace map.
Then the action $E$ on the $E_\Phi$-vector space  $E_\Phi\otimes _\BQ E$ gives a decomposition into a direct sum of  $E\otimes E_\Phi$-subspaces:
$$E_\Phi\otimes _\BQ E=\wt E_\Phi\oplus \wt E_{\Phi^c}$$
so that the traces of the actions of  $E$ are  $t$ and $t^c$ respectively.
 In particular  $\wt E_\Phi$ and $\wt E_{\Phi^c}$ are two quotient algebras of $E_\Phi \otimes _\BQ E$.
Let $R_\Phi$ denote the image  of $O_{E_\Phi}\otimes O_E$ in $\wt E_\Phi$.
Denote by $\gd_\Phi$  the relative discriminant of the extension $R_\Phi/O_{E_\Phi}$,
and by $d_\Phi$ the norm of $\gd_\Phi$.

\begin{thm}\label{PhitauPsi}
$$
h(\Phi)-\sum _{\tau\in \Phi} h (\Phi, \tau)=-\frac 1{4[E_\Phi:\BQ]}\log (d_\Phi d_{\Phi ^c}).$$
\end{thm}

\begin{proof} 
By definition,  we have morphisms
$$\phi:\quad \Omega (\CA)\lra \bigoplus _{\tau \in \Phi}\CW(\CA, \tau), \qquad 
\phi^t :\quad \Omega (\CA^t )\lra \bigoplus _{\tau \in \Phi}\CW(\CA^t , \tau  c)$$
Thus we have elements 
$$\det \phi\in \left(\bigotimes _{\tau\in \Phi}\CW (\CA, \tau)\right)\otimes\det  \Omega (\CA)^{-1},\qquad 
\det \phi^t \in \left(\bigotimes _{\tau\in \Phi}\CW (\CA^t, \tau c) \right)\otimes\det  \Omega (\CA^t )^{-1}.$$
This gives a section of the line bundle:
$$\ell \in \left(\bigotimes_{\tau \in \Phi} N (A, \tau )\right)\otimes (\omega(A)\otimes \omega(A^t))^{-1}.$$
With metrics defined on these line bundles, we have an adelic metric on $\ell$.
Now we have an identity:
$$h(\Phi)-\sum _{\tau\in \Phi} h (\Phi, \tau)=\frac 1{2[K:\BQ]}\sum _{p\le \infty}\sum _{\sigma: K\to \bar\BQ_p}
\log \|\ell _\sigma\|_p,$$
where 
$$\|\ell_\sigma\|_p =\|\det \phi_\sigma\|_p\cdot  \| \det \phi^t_\sigma\|_p.$$
By the above discussion, it is clear that $\ell$ has norm $1$ at all archimedean places. So we need only consider $p<\infty$.

As a $\BZ_p$-algebra, $O_{E, p}$ is generated by one element $x\in O_{E,p}$, which has a minimal equation
$$P(t)=\prod _{\sigma \in \Hom(E,K)}(t-x^\sigma)\in \BZ_p[t], \qquad x^\sigma \in K_p^\times. $$
Write 
$$P_\Phi (t)=\prod _{\tau \in \Phi}(t-x^\tau)\in E_{\Phi, p}[t], \quad P_{\Phi^c} (t)=\prod _{\tau \in \Phi^c}(t-x^\tau)\in E_{\Phi, p}[t].$$
It is clear that $R_{\Phi, p}=O _{E_\Phi, p}[t]/P_\Phi(t)$. Thus the ideal $\gd_{\Phi, p}$ of $O_{E_\Phi, p}$ is generated by 
 $\Delta (\Phi)_{p}=\prod_{i<j}(x^{\tau _i}-x^{\tau_j})^2$.  
 
To study $\ell _\sigma$, let us write $K_\sigma $ for the completion of $\sigma (K)$,  $O_\sigma$ for the ring of $p$-adic integers in $K_\sigma$,
and $\CA_\sigma$ for the model of $A$ over $O_\sigma$. Consider the Hodge--de Rham filtration
\begin{equation}\label{HdR1}
0\lra \Omega  (\CA_\sigma)\lra H^1_{\dR}(\CA_\sigma)\lra H^1(\CA_\sigma, O_{\CA_\sigma})\lra 0.
\end{equation}
With respect to the action of $O_E$, one has that $H^1_{dR}(\CA_\sigma)$ is free of rank $1$ over  $O_\sigma\otimes O_E$.
See \cite[Lem. II. 1.2]{Co}. The other two terms are free $O_\sigma$-modules under which $O_E$ acts with type $\Phi$ and $\Phi^c$ respectively.

\begin{lem}The above exact sequence of $O_\sigma \otimes O_E$-modules is isomorphic to the following sequence:
\begin{equation}\label{HdR2}
0\lra \frac {O_\sigma [t]}{P_\Phi(t) }\overset{P_{\Phi^c}(t)}\lra 
\frac {O_\sigma [t]}{P(t)}\lra \frac {O_\sigma [t]}{P_{\Phi^c}(t)}\lra 0.
\end{equation}
\end{lem}
\begin{proof}
First we want to show that  \ref{HdR2} is an exact sequence. It  is clear that the sequence is exact at the  first and the third term, and that it is exact at 
the middle term  after base change to $K_\sigma$. Thus the exactness at the middle term is equivalent to the following statement: 
{\em an element $\alpha \in O_\sigma[t]$  divisible by $P_{\Phi^c}(t)$ in $K_\sigma[t]$  is divisible by $P_{\Phi^c}(t)$ in $O_\sigma [t]$.}
This follows from the classical Gauss's lemma.  

It remains to construct an isomorphism from \ref{HdR1} to \ref{HdR2}. 
By the above discussion, we can fix an isomorphism of $O_\sigma\otimes O_E$-module
$$\varphi:\quad  H^1_\dR (\CA_\sigma)\lra \frac {O_\sigma[t]}{P(t)}.$$
We want to extend this isomorphism to an isomorphism from exact sequence \ref{HdR1} to \ref{HdR2}.
It is clear that under  the actions by $O_E$,
 all terms in the exact sequence \ref{HdR1}   are torsion-free with the same CM types as corresponding terms in \ref{HdR2}.
 It follows that  $\varphi$  induces an isomorphism from \ref{HdR1} to \ref{HdR2}.
 \end{proof}

\begin{cor} There is an isomorphism of $(O _\sigma\otimes O_E)$-modules 
$$\Omega  (\CA)_\sigma \simeq O_\sigma [t]/P_\Phi (t)$$
under which $x$ acts as $t$.
\end{cor}

By this corollary, the evaluation $t\mapsto x^\tau$ gives an isomorphism  $\Omega ^\tau\simeq O_\sigma$. Thus we have the following model of $\phi_\sigma$:
$$\phi_\sigma:  O_\sigma [t]/\Phi (t)\lra \bigoplus _{\tau\in \Phi}O_\sigma, \qquad t\longmapsto (x^\tau: \tau \in \Phi).$$
Notice that $O_\sigma [t]/\Phi$ has the the basis
$(1, t, \cdots, t^{g-1})$, and $\bigoplus _{\Phi}O_\sigma$ has a usual basis 
$e_1,\cdots, e_g$ by choosing an ordering $(\tau_1, \cdots, \tau _g)$. 
We have
$$(\det \phi _\sigma) (1\wedge t\wedge t^2\wedge \cdots\wedge t^{g-1})
=\pm \det((t^{\tau _j})^i)\cdot e_1\wedge\cdots \wedge e_g
=\sqrt {\Delta (\Phi)_p}\cdot e_1\wedge\cdots \wedge e_g.$$
Thus finally, we have shown
$$\|\det\phi_\sigma \|_p =|\Delta (\Phi)_p |^{1/2}.$$

Put everything together to obtain
\begin{align*}
h(\Phi)-\sum _{\tau\in \Phi} h (\Phi, \tau)=&\frac 1{4[K:\BQ]}\sum _{p< \infty}\sum _{\sigma: K\to \bar\BQ_p}
\log |\Delta (\Phi)_p \cdot \Delta (\Phi^c)_p |\\
=&-\frac 1{4[E_\Phi:\BQ]}\log (d_\Phi\cdot d_{\Phi^c}).
\end{align*}

\end{proof}

By a {\em nearby pair of CM types} of $E$, we mean a pair $(\Phi_1, \Phi_2)$ of CM types of $E$ such that $\Phi_1\cap \Phi_2$ has order $g-1$.
Let $\tau_i$ be the complement of $\Phi_1\cap \Phi_2$ in $\Phi_i$ for $i=1,2$.
Define 
$$h(\Phi_1, \Phi_2):=\frac 12 \left(h ( \Phi_1, \tau _1)+h(\Phi_2, \tau_2)\right).$$

\begin{cor}\label{Phi2Psi}
$$\frac{1}{2^g}\sum _\Phi h(\Phi)=\frac{1}{2^{g-1}}\sum _{(\Phi_1, \Phi_2) }h (\Phi_1, \Phi_2)
-\frac 14 \log d_F,$$
where the second sum is over non-ordered pairs of nearby CM types of $E$.
\end{cor}

\begin{proof}
Take the average over all types $\Phi$ in Theorem \ref{PhitauPsi} to obtain
$$\frac{1}{2^g}\sum _\Phi h(\Phi)- \frac{1}{2^g}\sum _{\Phi, \tau} h (\Phi, \tau)
=\frac 1{4[K:\BQ]}\sum _{p<\infty}\sum _{\sigma: K\to \bar\BQ_p}\frac 1{2^g}\sum _\Phi
\log |\Delta (\Phi)_p \cdot \Delta (\Phi^c)_p|$$
where the second sum is over pairs of CM type $\Phi\subset \Hom (E, \bar \BQ)$ and $\tau\in \Phi$.

For  a fixed $\sigma: K\lra \BQ_p$, the last sum on the right-hand side is a sum of $\log |x_1-x_2|_p^2$ over pairs  $x_1, x_2$ of roots of $\Phi$ with $x_2\neq x_1$ and $x_2\neq x_1^c$. 
Let $x_1, x_2, \cdots, x_{2g}$ be all roots of $P(t)$ such that $x_i^c=x_{i+g}$. Then the last sum on the right-hand side is a multiple of 
$$\log \left|\frac {\prod _{i<j}(x_i-x_j)^2}{\prod _{i\leq g}(x_i-x_{i+g})^2}\right|=\log \left|\frac {d_E}{d_{E/F}} \right|=\log |d_F|^2.$$
Since there are $2^{g-1}$ such terms, we have
$$
\frac 1{[K:\BQ]} \sum _{\sigma: K\to \bar\BQ_p}
\frac 1{2^g}\sum _\Phi
\log |\Delta (\Phi)_p \cdot \Delta (\Phi^c)_p |=\log |d_F|_p.$$
Thus we have 
$$\frac{1}{2^g}\sum _\Phi h(\Phi)- \frac{1}{2^g}\sum _{\Phi, \tau} h (\Phi, \tau)
=-\frac 14\log |d_F|.$$
Then it is easy to obtain the result. 
\end{proof}

\subsection{Some special abelian varieties}
In this subsection, we fix a nearby pair $(\Phi_1, \Phi_2)$ of  CM types  of $E$.
We want to compute the height $h(\Phi_1, \Phi_2)$ by a single abelian variety.

\begin{thm}\label{specialA}
Let $A, A_1, A_2$ be  abelian varieties over a number field $K$ with endomorphisms  by $O_E$ such that the following conditions hold:
\begin{enumerate}[(1)]
\item $A_1, A_2$ are CM-abelian varieties of type $\Phi_1$ and $\Phi_2$ respectively;
\item $A$ is $O_E$-isogenous to $A_1\times A_2$.
\end{enumerate}
Then 
$$ h(\Phi_1, \Phi_2)=\frac 12 \left(h(A_1, \tau_1)+h(A_2, \tau_2)\right)=\frac 12 h(A, \tau),$$
where $\tau_i$ is the complement of $\Phi_1\cap \Phi_2$ in $\Phi_i$,
and $\tau$ is the place $F$ under $\tau_i$. Here in the last equality, $A$ is considered to have a multiplication by $O_F$.
\end{thm}

\begin{proof}
From an $O_E$-isogeny $A_1\times A_2\lra A$, we obtain an $O_E$-morphism $i: A_1\lra A$ with a finite kernel.
By Theorem \ref{hAtau}, we may 
replace $A_1$ by the image of $i$ to  assume that $i$ is an embedding. Now we have an isogeny $A_2\lra A/A_1$.
Similarly, we may assume that $A_2=A/A_1$. Thus we have a dual pair of  exact sequences of $O_E$-abelian varieties:
$$0\lra A_1\lra A\lra A_2\lra 0, \qquad 0\lra A_2^t\lra A^t \lra A_1^t\lra 0.$$

After a base change, we may assume that $A_1$ and $A_2$ have good reductions over $O_K$. This implies that $A$ also has good reduction over $O_K$.
Thus we have a dual pair of  exact sequences of their Neron models:
$$0\lra \CA_1\lra \CA\lra \CA_2\lra 0, \qquad 0\lra \CA_2^t\lra \CA^t \lra \CA_1^t\lra 0.$$
These  exact sequences induce a dual pair of  exact sequences of their invariant differentials:
$$0\lra \Omega (\CA_2)\lra \Omega (\CA)\lra \Omega (\CA_1)\lra 0, \qquad 0\lra \Omega (\CA_1^t )\lra \Omega (\CA^t )\lra \Omega (\CA_2^t)\lra 0.$$
Then we have exact sequences:
$$0\lra \CW (\CA_2, \tau_2)\lra \CW (\CA, \tau)\lra \CW(\CA_1, \tau_1)\lra 0, $$
$$0\lra \CW(\CA_1^t, \tau_2)\lra \CW (\CA^t, \tau)\lra \CW (\CA_2^t, \tau_1)\lra 0.$$
Taking determinants, we obtain
$$\det \CW (A, \tau) =\CW(A_1, \tau_1)\otimes \CW(A_2, \tau _2),
\qquad \det \CW(A^t, \tau )=\CW (A_1^t,  \tau _2)\otimes \CW (A_1^t, \tau _1).$$
It follows that we have a canonical isomorphism
$$\CN(A, \tau)\simeq \CN (A_1, \tau_1)\otimes\CN (A_2, \tau_2).$$
It is easy to show that this isomorphism is compatible with the metric defined by Hodge theory at infinite places.
Thus we have 
$$h(A, \tau)=h(A_1, \tau_1)+h(A_2, \tau_2).$$
\end{proof}

\section{Shimura curve $X'$}

In this section, we study a Shimura curve of PEL type  
following Deligne \cite{De}, Carayol \cite{Ca}, and \v Cerednik--Drinfeld \cite{BC, Ce}. 
After reviewing the basic facts about the moduli problems, we will study in special cases of 
 the integral models over the ring of integers of the reflex field,  
and the Kodaira--Spencer map over complex numbers.

\subsection{Moduli interpretations}
Recall that we have a totally real number field $F$, a quadratic CM extension $E/F$, and 
a totally definite incoherent quaternion algebra $\BB$ over $\BA=\BA_F$. 
We will consider one of the following special cases later:
\begin{enumerate}[(1)]
\item  $E=F(\sqrt \lambda)$ with a $\lambda \in \BQ$ as in Carayol \cite{Ca};
\item  $\BA_E$ is embedded into $\BB$ over $\BA$
as in the introduction.
\end{enumerate}

Let $(\Phi_1, \Phi_2)$ be a nearby pair of CM types of $E$. Let $\tau$ be the place of $F$ missing in $\Phi_1\cap \Phi_2$, and $B$ the quaternion 
algebra over $F$ with ramification set $\Sigma (\BB)\setminus \{\tau\}$.
We form a reductive group $G'':=B^\times \times _{F^\times} E^\times$, the quotient of $B^\times \times  E^\times$ by $F^\times$ via the action $a\circ(b,e)=(ba^{-1},ae)$.
Let $B^1$ and $E^1$ denote respectively the subgroups of $B$ and $E$ with norm $1$. 
Then $G''$  has the same derived subgroup $G_1: =B^1$ as $G=B^\times$ with quotient  isomorphic to $F^\times \times E^1$ via the 
following map:
$$\nu=(\nu_1, \nu_2):G''/G_1\lra F^\times \times E^1, \quad \ (b, e)\longmapsto (q(b) e\bar e, e/\bar e).$$
Here $q(b)$ denotes the reduced norm of $b$. 

Define an algebraic group $G'$ over $\BQ$ as a subgroup of $G''$ by 
$$G'(\BQ)=\left\{g\in G''(\BQ) : \nu_1(g)\in \BQ^\times\right\}.$$
Let $h': \BC^\times \lra G'(\BR)$ be the complex structure which has a lifting to a morphism $h\times h_E$ to $(B\otimes \BR)^\times \times (E\otimes \BR)$ as follows:
the component to $(B\otimes \BR)^\times =G(\BR)$ is the same as $h$ for defining quaternion Shimura curve  as in Carayol \cite{Ca}, see also \S \ref{section 4.1}; the component to 
$(E\otimes \BR)^\times\overset {\Phi_1}\iso  (\BC^\times)^g$
is given by 
$$h_E: z\longmapsto (1, z^{-1}, \cdots, z^{-1})$$
where the first component corresponds to the place over $\tau$.
The class of $G'(\BR)$-conjugacy class of $h'$ is identified with $\gh^\pm=\BC\setminus \BR$ by  
$$ghg^{-1}\longmapsto g(i), \qquad g\in G'(\BR).$$

Thus we have Shimura curves over $\BC$ indexed by open and compact subgroups $U'$ of $G'(\wh \BQ)$:
$$X'_{U'}(\BC)=G'(\BQ)\bs \gh^\pm\times G'(\wh \BQ)/U'.$$
It is not difficult to show that the reflex field of $h'$ is the same as the reflex field of $\Phi_1+\Phi_2$.
Let $F'$ be the reflex field of $h'$. Then $X'_{U'}$ is defined over $F'$. The following is a relation between $F$ and $F'$:

\begin{pro} Let $\Psi$ denote $\Phi_1\cap \Phi_2$, and let  $\tau: F\lra \BC$ be the place of $F$ missing in $\Psi|_F$.
Then $F'$ contains $\tau (F)$.
\end{pro}
\begin{proof} 
By definition, $\Gal (\BC/F')$ consists of  elements $\sigma \in \Aut (\BC)$   fixing the weighted set   $\Phi_1+\Phi_2$.
It is clear that 
$$\Phi_1+\Phi_2=2\Psi +\tau _1+\tau_2$$
with $\tau_i$ the complement of $\Psi$ in $\Phi_i$. 
Considering multiplicity, such a $\sigma$ fixes $\tau_1+\tau_2$.
In other words, it fixes $\tau (F)$.
\end{proof}

Let $X'$ be the projective limit of $X'_{U'}$ for all $X'_{U'}$.
 Then $X'$ is a scheme over $F'$ with a  right action by $G'(\wh \BQ)$ and a uniformization given by
$${X'}_{\tau'}(\BC)=G'(\BQ) \bs \gh^\pm \times G' (\wh \BQ).$$
See Carayol \cite[\S3.1]{Ca}.

Denote by $G''(\BQ)_+$ the subgroup of elements $(b,e)$ in $G''(\BQ)=B^\times \times _{F^\times} E^\times$ such that $q(b)\in F$ is totally positive. 
As in Carayol \cite[\S3.4]{Ca}, the curve $X'$ is equipped with a right action of the subgroup  $\wt G=G''(\BQ)_+\cdot G'(\wh \BQ)$ of $G''(\wh \BQ)$ as follows: for any elements $(g _0, g_1)\in G''(\BQ)_+\times G'(\wh \BQ)$, define 
$$[z, h]\cdot (g_0g_1)=[g_0^{-1} z, g _0^{-1}h g_0g_1].$$
The subgroup of elements fixing every point on $X'$ is given by the center  $Z''(\BQ)\simeq E^\times$ of $G''(\BQ)_+$.

In the following, we want to describe the moduli problem associated to $X'_{U'}$ following Carayol \cite[\S2]{Ca}.
For this, we will work on the quaternion algebra $B'=B\otimes _F E$ over $E$.
Let $V':=B'$ as a left $B'$-vector space. Fix an invertible element $\gamma'\in B'$ such that $\bar \gamma'=-\gamma'$ where $b\mapsto \bar b$
is the involution on $B'=B\otimes _F E$ induced from the canonical involution on $B$ and the complex conjugation on $E$. Then we define 
a symplectic form on $V'$ by
\begin{equation}
\label{psi'}
\psi'(v, w)=\tr_{E/\BQ}\tr _{B'/E}(\gamma' v\bar w).
\end{equation}
Here $\tr _{B'/E}$ is the reduced trace on $B'$. 
This form induces an involution $*$ on $B'$ by:
\begin{equation}\label{ell^*}\psi' (\ell v, w)=\psi' (v, \ell^* w), \qquad \ell ^*=\gamma ^{'-1}\bar \ell \gamma'.\end{equation}
The group $G'$ can be identified with the group of $B'$-linear symplectic similitudes of $(V', \psi')$. More precisely, $G'$ is a subgroup of $G''$ 
which can be identified with 
the subgroup $B^\times\cdot E^\times$ of $B^{'\times }$ which acts on $V'=B'$ by right multiplications. 

The composition of $h'$ and the action of $G'(\BR)$ on $V'_\BR$ induce a Hodge structure on $V'$ of weights  $(-1, 0)$ and $(0, -1)$. 
One can choose a $\gamma$ such that $\psi'$ induces a polarization of the Hodge structure $(V', h')$:
$$\psi'(x, xh'(i)^{-1})\ge 0, \qquad \forall x\in V'_\BR.$$

 By Deligne \cite[\S6]{De}, $X'_{U'}$ represents the following  functor 
 $\CF_{U'}$ on the category of $F' $-schemes when $U'$ is sufficiently small.
For any $F'$-scheme $S$, $\CF_{U'}(S)$ is the set of isomorphism classes of quadruples  $[A, \iota, \theta, \kappa]$
where
\begin{enumerate}[(1)]
\item $A$ is an abelian scheme over $S$ up to isogeny;
\item $\iota: B'\lra \End ^0(A/S)$ is a homomorphism such that the induced action of $E$ on the $\CO_S$-module $\Lie (A/S)$ has the trace given by 
$$\tr (\ell, \Lie (A/S))=t(\tr _{B'/E}(\ell)), \qquad \forall \ell\in B', $$
where $t: E\lra F'$ is the trace map of $\Phi_1+\Phi_2$.
\item $\theta: A\lra A^t$ is a  polarization whose  Rosati involution on $\End ^0(A/S)$ induces the involution $*$ of $B'$ over $F$;
\item $\kappa: \wh V' \times S\lra H_1(A, \wh \BQ)$ is a $U'$-orbit of similitudes of  $B'$-skew hermitian modules.
\end{enumerate}

The group $\wt G$ acts on the inverse system of $\CF_{U'}$ as follows:
$$[A, \iota, \theta, \kappa]\cdot g =[A, \iota, \nu_1 (g)\theta, \kappa \cdot g].$$

\subsection{Curves $X'$ in case 1}

Let $p$ be a prime number, and $\wp'$ be a prime ideal of $O_{E}$ dividing $p$. We want to study the integral model of $X'_{U'}$ 
over the ring  $O_{(\wp')}:=O_{E}[x^{-1}: x\in O_{E}\setminus \wp']$ in the case considered in 
 Carayol \cite[\S2, \S5]{Ca}, i.e.,  $E=F(\sqrt{\lambda})$ with $\lambda$ a negative integer
 such that $p$ is split in $\BQ(\sqrt \lambda)$. Fix a square root $\mu$ of $\lambda$ in $\BC$ which gives 
a CM type of $E$ by 
$$\Phi_1: \quad E=F\otimes _\BQ \BQ(\sqrt\lambda)\lra F\otimes _\BQ \BC\simeq \BC^g,
\qquad \sqrt \lambda\mapsto (\mu, \cdots, \mu).$$
Let $\Phi_2$ be a nearby  CM type of $E$ which differs from $\Phi_1$ at the place over $\tau$ of $F$. Then the reflex field of $\Phi_1+\Phi_2$ is $E$.

 Using the isomorphism 
$$E_p=F_p\oplus F_p, \qquad \lambda \longmapsto (\mu, -\mu),$$
we have an identification 
$B'_p=B_p\times B_p$ so that the involution $*$ on $B'$ defined in \ref{ell^*}  induces an involution on $B_p$,  still denoted by $*$, so that 
 $(a, b)^*=(b^*, a^*)$. In this way we may assume that $O_{B', p}=O_{B, p}^*\oplus O_{B, p}$. 
The form $\psi'$ induces a perfect $(B_p, *)$-hermitian  pairing  $\psi_p: B_p\times  B_p\lra \BQ_p$ as follows 
$$\psi'_p((a,b), (c, d))=\psi_p (a, d)-\psi_p (c, b).$$
The subgroup $G'(\BQ_p)$ of $B_p^{'\times}$ consists of elements $(\lambda b, b)$ with $\lambda \in \BQ_p^\times$ 
and $b\in B_p^\times$. We identify $G'(\BQ_p)\simeq \BQ_p^\times\times B_p^\times$ by this description.

Let $O_{B', p}$ be an order of $B'_p$ stable under involution $\ell\mapsto \ell^*$, and let $\Lambda'_p$ be an $O_{B', p}$- 
lattice of $V'_p$ such that  $\psi'|_{\Lambda'_p}$ takes integral value and is perfect. Such an oder $O_{B', p}$ and a lattice $\Lambda'_p$ 
can be constructed from a maximal order $O_{B, p}$ of $B_p$ by the following formulae:
$$O_{B', p}:=O_{B, p}^*\oplus O_{B, p}, \qquad \Lambda_p': =O_{B, p}^\vee\oplus O_{B, p}$$ where 
$$O_{B, p}^\vee:=\left\{x\in B_p: \quad \psi _p(x, y)\in \BZ_p, \quad \forall y\in O_{B, p}\right\}.$$
The elements  of $G(\BQ_p)$ fix $\Lambda_p'$ form  a maximal compact subgroup $U_p'(1):=\BZ_p^\times \times O_{B, p}^\times$.

Let $\wp$ be the prime of $O_F$ under $\wp'$. Write   $O_{F, p}=O_\wp+O^\wp$ as a direct sum of $\BZ_p$-algebras, then we have a decomposition: 
$$O_{E,p}=O_{F, p}\oplus O_{F, p}=O_\wp\oplus O^\wp\oplus O_\wp\oplus O^\wp.$$
For any $O_{E, p}$-module $M$, there is a corresponding decomposition 
$$M=M_{1\wp}+M_1^\wp+ M_{2\wp}+M_2^\wp.$$

Let $\BZ_{(p)}=\BZ_p\cap \BQ$ be the localization of $\BZ$ at $p$. 
Let $O_{B', (p)}=O_{B', p}\cap B'$ be the $\BZ_{(p)}$-lattice in $B'$. 

For an open compact subgroup $U'^p$ of $G'(\wh \BQ^p)$,  define a moduli problem $\CF_{1, U'^p}$ over $O_{\wp}$  as follows:
for any $O_{\wp}$-scheme $S$, $\CF_{1, U'^p}(S)$ is the set of isomorphism classes  of quadruple  $[A, \iota, \theta, \kappa]$
where
\begin{enumerate}[(1)]
\item $A$ is an abelian scheme over $S$ up to prime-to-$p$ isogeny;
\item $\iota: O_{B', (p)}\lra \End (A/S)\otimes \BZ_{(p)}$ is a homomorphism such that the induced action of $O_{B'}$ on the $\CO_S$-module $\Lie (A/S)$ has the following properties:
\begin{itemize}
\item $\Lie (A)_{2\wp}$ is a special $O_{B, \wp}$-module in the sense that it is locally free of rank $1$ over $O_K\otimes O_S$ for any unramified  quadratic extension $K$ 
of $O_\wp$ embedded into $O_{B, \wp}$;
\item $\Lie ( A)_2^\wp=0$.
\end{itemize}
\item $\theta: A\lra A^t$ is a polarization whose  Rosati involution on $\End (A/S)\otimes \BZ_{(p)}$ induces the involution $*$ of $O_{B', p}$;
\item $\kappa: \wh V^p \times S\lra H_1(A, \wh \BQ^p)$ a $U'^p$-orbit of  similitudes  of $\wh O_{B'}^p$-skew hermitian modules.
\end{enumerate}

\begin{pro} \label{regularity}
When $U'^p$ is sufficiently small,  the scheme $\CF_{1, U'^p}$ is represented by a regular scheme $\CX'_{1, U'^p}$ over $O_{(\wp')}$ with the following properties:
\begin{enumerate}[(1)]
\item for the embedding $\tau': O_{(\wp')}\lra \BC$, the curve $\CX_{1, U'^p}(\BC)=X_{U_p'(1)\cdot U'^p}(\BC)$, where $U_p'(1)$ is the maximal open compact subgroup of $B_p'^\times$ 
fixing $\Lambda'_p$;
\item if $\wp$ is split in $B$, then $\CX_{1, U'^p}'$ is smooth over $O_\wp$;
\item if $\wp$ is ramified in $B$, then $\CX_{1, U'^p}'$ is a semistable relative Mumford curve in the sense that every irreducible component in the special fiber is isomorphic to $\BP^1$. 
\end{enumerate}
\end{pro}

\begin{proof} 
Let $O_{B'}$ be an $O_{E}$-order of $B'$. Replacing $O_{B'}$ by $O_{B'}\cap O_{B'}^*$, we may assume that $O_{B'}$ is stable 
under $*$. Let $\Lambda'$ be an $O_{B'}$-lattice of $B'$ with localization $\Lambda'_p$. With $\Lambda $ replaced by $m\Lambda$ with an $m$ prime to $p$,
we may assume that $\psi'$ takes integral value on $\Lambda'$. 
Assume now $U'^p$ fixes $\wh \Lambda'^p$ and fixes every point in $\Lambda'^p/n\Lambda'^p$ for some $n\ge 3$ prime to $p$.
It is easy to see that above functor is isomorphic to the following  functor $\wt \CF_{U'^p}$ over $O_{\wp}$-schmes:
for any $O_{\wp}$-scheme $S$, $\wt \CF_{U'^p} (S)$ is the set of isomorphism classes  of quadruple  $[A, \iota, \theta, \kappa]$
where
\begin{enumerate}[(1)]
\item $A$ is an abelian scheme over $S$;
\item $\iota: O_{B'}\lra \End (A/S)$ is a homomorphism such that the induced action of $O_{B'}$ on the $\CO_S$-module $\Lie (A/S)$ has the following properties:
\begin{itemize}
\item $\Lie (A)_{2\wp}$ is a special $O_{B, \wp}$-module in the sense that it is locally free of rank $1$ over $O_K\otimes O_S$ for any unramified  quadratic extension $K$ 
of $O_\wp$ embedded into $O_{B, \wp}$;
\item $\Lie ( A)_2^\wp=0$.
\end{itemize}
\item $\theta: A\lra A^t$ is a polarization whose  Rosati involution on $\End (A/S)$ induces the involution $*$ of $O_{B'}$;
\item $\kappa: \wh \Lambda^p \times S\lra H_1(A, \wh \BZ^p)$ a $U'^p$-orbit of  similitudes of $\wh O_{B'}$-skew hermitian modules.
\end{enumerate}

The condition (4) implies that the relative dimension of $A/S$ is $2g$.
 Also   the degree of the polarization $\theta$ in (3)  is $d=[\Lambda'^\vee, \Lambda']$
 where $\Lambda'^\vee$ is  the dual lattice of $\Lambda'$.
By Mumford theory, there is a fine  moduli space $\CM_{2g, d, n}$ over $\BZ_{(p)}$ classifying the the triples of $(A, \theta, \kappa_n)$ of an abelian variety $A$ of dimension $2g$,
and a polarization $\theta$ of degree $d$, and a full level $n$ structure $\kappa _n$. Thus we have a morphism of functor $\CF'_{U'^p}\lra \CM_{2g, d, n}$.
Now we can use the theory of Hilbert schemes to prove the existence of a scheme $\CX'_{0, U'^p}\lra \CM_{2g, d, n}$ to classify other additional structures on the triple $(A, \theta, \kappa_n)$
required in the functor $\wt \CF_{0, U'^p}$.

The second statement is proved in Carayol \cite[\S5.4]{Ca} in the case $\wp$ is split in $B$, and proved by \v Cerednik--Drinfeld (cf. \cite{BC, Ce}) in case $\wp$ is not 
split in $B$. \end{proof}

\begin{remark}
Our moduli problem here is slightly different from the moduli problem $\frak M_{0, H'}^2$ in Carayol \cite[\S5.2.2]{Ca} in three points: 
\begin{enumerate}[(1)]
\item we do not require that $p$ is prime to the discriminant $\gd_B\subset O_F$ of $B$;
\item  we allow $A$ to have prime-to-$p$ isogeny which is more flexible than \cite{Ca};
\item  we do not input a level structure  $k_p^\wp$  as in \cite{Ca}.
\end{enumerate}
\end{remark}

 \subsubsection*{$p$-divisible groups}
 Let $U'=U'_p(1)\cdot U'^p$ with $U'^p$ sufficiently small so that the functor $\CF _{U'}$ is representable by a universal family of abelian varieties:
 $$\CA_{U'}\lra \CX_{U'}.$$ 
  There is a  Barsotti--Tate  $O_{B', p}$-module $\CA_{U'}[p^\infty]$ on $\CX'_{U'}$ for any sufficiently small compact open subgroup $U'^p$ of $G'(\wh \BQ)^p$. With our assumption, this group has a decomposition 
$$\CA_{U'}[p^\infty]=\CA_{U'}[p^\infty]_1+\CA_{U'}[p^\infty]_2=\CA_{U'}[p^\infty]_{1\wp}+\CA_{U'}[p^\infty]_1^\wp+\CA_{U'}[p^\infty]_{2\wp}+\CA_{U'}[p^\infty]_2^\wp.$$
We define 
$$\CH'_{U'}:=\CA_{U'}[p^\infty]_2.$$
By  part (2) in the definition of $\CF _{1, U'^p}$, the $\wp$-part $\CH'_{U', \wp}$ is a special $O_{B, \wp}$-module,
and the prime-to-$\wp$-part $\CH _{U'}'^\wp$ is an \'etale $O_{B'}^\wp$-module.

It is clear that  the generic fiber $H_{U'}'=A_{U'}[p^\infty]_2$ of  $\CH_{U'}'$ on $X'_{U'}$ is dual to $A_{U'}[p^\infty]_1$ by the polarization; thus $H_{U'}'$ determines the structure of 
$A_{U'}[p^\infty]$. Notice that $H_{U'}'$  can be constructed without using abelian varieties:
$$H'_{U'}=\left(p^{-\infty}O_{B, p}/O_{B, p}\times X'\right)/U_p'(1)\times U'^p.$$
Where $U_p'(1)\simeq  \BZ_p^\times \times O_{B, p}^\times $ acts on $p^{-\infty}O_{B, p}/O_{B, p}$  
by the right multiplication of $O_{B, p}^\times$ (cf. \cite[\S 2.5]{Ca}).
\begin{remark} Our $p$-divisible group $H_{U'}'$ relates to the group  $E_\infty'$ of \cite[\S3.3]{Ca} in the case $O_{B,\wp}\simeq M_2(O_\wp)$ by
$$\begin{pmatrix}1&0\\ 0&0\end{pmatrix}\cdot  H_{U'}'[\wp^\infty]=E_\infty'.$$
\end{remark} 

\subsubsection*{Level structure at $p$}
For any ideal $\gn$ of $O_F$ dividing a power of $p$, let $U_p'(\gn)$ denote the subgroup of $B_p^\times$ of the form $\BZ_p^\times \times (1+\gn O_{B, p})^\times$,
and $X_{\gn, U'^p}'$ denote $X_{U_p'(\gn )\times U'^p}'$. Let $H'_{\gn, U'^p}$ denote the pull-back of $H'_{1, U'^p}=H'_{U'_p(1)U'^p}$ to $X'_{\gn, U'^p}$.
Using the above description, the map $X_{\gn, U'^p}\lra X_{1, U'^p}$ defines  a  full level $\gn$-structure on $H'_{n, U'^p}$, i.e.,
an isomorphism of $O_{B, p}$-modules:
$$\kappa_p: \quad \gn ^{-1}O_{B, p}/O_{B, p}\lra H'_{n, U'^p}[\gn].$$
When $\gn$ is prime to $\gd_B$, this level structure extends to the  minimal model   $\CX_{\gn, U'^p}'$. More precisely, the scheme $\CX_{\gn, U'}'$ represents a functor 
$\CF_{\gn, U'^p}$  over $\CF_{1, U'^p}$  to classify a pair of level structures $\kappa_p=(\kappa_\wp, \kappa_p^\wp)$ so that $\kappa_p^\wp$ is a  full-level structure on the \'etale sheaf $\CH'^\wp_{n, U'^p}[\gn]$, and $\kappa_\wp$ is a   Drinfeld basis of $\CH'_{n, U'^p, \wp}[\gn]$.

\subsubsection*{Integral models}
In the above, we have interpreted $\CX_{\gn, U'^p}'$ at a prime $\wp$ as the functor $\CF_{\gn, U'^p}$
when $\gn$ is prime to $\gd_B$, and  $U'^p$ is sufficiently small (in dependent of $\gn$). In the following, we want to 
extend such interpretation to large $U'^p$.
Fix  a lattice $\Lambda'$ of $B'$ with a completion $\Lambda_p'$. For any positive integer $N$, let $U'(N)$ denote the 
 subgroup of $G'(\wh \BQ) $ consisting of elements which stabilize $\Lambda'$ and induce the identity action on $\Lambda'/N\Lambda'$. 
 
 \begin{pro}\label{free action}
 Assume that $U'$ is contained in $U'(N)$ as a normal subgroup for some 
  $N\ge 3$ and prime to $p$. Then the functor $\CF_{\gn, U'^p}$ is represented by the minimal regular model $\CX_{\gn, U'^p}'$ over $O_\wp$.\end{pro} 
\begin{proof}
First let us reduce the proposition to the case $U'=U'(N)$. In fact if $\CF_{U(N)}$ is represented by $\CA_{U'(N)}\lra \CX_{U'(N)}'$, then $\CF_{\gn, U'^p}$ is represented by an $\CX'_{U'(N)} $-scheme $\CY_{\gn, U'^p}$
to classify a pair $(\kappa_\wp, \kappa^\wp)$ of a full Drinfeld  level structure $\kappa_\wp$ and an etale level structure $\kappa^\wp$. Thus it is clear that $\CY_{\gn, U'^p}$ is regular without any 
exceptional curve. Thus $\CY_{\gn, U'^p}=\CX_{\gn, U'^p}$.

Assume now $U'=U(N)$.  Let $U'^p_0$ be a sufficiently small normal subgroup of $U'(N)^p$ so that $\CF_{1, U'^p}$ is representable by $\CA_{1, U'^p_0}\lra \CX_{1, U_0^p}$. Then we have an action of $U(N)$ on this family. It suffices to show that $U(N)$ acts freely on $\CX_{1, U_0^p}$.
Let $\gamma \in U'(N)$ fixes a closed point $x$ in $\CX_{1, U'^p_0}'$. Let $[A, \iota, \theta, \kappa]$ be the quadruple corresponding to $x$.
Replace $A$ by some abelian variety   prime to $p$ isogenous to $A$, we may assume that $\kappa ^p$ induces an isomorphism morphism between $\wh \Lambda'^p$ and $\wh \RT^p(A)$.
In this way, we have an automorphism $\varphi$ of $(A, \theta)$, an $u\in U'^p$ such that 
$\kappa \cdot \gamma \cdot u=\kappa \circ \RT (\varphi)$. Since $\gamma \in G(N)$, it follows that $\varphi$ fixes all points in $A[N]$. Thus $\varphi=1$.
Thus $\gamma=u^{-1}\in U'$.
\end{proof}

\begin{cor} \label{regularity n}
The integral models $\CX'_{\gn, U'^p}$, with $\gn$ prime to $\gd_B$ and $U'^p$ contained in $U'(N)$ with $N\ge 3$ and prime to $p$,
form a projective system of regular schemes over $O_\wp$. Moreover the special fiber of each $\CX'_{\gn, U'^p}$ above $\wp$ is a smooth curve if $\wp\nmid \gn\gd_B$, and a relative Mumford curve if $\wp\mid \gd_B$.
\end{cor}

\subsection{Curve $X'$ in case 2}

In this subsection, we assume that  $E$ is embedded into $B$ over $F$. Then we can write $B=E+Ej$ where $j\in B^\times $ such that $jx=\bar x j$ for all $x\in E$. We can identify  $B'=B\otimes E$ with $M_2(E)$ by the following maps: 
$$a\otimes b\mapsto \begin{pmatrix}ab& \\ &\bar ab\end{pmatrix}, \qquad j\mapsto \begin{pmatrix}&1\\ j^2&\end{pmatrix}.$$ It follows that $V'=B'$ is the sum of two copies of a subspace $V$ over $E$.
In fact, we can take $V_i=B$ with two conjugate  left  multiplication of $E$
$$V'\iso V_1\oplus V_2:\qquad b\otimes e\longmapsto (eb, \bar eb).$$
The operator $w=\begin{pmatrix}&1\\ 1&\end{pmatrix}$ switches two factors by $(u, v)\mapsto (jv, j^{-1}u)$. 
We may assume that $\gamma '=\gamma \otimes 1$ with $\gamma \in E\otimes 1$ so that $\psi'$ is the sum of two copies of a symplectic form $\psi$ on $V_i=B$ by
$$\psi (u, v)=\tr_{F/\BQ}\tr_{B/F}(\gamma u\bar v), \qquad u, v\in V_i=B.$$
The group $G'$ can be identified with the group of $E$-linear symplectic similitudes of $(V, \psi)$ by right action  on $V$:
 $(b, e)x=exb$.
 
It follows that when $U'$ is sufficiently small,  $X'_{U'}$ represents the following functor 
 $\CF_{U'}^0$ on the category of $F' $-schemes. Here $F'$ is the flex field as before.
For any $F'$-scheme $S$, $\CF_{U'}'^0 (S)$ is the set of isomorphism classes of quadruples  $[A, \iota, \theta, \kappa]$
where
\begin{enumerate}[(1)]
\item $A$ is an abelian scheme over $S$ up to isogeny;
\item $\iota: E\lra \End ^0(A/S)$ is a homomorphism such that the induced action of $E$ on the $\CO_S$-module $\Lie (A/S)$ has the trace given by 
$$\tr (\ell, \Lie A)=t(\ell), \qquad \forall \ell\in E,$$
\item $\theta: A\lra A^t$ is a  polarization whose  Rosati involution on $\End ^0(A/S)$ induces the complex conjugation $c$  of $E$ over $F$; 
\item $\kappa: \wh V \times S\lra H_1(A, \wh \BQ)$ is a  $U'$-orbit of similitudes  of skew hermitian $E$-modules.
\end{enumerate}

Let $O_B$ be a maximal order of $B$,  and let $\Lambda=O_B$ be viewed as a lattice in $V$.
Assume that $\psi$ takes integral value on $\Lambda$. Then $\CF_{U'}'^0$ is equivalent to the following functor $\CF'_{U'}$.
For any $F'$-scheme $S$, $\CF_{U'}' (S)$ is the set of isomorphism classes of quadruples  $[A, \iota, \theta, \kappa]$
where
\begin{enumerate}[(1)]
\item $A$ is an abelian scheme over $S$;
\item $\iota: O_E\lra \End (A/S)$ is a homomorphism such that the induced action of $O_E$ on the $\CO_S$-module $\Lie (A/S)$ has the trace given by 
$$\tr (\ell, \Lie A)=t(\ell), \qquad \forall \ell\in O_E,$$
\item $\theta: A\lra A^t$ is a  polarization whose  Rosati involution on $\End (A/S)$ induces the complex conjugation $c$  of $O_E$ over $O_F$; 
\item $\kappa: \wh \Lambda  \times S\lra H_1(A, \wh \BZ)$ is a  $U'$-orbit of similitudes  of skew hermitian $O_E$-modules.
\end{enumerate}

\subsubsection*{CM points}
Again assume that $E$ is embedded into $B$ over $F$.
Let $T'$ (resp. $\wh T'$) be the subgroup of $G'$ (resp. $G'(\wh \BQ)$)   of  elements $(b, e)\in (E^\times)^2$ (resp. $(b, e)\in (\wh E^\times )^2$).
Then  the subscheme $X'^{T'}$ of $X'$ of points fixed by $T'$ is a principal homogenous space of $\wh T'$. 
Moreover each  point $P'\in X'^{T'}$ represents an abelian variety $A_{P'}$ which is isogenous to a product $A_{\Phi_1}\times A_{\Phi_2}$
of CM abelian varieties by $O_E$ with types $\Phi_1, \Phi_2$. In fact, in terms of  above complex uniformization,
$X'^{T'}$ is represented by pairs $(z, t)$ with $z$ the unique point on $\gh$ fixed by $T$, and $t\in \wh T$.
Fix a point $P'\in X'^{T'}$.

\subsubsection*{Hodge de Rham sequence}

 In the following, we want to study the Kodaira--Spencer map. 
Assume that $\CF_{U'}$ is represented by a universal abelian variety $\pi: A_{U'}\lra X'_{U'}$. Then there is a 
 local system $H_1^\dR (A_{U'})$ of $F\otimes \CO_{X'_{U'}}$-modules  with an integrable connection $\nabla$ and a Hodge filtration
$$0\lra \Omega (A^t_{U'})\lra H_1^\dR (A_{U'})\lra \Omega (A_{U'})^\vee\lra 0,$$
where $\Omega (A_{U'}):=\pi_*(\Omega _{A_{U'}/X'_{U'}})$ and $\Omega (A_{U'}^t):=\pi_*(\Omega _{A_{U'}^t/X'_{U'}})$. 
This sequence of vector bundles on $X_{U'}'$ has an action by $F$ by pulling back of cohomology classes.
Taking a quotient according to the morphism $F\otimes \CO_{X'_{U'}}\lra \CO _{X'_{U'}}$ given by sending $(x\otimes y)\mapsto \tau (x)y$, we have
$$0\lra \Omega (A^t_{U'})^\tau \lra H_1^\dR (A_{U'})^\tau \lra \Omega (A_{U'})^{\tau, \vee}\lra 0.$$
For simplicity, let us introduce  the following notations:
$$M_{U'}:=H_1^\dR (A_{U'})^\tau, \qquad W_{U'}:=\Omega (A_{U'})^\tau,\qquad W_{U'}^t:=W(A^t_{U'})^\tau.$$
Then we have an exact sequence of vector bundles:
\begin{equation}
\label{WMW}
0\lra W_{U'}^t\lra M_{U'}\lra W_{U'}^\vee\lra 0.
\end{equation}

In terms of the complex uniformization, the bundle  $(M_{U'}, \nabla)$ and its filtration can be described explicitly by representations of $G' (\BQ)$ as follows.
First define the local system of $\BR$-vector spaces on ${X'}_{{U'}, \tau'}(\BC)$:
$$\BV:=G (\BQ)\bs {V}_\tau\times \gh^\pm \times G' (\wh \BQ)/{U'}, \qquad {V}_\tau:={V}\otimes _{F, \tau} \BR$$
This system has a Hodge structure given by $\gh^\pm$.
This definition makes sense since the stabilizer of $G (\BQ)$ on every point of  $\gh^\pm \times G' (\wh \BQ)/{U'}$ is its  center $Z(\BQ)$ which acts trivially on ${V}$.
Then we have 
$$M_{U'}=\BV\otimes_\BR \CO _{X'_{U'}}, \qquad W^t_{U'}=H^{0, -1}(\BV), \qquad W_{U'}=(M_{U'}/W_{U'}^t)^\vee.$$

\subsubsection*{Kodaira--Spencer maps at archimedean places}
Inserting the Gauss--Manin connection  to the sequence (\ref{WMW}) gives a chain of morphisms:
$$W^t_{U'} \lra M_{U'} \overset\nabla\lra M_{U'} \otimes\Omega _{X'_{U'}}\lra W^\vee _{U'}\otimes\Omega _{X'_{U'}}.$$
By Kodaira--Spencer, this induces an isomorphism of $E\otimes_{F} \CO _{X'}$-line bundles:
$$W^t_{U'} \lra  W^ \vee_{U'}\otimes\Omega _{X'_{U'}}.$$
Taking determinants, this gives an isomorphism of $\CO_{X'}$-line bundles:
$$\KS_{U'}: N_{U'}\lra \Omega_{X'_{U'}}^{\otimes 2},$$ 
where $N_{U'}$ is a line bundle on $X'_{U'}$ defined by 
$$N_{U'}:=\det W_{U'}\otimes \det W_{U'}^t.$$

In the remaining part of this subsection, we want to study 
the  Kodaira--Spencer isomorphism at a fixed place $\tau'$ of $F'$.
Here we put a metric on $N_{U'}$ by the Hodge theory as in \S\ref{section hermitian}, and put a metric on $\Omega _{X'_{U'}}$ by the following formula 
$$|dz|=2y$$ 
in terms of  the complex unformization.

\begin{thm}  \label{N2Linf} 
The morphism  $\KS _{U'}$ is isometric.
\end{thm}

\begin{proof}
 
 The Kodaira--Spencer isomorphism induces a norm on $\Omega _{X'_{U'}}$. We want to give an explicit description of this metric as follows.
First, let us give an explicit formula for the Kodaira--Spencer map. Fix an  isomorphism $B_\tau={V}_\tau \simeq M_2(\BR)$ and 
identify  $\gh^\pm$ with   the moduli space of $B_\tau$-Hodge structures  on
$M_2(\BR)$. It is equivalent to study the Hodge structures on $\BR^2$.  
In a concrete matter, for each $z\in \gh^\pm$, take a Hodge structure on $L=\BR^2$ inducing a complex structure 
 given  by isomorphisms 
$$\varphi_z: L\lra \BC, \qquad (a, b)\longmapsto a+bz.$$
Then $L^{0, -1}$ is given as $\ker \varphi_{z, \BC}$, so we have 
$$L^{0, -1}_z=\BC e_z, \qquad L^{-1, 0}_z=\BC e_{\bar z}, \quad e_z:=(-z, 1) .$$
Thus the filtration of the de Rham homology has the following form:
$$0\lra \BC e_z\lra \BC^2\lra \BC e_{\bar z}\lra 0.$$
Apply the Gauss--Manin connection to obtain
$$\nabla (e_z)=(-1, 0)dz=\frac {\bar e_z-e_z}{2iy}dz.$$
It follows that under Kodaira--Spencer map, 
$$ dz=2iy\frac { e_z}{\bar e_z}, \qquad |dz|=2y.$$
\end{proof}

\subsubsection*{$p$-divisible groups}
Assume that $U'$ is sufficiently small so that $X'_{U'}$ has a universal abelian scheme $A_{U'}$ representing the functor $\CF_{U'}$.
Then we have a $p$-divisible group
$$H_{U'}':=A_{U'}[p^\infty].$$
Notice that this $p$-divisible group can be constructed directly by the following formula:
$$H_{U'}'=(B_p/O_{B, p}\times X')/U'$$
where $U'$ acts on $B_p/O_p$ via its projection to the subgroup $O_{B, p}^\times\times_{O_{F, p}^\times} O_{E, p}^\times $
of $G(\BQ_p)$ and the action
$$x(b, e)=exb, \qquad b\in B_p /O_{B, p}, \quad (b, e)\in O_{B, p}^\times\times O_{E, p}^\times.$$

\subsubsection*{Integral models}
In this subsection we give some  results about integral models of $X'_{U'}$, $A_{U'}$, and $H_{U'}'$ which can be proved in the later section \ref{integral-X''}.
The results here will not be used in the rest of paper.

Assume that $U'$ is sufficiently small as in the previous paragraph.  A natural question is to extend the universal family  $A_{U'}\to X_{U'}$ to a flat family 
 $\CA_{U'}\to \CX_{U'}$ over $O_{F'}$. The natural way is to extend the functor $\CF _{U'}$ over schemes over $O_{F'}$, which we don't know 
how to define. However we can extend this abelian scheme pointwise on $X_{U'}$.

\begin{prop}\label{integral-A'}
Let $L$ be a finite extension of $F'$ and $x'\in X_{U'}(L)$ a point which represents an abelian variety $A_{x'}$ over $L$.
Then $A_{x'}$ has good reduction $\CA_{x'}$ over $O_L$.
\end{prop}

By the works of Grothendieck \cite{SGA7} and Raynaud \cite{Ra}, it is sufficient to extend $p$-divisible groups locally. 
We will prove this extension in Proposition \ref{integral-IH} using Breuil--Kisin theory.

One consequence of this integral model is to give a hermitian  integral structure on $N_{U', x'}'$ at each point $x\in X_{U'}'(L)$
by $\CN (A, \tau)$.  Using method in \S\ref{integral-X''}, we can construct an integral model $\CX'_{U'}$ of $X'_{U'}$ over $O_{F'}$
and a line bundle $\CN'_{U'}$ such that 
$$\CN(A, \tau)=\CN'_{U',x'}$$
as integral structures on the Hodge bundle $L_{x'}^2$.

  \section{Shimura curve $X$}
 In this section, we study a quaternionic Shimura curves $X$  over a totally real field. 
 We will first review some basic facts  about the integral models $\CX$  studied in Carayol \cite{Ca} at split primes,
 and \v Cerednik--Drinfeld \cite{BC, Ce} at non-split primes. Then  we will  construct  integral models of the curve $X$ by a comparison with the curve $X'$
 in the last section.  Finally we will  study the integral models of $p$-divisible groups $H$ using the $p$-divisible groups $H'|X'$, and study the  local Kodaira--Spencer morphisms 
 induced from the Hodge--de Rham filtration and the Gauss--Manin connections, following a deformation theory of $p$-divisible 
 groups $\CH$ of Grothendieck--Messing \cite{Il, Me}.

 \subsection{Shimura curve $X$} \label{section 4.1}

Let $F$ be a totally real field and  $\BB$  a totally definite incoherent  quaternion algebra over $\BA:=\BA_F$ as before.
Then we have a projective system of Shimura curves $X_U$ over $F$ indexed by open and compact subgroups $U$ of $\BG_f:=\BB_f^\times$,
see \cite {Ca, YZZ}.

For any archimedean place $\tau$ of $F$, the curve $X_{U, \tau}$ over $\BC$ is defined by the following Shimura data $(G, h)$ where $G=\Res _{F/\BQ}(B^\times)$
with $B$ a quaternion algebra over $F$ with the ramification set $\Sigma (\BB)\setminus \{\tau\}$, and $h: \BC^\times\lra  G(\BR)$ a morphism 
as follows.
Fix an isomorphism 
$$G(\BR)=\GL_2(\BR)\times (\BH^\times)^{g-1}.$$
Then $h$ brings $z=x+yi$ to 
$$\left[\begin{pmatrix}x&y\\ -y&x\end{pmatrix}^{-1}, 1, \cdots, 1\right].$$
The class of $G(\BR)$-conjugacy class of $h$ is identified with $\gh^\pm=\BC\setminus \BR$ by  
$$ghg^{-1}\longmapsto g(i), \qquad g\in G(\BR).$$

Fix an isomorphism $\BB_f\simeq \wh B$ which gives an isomorphism $\BG_f\simeq G(\wh \BQ)$.
Then we have a uniformization 
$$X_{U, \tau}(\BC)=G(\BQ)\bs \gh ^\pm \times G(\wh \BQ)/U.$$
This curve is compact if $B\ne M_2(\BQ)$ or equivalently $\Sigma (\BB)$ is not a singlet. In the following discussion 
we always assume that $X_U$ is compact; but the results hold in general with taking care of cusps.

If $F\ne \BQ$, this curve does not parametrizes abelian varieties but its geometric connected component 
can be embedded into Shimura curves of PEL types over $\bar F$.
In the following we want to review the work of Carayol \cite{Ca} on $p$-divisible groups on some integral model of $X_U$ with infinite level.

Let $X$ denote the projective limit of $X_U$. Then $X$ has a right action by $G(\wh \BQ)=\BB^\times_f$. The maximal subgroup of $\BB^\times_f$ which acts trivially on $X$ is  $\ol{F^\times}$,
the closure of $Z(\BQ)=F^\times$ in $\BB_f^\times$. Thus we can write $X_U=X/\ol U$ with $\ol U:=U/(U\cap \ol {F^\times})$.
When $U$ is sufficiently small, $\ol U$ acts freely on $X$.  If $F\neq \BQ$, then $\ol{F^\times}\ne F^\times$. This means that 
the intersection $F^\times \cap U\ne \{1\}$ for any open compact subgroup $U$ of $\ol{F^\times}$. 

Fix a maximal order $O_\BB$ of $\BB_f$  and consider the projective system of Shimura curves $X_U$ indexed by open compact subgroup $U$ of $O_\BB^\times$.
For each positive integer $N$, let $U(N)$ denote
a compact subgroup of $O_\BB^\times$ of the form $U(N):=(1+N O_\BB)^\times$.  
\begin{pro}\label{genus}
If $U$ is contained in $U(N)$ for some $N\ge 3$, then $g(X_U)\ge 2$. 
\end{pro}
\begin{proof}
This can be seen from the above complex uniformization. The curve $X_{U, \tau}$ is a disjoint union of quotients
$X_g:=\Gamma _g\bs \gh$, for $g$ sits in a subset of $G(\wh \BQ)$ representing  the double coset quotient $G(\BQ)\bs G(\wh \BQ)/U$,
and 
$$\Gamma _g=B_+^\times \cap gUg^{-1}\subset B_+^\times \cap (1+N gO_B g^{-1})^\times.$$
  Let $\ol{\Gamma _g}$ denote the quotient $\Gamma _g/(\Gamma _g\cap F^\times)$. We claim that $\ol{\Gamma _g}$ acts freely on $\gh$.
This claim will show that $X_g$ has a  (free) uniformization by $\gh$, thus its genus greater than 1.

Let $\gamma \in \Gamma _g\setminus F^\times $ be  an element fixing a point $z\in \gh$. Then the subfield  $E:=F(\gamma)$ of $B$ generate by $\gamma$ over $F$
is a quadartic CM extension  of $F$.  It is clear  that $\gamma\in O_E^\times$ and $\gamma -1\in NO_E$. Write $\zeta=\gamma /\bar \gamma$. Then $\zeta$ has norm $1$ at all places of $E$. Thus $\zeta$ is a roots of unity with the property $\zeta-1\in N O_E\cap \BQ(\zeta)\subset N\BZ[\zeta]$. It follows that $\BZ[\zeta]/N \BZ[\zeta]=\BZ/N\BZ$.
On the other handn we know that $\BZ[\zeta]/N\BZ[\zeta]$ is a free module over $\BZ/N\BZ$ of rank equal to  $\deg \BQ(\zeta)$. It follows that $\zeta \in \BQ$, or $\zeta =\pm 1$. 
Since $N\ge 3$, $\zeta =1$. It follows that $\gamma \in (1+NO_F)^\times$. 
\end{proof}

\subsubsection*{$p$-divisible groups}
Let $p$ be a prime and fix a maximal order $O_{\BB, p}$ of $\BB_p$ containing $O_{E, p}$.
 For any ideal  $\gn$ of $O_F$ dividing a power of $p$, let $U_p(\gn)$ denote $(1+\gn O_{\BB, p})^\times$.
 Then we have a Shimura curve $X_\gn:=X/U_p(\gn)$. Write $U_p(1)=U_p(O_F)=O_{\BB,p}^\times$ the maximal compact subgroup of $\BB_p^\times$, 
 and $X_1=X_{U_p(1)}$.
 We define the $p$-divisible group $H_\gn$ on $X_\gn$ by
 $$H_\gn=\left[\BB_p/O_{\BB, p}\times X\right]/U_p(\gn),$$
 where $U_p(\gn )\subset O_{\BB, p}^\times$ acts on $\BB_p/O_{\BB, p}$ by right multiplications. 
 This definition makes sense, since $U_p(1)$ acts freely on $X$. Moreover, for each $\gn$,  its  $\gn$-torsion subgroup $H_1[\gn]$ 
 can be descended to $X_{U_p(1)\times U^p}$ for some open compact subgroup $U^p$ of $\BB^{p, \times }_f$ as follows:
 $$H_{U_p(1)\times U^p}[\gn ]=\left[\gn ^{-1}O_\BB/O_\BB\times X/(U_p(\gn )\times U^p)\right]/(U_p(1)/U_p(\gn)).$$
 For this we need to find $U^p$ so that $U_p(1)/U_p(\gn )$ acts freely on $X/(U_p(\gn )\times U^p)$. The existence of such a $U^p$ can be proved in the same 
 way as \cite[Cor. 1.4.1.3]{Ca}. 
 

\subsubsection*{Relation between $X^0$ and $X'^0$}
In the following sections we want to study integral models of $X_U$ and $H_U$ by Carayol \cite{Ca} by relating them to $X'_{U'}$ and $H'_{U'}$ studied in \S3.1 and \S3.2
for Shimura curves defined using imaginary quadratic field $E= F(\sqrt \lambda)$  with $\lambda \in \BQ$ such that
$p$ is split in $\BQ(\sqrt \lambda)$. 

Let $X^0$ be the identity connected component of $X$ over $\bar F$ (which was denoted as $M^+$ in \cite[\S4.1]{Ca}),
and $\ol \Delta$ the stabilizer   of $X^0$ in $\bar G=G(\wh \BQ)/\ol {Z(\BQ)}$.
Then $\ol \Delta$ is represented by the subgroup $\Delta\subset G(\wh \BQ)=\wh B^\times $  consisting of  elements  $g$ with determinants
 $q(g)\in F_+^\times $. In other words, we have $\ol \Delta =\Delta /Z(\BQ)$.

Similarly, let $X'^0$ be the identity connected component of $X'$ over $\bar F$ (which was denoted as in $M'^+$ in \cite{Ca}, \S4.1),
and $\ol \Delta'$ the stabilizer   of $X'^0$ in $\ol G':=\wt G/Z''(\BQ)$.
Then $\ol \Delta'$   is represented by the subgroup $\Delta'\subset G''(\wh \BQ)=\wh E^\times \times _{\wh F^\times}\wh B^\times$  by elements $(e, b)$ with norm $(q(b) e\bar e, e/\bar e)\in F_+^\times \times E_1^\times$
in $F_+^\times $. In other words, we have $\ol \Delta' =\Delta' /Z''(\BQ)$.

It is clear that the embedding $G\lra G''$ induces an isomorphism $\ol\Delta\simeq \ol \Delta'$. Here is the 
 first comparison  result:
 
 \begin{pro}\label{comparison 1}
There is an isomorphism $X^0\simeq X'^0$ with compatible actions by 
$\ol \Delta\simeq \ol \Delta '$.
\end{pro}

\begin{proof}
Same as  Carayol \cite[Prop. 4.2.2]{Ca}.
\end{proof}

For the second fundamental result,  let $p$ be a prime and let $X_1^0$ and $X_1'^0$ be the quotients
$$X_1^0=X^0/O_{B, p}^1, \qquad X_1'^0=X'^0/O_{B, p}^1$$
where $O_{B, p}^1$ the subgroup of $O_{B, p}$ with norm $1$. 
Then $X_1^0$ and $X_1'^0$ are defined over a maximal extension of $F$ which is unramified over every place of $F$ dividing $p$.
Let $\wp$ be a prime of $O_F$ over a prime $p$, and $F_\wp^\ur$ the completion of the maximal unramified extension of $F_\wp$.
Then $X_1^0$ (resp. $X_1'^0$) is the connected component of the limit $X_1$ (resp. $X_1'$) of $X_{1, U^p}$ ($X'_{1, U'^p}$) over $F_\wp^\ur$.
Let $\Delta _0$ denote the subgroup  $\Delta$ consisting of  elements whose components over $p$ are in $O_{B, p}^\times$. Define $\Delta _0'$ in the same way. 
Then $X_1^0$ and   $X_1'^0$ have actions respectively  by $\ol \Delta_0/O_{B, p}^1\subset  \ol \Delta _0/O_{B, p}^1$.

Define the $p$-divisible groups on these schemes by
$$H|X_1^0=\left(B_p/O_{B, p}\times X^0\right)/O_{B, p}^1, \qquad H'|X_1'^0=\left(B_p/O_{B, p}\times X'^0\right)/O_{B, p}^1$$
These are also defined over $F_\wp^\ur$ with natural actions by $\Delta _0/O_{B, p}^1$ and $\Delta _0'/O_{B, p}^1$ respectively. 
Our second comparison result is as follows:
\begin{pro}\label{comparison 2}
There is an  isomorphism of the $p$-divisible groups $H|X_1^0$ and $H'|X_1'^0$ with compatible action by 
$\ol \Delta_0/O_{B, p}^1\subset  \ol \Delta _0/O_{B, p}^1$.
\end{pro}

\begin{proof}
Same as Carayol \cite{Ca}, Proposition 4.4.3.
\end{proof}

Here is a  consequence of the above two comparison results:
\begin{pro}\label{comparison 3}
For any ideal  $\gn$ of $O_F$ dividing a power of $p$ and prime to $\gd_B$,
 and any  sufficiently small open compact $U^p\subset G(\wh \BQ)$ depending on $\gn$, there is an open compact
$U'^p\subset G'(\wh \BQ)$ such that $X_{\gn,  U^p}^0$ is isomorphic to $X'^0_{\gn, U'^p}$ over $K$. 
\end{pro}
\begin{proof}
Same as  Carayol \cite[Prop. 4.5.5]{Ca}.
\end{proof}

\subsection{Integral models and arithmetic Hodge bundles}  
\label{section integral models}

The goal of this subsection is to introduce integral models $\CX_U$ of $X_U$ for any open compact subgroup $U=\prod_v U_v$ of $\BB_f^\times$ which is maximal at every prime ramified in $\BB$. Then we introduce an arithmetic Hodge bundle $\ol\CL_U$ on $\CX_U$. 

\subsubsection*{Integral models of Shimura curves}

By Proposition  \ref{genus}, $X_U$ has a unique minimal regular (projective and flat) model $\CX_U$ over $O_F$ when $U\subset U(N)$ for some $N\ge 3$. We want to check if these integral models form a projective system. More precisely,  for any $U_1\subset U_2\subset U(N)$ there is a morphism $X_{U_1}\lra X_{U_2}$, thus a rational map
$\CX_{U_1}\lra \CX_{U_2}$.   We want to check if this rational map is actually a morphism. For this, we first check the regularity over a prime $\wp$ of $O_F$ dividing a prime $p$.
Let $K=F_\wp^\ur$ be the completion of the maximal unramified extension of $F_\wp$. 
We will consider the open subgroups of $O_\BB^\times$ of the form  $U=U_p(\gn)U^p$, where $U_p (\gn)=(1+\gn O_{\BB, \wp})^\times$ 
for some ideal  $\gn $ dividing a power of $p$, and $U^p$ is an open compact subgroup of $O_{\BB^p}^\times$. Let $\CX_{\gn, U^p}$ denote 
$\CX_{U_p(\gn)\times U^p}$.

\begin{thm} \label{system1}
Consider the system of regular surfaces $\CX_{\gn, U^p}\otimes O_\wp$ indexed by pairs $(\gn, U^p)$ with the following properties:
\begin{enumerate}[(1)]
\item $\gn$ is prime to $\gd_\BB$;
\item $U^p\subset U^p(N):=(1+NO_{\BB^\wp})^\times$ for some $N\ge 3$ and prime to $p$.
\end{enumerate}
Then these surfaces form a projective system of curves over $O_\wp$. Moreover if $\wp\nmid \gn$, each such curve $\CX_{\gn, U^p}\otimes O_\wp$ is smooth if $\wp$ 
is split in $B$,  and a relative Mumford curve if $\wp$ is ramified in $B$.
\end{thm}

\begin{proof} By Proposition \ref{comparison 3}, Theorem \ref {regularity} and Corollary \ref{regularity n},  there is a 
 system of regular  models $\CX'^0_{\gn, U'^p}$ of $X_{\gn, U^p}\simeq X_{\gn, U'^p}'$  (for $U^p$ sufficiently small 
over $O_\wp$ depending on $\gn$) which  is smooth if $\wp\nmid \gn$ 
is split in $B$ and a relative Mumford curve if $\wp\mid \gd_B$.
Under the condition of the theorem, these models must be $\CX_{1, U^p}$ by the uniqueness of the smooth models of curves with genus $\ge 2$.
It remains to enlarge this system to all cases of $U^p$ satisfying the condition of the theorem. 

Let $\CX_\gn$ be the projective limit of $\CX_{\gn, U^p}$, which has generic fiber $X_K/U_p(\gn)$.
  Then $\CX_\gn$ has an action by $\ol{\BB^\times}:=(O_{\BB, p}^\times\cdot \BB^{p, \times}_f)/\ol{O_{(\wp)}^\times }$. For any  open compact subgroup $U^p$, 
  we can construct a normal integral model $\CX_{U, K}$ of $X_{U, K}$ by the categorical quotient:
  $$\CX_{U, K}=\CX_\gn/U=\CX_{U_0, K}/(U/U_0),$$
  where $U_0$ is  a sufficiently small normal subgroup of $U$.
  This model satisfies the condition of the theorem if $\ol {U}:=U/[(U\cap \ol {F^\times})O_{\BB, p}^\times]$
  has a free action on $\CX_\gn$. Thus it suffices to show that $\ol{U(N)}$ acts freely on $\CX_\gn$ for any $N\ge 3$ prime to $p$. 
   Furthermore, we need only check this 
freeness on the identity connected component $\CX^0_1$,   i.e.,   $\ol \Delta (N)_0:=\ol U(N)\cap \ol \Delta_0$ acts freely on $\CX_1^0$.

By our construction,  the model $\CX_1^0$ is isomorphic to the identity connected component $\CX_1'^0$ of the limit  $\CX_0'$ of $\CX_{1, U'^p}'$ 
 constructed in Theorem \ref{regularity} with compatible action by 
 $\ol \Delta_0=\ol \Delta'_0$. Thus it suffices to show that $\ol{\Delta '_0}(N):=\ol\Delta _0'\cap \wt G(N)$ acts freely on  $\CX_1'^0$,
 where $\wt G_0(N)$ is the subgroup of $G_0$ fixes $O_{B'}$ and induces identity on $O_{B'}/NO_{B'}$.
 Let $\delta\in \ol\Delta '_0(N)$ fix a point $x$ on $\CX_1^0$. We want to show that $\delta\in U'^p\cdot F^\times$.
 Let $[A, \iota, \theta, \kappa]$ be the object represented by $x$.
 There is an element $\varphi\in \End (A)\otimes \BZ_{(p)}, u\in U'^p$ such that $\kappa \circ \delta\circ u=\RT(\varphi)\circ \kappa$.
 Replace $\delta$ by $\delta \circ u$, we may simply  assume that $u=1$.
 The effect on the polarization gives an identity $q(\delta)=\varphi\circ \varphi^*\in F_+^\times.$
 It follows that $q(\delta)$ also fixes $x$,  and  that $\delta/\bar \delta $  fixes $x$ too. Since $\delta /\bar \delta \in U'(N)$, by Proposition \ref{free action},
 $\delta=\bar\delta$. Thus $\delta \in O_F^\times$. 
   \end{proof}

Now we extend the definition of the integral model $\CX_U$ to any open compact subgroup $U=\prod_vU_v$ of $\BB_f^\times$ which is maximal at every prime ramified in $\BB$. 
Let $p$ be a prime number coprime to $2\gd_B$ such that $U_p$ is maximal.
Denote $U'=U^pU_p(p)$ with $U_p(p)=(1+p O_{\BB, p})^\times$.
Define $\CX_U$ to be the quotient scheme
 $$\CX_U:=\CX_{U'}\big/U=\CX_{U'}\big/(U/U')=\CX_{U'}\big/(\ol U/\ol U').$$
Here $\ol U:=U/(U\cap \ol {F^\times})$ as before, so the stabilizer of $\ol U/\ol U'$ at the generic point of $\CX_{U'}$ is trivial.
Note that $U/U'$ is a finite group, so $\ol U/\ol U'$ is also a finite group. 
Then $\CX_U$ is a normal integral scheme, projective and flat over $O_F$, and the quotient map $\pi: \CX_{U'}\to \CX_{U}$ is finite of degree $[\ol U:\ol U']$. By Theorem \ref{system1}, the definition does not depend on the choice of $p$. It recovers the minimal regular model if $U\subset U(N)$ for some $N\ge 3$. 

By construction as above, the morphism $\pi: \CX_{U'}\lra \CX_U$ is flat at all codimension one points but not necessarily at all points.
Thus $\pi_*\CO_{\CX_{U'}}$ is not  necessarily a locally free sheaf over $\CX_U$. But we can still define the norm map
$\RN _\pi:  \pi _*\CO_{\CX_{U'}}\lra \CO_{\CX_U}$ by
$$\RN_\pi (f):=\prod _{u\in \ol U/\ol U'} u^*f.$$

Using this norm map, for any line bundle $\CL$ on $\CX_{U'}$ we can define the norm bundle $\RN_\pi (\CL)$ on $\CX_U$ as the line bundle locally generated by 
the symbols $\RN _\pi (\ell)$, where $\ell$ are sections of $\pi_*\CL$, with relations for local sections $f$ of $\pi_*\CO_{\CX_{U'}}$:
$$\RN _\pi (f\ell)=\RN _\pi (f)\cdot \RN _\pi(\ell).$$
It is clear that if $\CM$ is a line bundle on $\CX_U$, then we have
$$\RN _\pi (\pi^*\CM)=\deg \pi\cdot \CM.$$

\begin{cor} \label{system2}
Consider the system $\{\CX_U\}_U$ of surfaces with $U=\prod_vU_v$ maximal at every prime ramified in $\BB$. Then this system is a projective system of curves over $O_F$ extending the system $\{X_U\}_U$. Moreover, the following are true:
\begin{enumerate}[(1)]
\item If $U\subset U(N)$ for some $N\ge 3$, then $\CX_U$ is smooth at any prime $\wp\nmid \gd_B$ such that $U_\wp$ is maximal, and is a relative Mumford curve at any prime $\wp\mid \gd_B$.
\item Let $\CX_U$ be any element in the system. 
Let $H$ be any finite extension of $F$ which is unramified above every finite prime $v$ of $F$ such that $\BB_v$ is ramified or $U_v$ is not maximal. Then the base change $\CX_U\otimes_{O_F}O_H$ is 
$\BQ$-factorial in the sense that any Weil divisor of $\CX_U\otimes_{O_F}O_H$ has a positive multiple which is Cartier. 
\end{enumerate}
\end{cor}
\begin{proof}
We already know (1) from Theorem \ref{system1}. 
For (2), to illustrate the idea, we first treat the case $H=F$. 
Let $\pi:\CX_{U'}\to \CX_{U}$ be a quotient map in the construction of $\CX_U$, where $U'=U^pU_p(p)$ and $U_p(p)=(1+p O_{\BB, p})^\times$ are as above.
Let $C$ be a prime divisor of $\CX_U$.
The schematic preimage $\pi^{-1}(C)$ in $\CX_{U'}$ is locally defined by a single equation $f\in \CO_{\CX_{U'}}$ since $\CX_{U'}$ is regular. 
Then the divisor $(\deg \pi)\cdot C$ is locally defined by the image of $f$ under the norm map $\RN_\pi: \pi_*\CO_{\CX_{U'}}\to \CO_{\CX_{U}}$. 
This proves the case $H=K$. 
In general, the map 
$\CX_{U'}\otimes O_H\to \CX_{U} \otimes O_H$
is still a quotient map by the same finite group $U/U'$.
By (1), $\CX_{U'}\otimes O_H[1/p]$ is regular.
Then the same proof shows that $\CX_{U}\otimes O_H[1/p]$ is $\BQ$-factorial.
Take a different prime $p'$ and apply the same argument. 
Then $\CX_{U}\otimes O_H[1/p']$ is also $\BQ$-factorial.
This implies the result for $\CX_{U}\otimes O_H$.
\end{proof}

For any ideal $\gn$ of $O_F$, let $U(\gn)$ denote the compact group $U(\gn )=(1+\gn O_\BB)^\times$. 
Let $\CX(\gn)$ denote the integral model $\CX_{U(\gn)}$ over $O_F$ if 
$\gn$ is coprime to $\gd_B$.
In particular we have an integral model $\CX(1):=\CX(O_F)$ which is a normal, projective, and flat scheme over $O_F$, and every $\CX(\gn)$ is the normalization of $\CX(1)$ in the projection $X(\gn)\lra X(1)$. 
 
 In the modular curve case, $\CX(1)\simeq  \BP ^1_\BZ$ is regular. In general, it is not clear if $\CX(1)$ is regular. For the purpose of intersection theory, the property of being $\BQ$-factorial is sufficient.

\subsubsection*{Arithmetic Hodge bundle}
 
For any scheme $S$, denote by $\mathcal{P}ic (S)$ the groupoid of line bundles on $S$, and by $\Pic (S)$ the group of isomorphism classes of line bundles on $S$.
Denote by $\mathcal{P}ic (S)_\QQ$ the groupoid  of $\BQ$-line bundles on $S$.
The objects of $\mathcal{P}ic (S)_\QQ$ are of the form $a L$ with $a\in \BQ$ and $L\in \mathcal{P}ic (S)$.
The homomorphism of two such objects is defined to be 
$$\Isom (aL, bM):=\varinjlim_m \Isom (L ^{\otimes am}, M^{\otimes bm}),$$
where $m$ runs through positive integers such that $am$ and $bm$ are both integers. The group of isomorphism classes of such $\BQ$-line bundles is isomorphic to $\Pic (S)_\BQ:=\Pic (S)\otimes \BQ$.

Similarly, we define the groupoid $\wh{\mathcal{P}ic} (S)_\BQ$ of hermitian $\BQ$-line bundles on an arithmetic variety $S$.  
We will usually write the tensor products of (hermitian) line bundles additively. 

In \cite[\S 3.1.3]{YZZ}, for each open compact subgroup $U$ of $\BB_f$, 
the  curve $X_U$ has a {\em Hodge bundle} $L_U\in \mathcal{P}ic (X_U)_\BQ$.
It is the $\BQ$-line bundle for holomorphic modular forms of weight two, and it is  the canonical bundle modified by ramification   points. 
It is determined by the following two conditions:
\begin{enumerate}[(1)]
\item The system $\{L_U\}_U$ is compatible with pull-back maps.
\item If $\bar U$ acts freely on $X$, then $L_U=\omega _{X_U/F}$.
\end{enumerate}
For general $U$,  we have the following explicit formula.
$$
L_{U}= \omega_{X_{U}/F} +\sum_{Q\in X_U(\overline F)} (1-e_Q^{-1})\ \CO(Q).
$$
where the operation in $\mathcal Pic (X_U)_\BQ$ is written additively, and $e_Q$ is the ramification index of the map
$X\lra X_U$.

Next, we want to extend the Hodge bundle $L_U$ to a hermitian $\QQ$-line bundle $\overline \CL_U $ over $\CX_U$
for $U=\prod_vU_v$ maximal at every prime ramified in $\BB$. 
Note that our definition is different from that of \cite[\S 7.2.1]{YZZ} including the normalization of the hermitian metric. 

\begin{thm}\label{hodge bundles} \label{system3}
There is a unique system $\{\ol \CL_U\}_{U}$ of hermitian $\BQ$-line bundles $\ol \CL_U$ on the arithmetic surface $\CX_U$ extending the system $\{L_U\}_{U}$, where $U=\prod_vU_v$ is maximal at every prime ramified in $\BB$,  so that the following  conditions hold:
\begin{enumerate}[(1)]
\item The system $\{\ol \CL_U\}_{U}$ is invariant under the pull-back maps among different $U$.
\item If $U$ is sufficiently small in the sense that $U\subset U(N)$ for some $N\ge 3$,  then there is a canonical isomorphism for any $\wp$ such that $U_\wp$ is 
maximal
$$\CL_{U}\otimes O_\wp=\omega _{\CX_U\otimes O_\wp/O_\wp}.$$
Here the right-hand side denotes the relative dualizing sheaf. 
\item At an archimedean place, the metric is given by $|dz|=2y$ under the complex uniformization.
\end{enumerate}
\end{thm}

\begin{proof}
The third property is simply a definition of metrics. So we only need to consider the first two properties. To construct the system, by pull-back, it suffices to construct the $\BQ$-line bundle $\CL_U$ for the maximal compact subgroup $U=O_{\BB_f}^\times$ of $\BB_f^\times$.  
Let $\pi: \CX_{U'}\to \CX_{U}$ be a quotient map in the construction of $\CX_U$.
Then $U'=U^pU_p(p)$ with $U_p(p)=(1+p O_{\BB, p})^\times$ for some prime $p$ coprime to $2\gd_B$. 
Let $\omega^p=\omega_{\CX_{U'}[1/p]/O_F[1/p]}$ be the relative dualizing sheaf of $\CX_{U'}$ away from $p$. 
Here we write $\CX_{U'}[1/p]=\CX_{U'}\otimes O_F[1/p]$.
Then the bundle  $\RN_{\pi}(\omega^p)$ is a line bundle on $\CX_{U}[1/p]$ with restriction $\deg \pi L_U$ on the generic fiber $X_U$.
Then $\frac{1}{\deg(\pi)} \RN_{\pi}(\omega^p)$ already defines the restriction of $\CL_U$ to $\CX_{U}[1/p]$. 
To get the whole $\CL_{U}$, take a different prime $p'$, and glue $\frac{1}{\deg(\pi)} N_{\pi}(\omega^p)$ and $\frac{1}{\deg(\pi')} N_{\pi'}(\omega^{p'})$ along $\CX_{U}[1/pp']$. 
This finishes the proof. 
\end{proof}

For any ideal $\gn$ of $O_F$ coprime to $\gd_B$, 
 we have written $\CX(\gn)$ for $\CX_{U(\gn)}$. Here
$U(\gn )=(1+\gn O_\BB)^\times$. 
Write $(L(\gn), \CL(\gn), \ol \CL(\gn))$ for $(L_{U(\gn)}, \CL_{U(\gn)}, \ol \CL_{U(\gn)})$ similarly.

\begin{remark} \label{stack}
For an alternative approach of this paper,
instead of defining $\CX_U$ as the quotient scheme $\CX_{U'}\big/(\ol U/\ol U')$, 
one may define it as the quotient stack $\left [\CX_{U'}\big/(\ol U/\ol U')\right]$.
It is a regular Deligne--Mumford stack, proper and flat over $O_F$. 
The quotient scheme is just the coarse scheme of the quotient stack. 
Then one may define $\CL_U$ to be the relative dualizing sheaf of the quotient stack. 
\end{remark}

\subsection{Integral models of $p$-divisible groups}
 Let $\wp$ be a prime of $O_F$ dividing $p$,  
 and $O_\wp$ the ring of integers in $F_\wp$, and  $H=H_\wp\times H^\wp$ the decomposition according to 
 the decomposition $O_{F, p}=O_\wp \oplus O_{F, p}^\wp$ of $\BZ_p$-algebras.
When $\BB_\wp\simeq M_2(F_\wp)$ is split,  
 Carayol \cite[\S1.4.4]{Ca} has defined a $p$-divisible group $E_\infty|M_0$ related to our 
 $H| X_1$ by the formula:
 $$M_0/U_p(1)=X_1, \qquad E_\infty =\begin{pmatrix}1&0\\ 0&0\end{pmatrix} H_{\wp}|_{M_0}.$$
 The treatment of all  facts in Carayol \cite{Ca} can be copied to $H|X_1$ with some little  modifications. 
 In the following, we want to use his method to study integral model for $H|X_1$.
  
  Let $K=F_\wp^\ur$ be the completion of the maximal unramified extension of $F_\wp$, and $O_K$ its ring of integers.
  
  \begin{thm} \label{H-integrality}
  Let $\gn$ be an ideal of $O_F$ prime to $\gd_B$, and 
  $\CX_\gn $  the projective limit of $\CX_{U_p(\gn)U^p}\otimes O_K$ as $U^p$ varies.  
 Then $H_\gn$ has an integral  model $\CH_\gn $ over $\CX _\gn $ with the following properties:
\begin{enumerate}[(1)]
 \item  $\CH^\wp$ 
 is \'etale over $\CX_1$, and   $\CH_\wp$ is
 a special formal $O_{B, \wp}$-module in the sense that
  $\Lie (\CH_\wp)$ is a locally free sheaf over $\CO _{\CX_{1,\wp}}\otimes O_{K_0}$ of rank $1$ where $K_0$ is an unramified quadratic extension 
of $F_\wp$ embedded into $\BB_\wp$.  
\item the formal completion $\wh \CX_{1}$ along its special fiber over $\bar k$ ($k=O_F/\wp$)
is the universal deformation space of $\CH_{\bar k}$;
\item  for any $\gn$ prime to $\gd_B$ and with decomposition $\gn=\wp^n\cdot \gn'$ with $\gn'$ prime to $\wp$, 
the morphism $\CX_\gn\lra \CX_1$ classifies pairs of a  full level-$\gn'$ structure on on $\CH_1^\wp$ and 
 a Drinfeld level $\wp^n$-structure on $\CH_{1, \wp}$.
 \end{enumerate}
\end{thm}

\begin{proof} It suffices to prove the corresponding statement for the connected component $X^0_\gn$ of $X_\gn$.
By Proposition \ref{comparison 3}, $H|X_\gn^0$ is isomorphic to $H'|X'^0_\gn$.
Thus the all conclusions of above theorem  follow from Theorem \ref{free action}. See also   Carayol \cite[\S6.4, \S6.6,  \S7.2, \S7.4, \S9.5]{Ca} 
and \v Cerednik--Drinfeld \cite{BC}.
\end{proof}

Let us define $\CM_\wp=\BD(\CH_\wp)$ to be the covariant  Deudonn\'e crystal \cite{Il, Me},
and $\CW_\wp=\Lie (\CH)^\vee$, $\CW_\wp^t=\Lie (\CH^t)^\vee$, where $\CH_\wp^t$ is the Cartier dual of $\CH_\wp$. Then we have an exact sequence 
$$0\lra \CW _\wp^t\lra \CM_\wp\lra \CW_\wp ^\vee\lra 0.$$
Applying the Gauss--Manin connection $\nabla$ on $\CM_\wp$, we obtain the following composition of morphisms:
 $$\CW_\wp^t\lra \CM_\wp\overset {\nabla}\lra \CM_\wp\otimes \omega _{\CX_\wp}\lra \CW_\wp^{\vee}\otimes \omega _{\CX_\wp}.$$
Taking determinants, we obtain a morphism
$$ \det \CW_\wp^t \lra \det \CW_\wp^{\vee}\otimes  \omega _{\CX_\wp}^{\otimes 2}.$$
In other words, we obtain a Kodaira--Spencer  morphism of line bundles:
$$\KS_\wp: \CN_\wp\lra \omega _{\CX_\wp}^{\otimes 2}, \qquad \CN_\wp:=\det \CW_\wp^t \otimes \det \CW_\wp^{\vee}.$$

\begin{thm}  \label{N2Lfin} 
Let $\gd_{\BB, \wp}$ be  the divisor on $\Spec\, O_\wp^\ur $ corresponding to $\BB_\wp$.
Then $\KS _\wp$ extends to an isomorphism  of line bundles on $\CX_\wp$:
$$\KS _\wp: \CN_\wp\iso \omega _{\CX_\wp}^{\otimes 2}(-\gd_{\BB, \wp}).$$
\end{thm}

\begin{proof}
 Let $(\wh \CX_\wp, \wh \CH_\wp)$ be the formal completion of pair $(\CX_\wp,  \CH_\wp)$ along its special fiber over the residue field $\bar k:=\ol{k(\wp)}$
of $O_\wp^\ur$. Then $(\wh \CX _\wp, \wh \CH_\wp)$ is the universal deformation of $(\CX_{\wp,\bar k}, \CH_{\wp, \bar k})$.
By deformation theory of $p$-divisible groups \cite{Il} and \cite {Me},  we have an isomorphism 
$$\omega _{\CX_\wp}^\vee\iso\Hom _{O_{\BB_\wp}}(\CW_\wp^t, \CW_\wp^\vee)$$
induced from the  above  composition of morphisms:
 $$\CW_\wp^t\lra \CM_\wp\overset {\nabla}\lra \CM_\wp\otimes \omega _{\CX_\wp}\lra \CW_\wp^{\vee}\otimes \omega _{\CX_\wp}.$$
Taking determinants, we obtain an embedding 
 $$\omega _{\CX_\wp}^{-2}\subset \CN_\wp^\vee .$$
  
 If $\wp$ is  split in $\BB$, then we can write $O_{\BB, \wp}=M_2(O_\wp)$.
 Using idempotents $e_1=\begin{pmatrix}1&0\\ 0&0\end{pmatrix}$ and $e_2=\begin{pmatrix}0&0\\ 0&1\end{pmatrix}$,
  we can write $\Omega (\CH_0)$ (resp.  $\Omega (\CH^t_\wp)$)  as a direct sum of components $\Omega (\CH_\wp)^i:=e_i\Omega (\CH_\wp)$
 (resp.  $\Omega (\CH^t_\wp)^i=e_i \Omega (\CH^t_\wp)$). These two components are isomorphic by the operator $\begin{pmatrix}0&1\\ 1&0\end{pmatrix}$.
 Thus we have 
 $$\Omega _{\CX_{\wp}}^\vee \simeq \Hom _{O_v}(\Omega (\CH_\wp^t)^{i}, \Omega (\CH_\wp)^{i\vee})=\Omega (\CH_\wp ^t)^{i\vee}\otimes \Omega (\CH_\wp)^{i\vee}.$$
 This shows in particular that 
 $$\omega _{\CX_\wp}^2= \CN_{\wp}.$$

 Now assume that $\wp$ is nonsplit in $F$. Then $\CM_\wp$ is a free module over $O_{\BB, \wp}\otimes O_{\CX_\wp}$.
Let $K$ be a unramified extension of $F_\wp$ in $\BB_\wp$.
 Then we have a decomposition
 $$O_{\BB, \wp}=O_K+O _K j$$
 where  $j$ a uniformizer of $O_{\BB, \wp}$ such that $jx=\bar xj$ for all $x\in O_K$. 
  Making a base change to $O_K$, then we have a decomposition of $\Omega (\CH_\wp)$ to the direct sum of the eigenspaces of $O_K$ according 
  to the embedding $O_K\lra \CO_{\CX_{U, \wp}}$ and its conjugate:
 $$\Omega  (\CH_\wp^t) =\CL_1\oplus \CL_2, \qquad (resp. \quad \Omega  (\CH_\wp )^\vee =\CN_1\oplus \CN_2)$$
 The action of $j$ has grade $\BZ/2\BZ$ with $j^2=\pi$ a uniformaizer of $O_\wp$.  Let  $j_1$ and $j_2$ be the restrictions of $j$ on two components, then $j_1\circ j_2=\pi$.
 It follows for each point on $\CX_{\wp}$, exactly one of  $j_1$ or $j_2$ is an  isomorphism. Thus we can  assign a type $i\in\{1, 2\}$ to $\Omega (\CH_\wp)$ if $j_i$ is an isomorphism.
 Notice that the types of $\Omega (\CH_\wp)$ and $\Omega (\CH^t_\wp)^\vee$ are opposite. 
 
 We claim that the condition $j_1\circ j_2=\pi$ implies the following identity:
  $$\pi\omega _{\CX_\wp}^{2}= \CN _{\wp}.$$
  To prove this claim, without loss of generality, we assume that
  $\CL_2=j\CL_1$ and $\CN_1=j\CN_2$.
  Now an element $\alpha\in \Omega _{\CX_\wp}$ corresponds a pair of morphism of line bundles 
  $$\phi_i: \CL_i\lra \CN_i$$
  compatible with action of $j$.  It is clear that this morphism  determines and is determinated by $\phi_1$,
  and that $\phi_2=j\phi_1  j^{-1}$ always has image included into $\pi \CN_2$.
  Conversely, for any morphism $\phi_2$ divided by $\pi$, the above equation determines a $\phi_1$. 
  Our claim follows from this  description of $\phi_1\otimes \phi_2$. 
\end{proof}

Define a system of $\BQ$-line bundles $\ol\CN(\gn)$ on $\CX(\gn)$ by
$$\ol\CN (\gn)=\ol\CL (\gn)^{\otimes 2}(-\gd_B).$$
Then the following Theorem \ref{N2Lfin} shows that for any prime $\wp$ of $O_F$, this bundle has the pulling back  $\CN _\wp$ on $X/(\CO_{\BB, \wp}^\times)$.

\section{Shimura curve $X''$}

In this section, we study the relation between   Shimura curves $X$  and  $X'$ in case 2: $\BA_E$ is embedded into $\BB$.
For this, we need to consider   another  Shimura curve $X''$ which includes both $X$ and $X'$.
  We will first study some  basic properties of $X''$, especially the $p$-divisible groups parametrized by $X''$, and the construction of $X''$ using $X$ and a Shimura variety $Y$ of 
  dimension $0$.
 Then we construct an integral model $\CX''$ of $X''$ using  the integral model $\CX$, and a $p$-divisible group $\CH''_{x''}$ for each $p$-adic point $x''$  of $X''$
 using Breuil--Kisin's theory \cite{Ki1, Ki2}. We show that the   deformations of the $p$-divisible group $\CH''_{x''}$ is given by deformations of $\CH_x$.
 Finally, we use all results in this section to complete the proof of Theorem \ref{U2O}.

\subsection{Shimura curve $X''$}

Let $(\Phi_1, \Phi_2)$ be a nearby pair of CM types of $E$, and $F'$ the reflex field of $\Phi_1+\Phi_2$. 
In the following, we want to define a  Shimura curves $X''$   defined over  $F'$, depending  on $(\Phi_1, \Phi_2)$, and with an action  by the group 
$$\BG'':=\BB^\times \times_{\BA^\times} \BA_E^\times.$$
The stabilizer subgroup  $\ol {Z''}$ is generated by $(1, x)$ with $x\in \ol {E^\times}$,  the closure of $E^\times $
in $\wh E^\times$. The scheme $X''$  includes ${X'}$ as a union of connected components via the embedding 
 ${\BG'}\lra {\BG''}$.

 At an  archimedean place $\tau '$ of $F'$ over a place $\tau$ of $F$, we define a reductive  group over $\BQ$ as follows:
 $$G''=B^\times \times_{F^\times}  E^\times,$$
 where as before $B$ is a quaternion algebra over $F$ with ramification set $\Sigma (\BB)\setminus \{\tau\}$. 
 Then we have an embedding $G'\lra G''$. The Hodge structure $h': \BC^\times\lra G'(\BR)$ induces the Hodge structure $h'': \BC^\times\lra G''(\BR)$.
 The congugacy class of $h''$ is $\gh^\pm$. It is easy to show that the reflex field of $(G'', h'')$ is still $F'$. Thus for each open compact subgroup $U$ of $G''(\wh \BQ)\simeq \BG''_f$,
 we have a Shimura curve $X''_U$ over $F'$ with  uniformization at $\tau'$ given by 
 $$X''_{U, \tau '}(\BC)=G'' (\BQ)\bs \gh^\pm \times G''(\wh \BQ)/U.$$
 Let $X''$ be the projective limit of $X_U''$. Then $X''$ has a uniformization as follows:
 $$X''_{\tau '}(\BC)=G'' (\BQ)\bs \gh^\pm \times G''(\wh \BQ)/\ol {Z''}$$
  The embedding $G'\lra G''$ defines an embedding $i: X'\lra X''$.

In the following, we want to study the relation between  $X$ and  $X''$. 
First let us start with a  Shimura variety $Y$ of dimension $0$ defined by the group $E^\times$ with the Hodge structure on $h_\Psi:\BC^\times \lra (E\otimes \BR)^\times$ given by the composition of 
$$
\BC^\times \lra  (\BC^\times)^g, \quad
z\longmapsto (1, z^{-1}, \cdots, z^{-1})
$$
with the inverse of the isomorphism $\Phi_1: (E\otimes \BR)^\times \lra (\BC^\times)^g$. Here the component $1$ corresponds to the unique element of $\Phi_1\setminus \Phi_2$. 
Note that $h_\Psi$ is determined by $\Psi=\Phi_1\cap \Phi_2$. 
For any open compact subgroup $J$ of $\wh E^\times$, we have a  Shimura variety $Y_J$ of dimension zero defined over $F'$ (which include the reflex field of $h_\Psi$).
This set has an action by $\wh E^\times$. In fact the set of its geometric points is a homogenous space over $E^\times \bs \wh E^\times/J$.
Let  $Y$ be the projective limit of $Y_J$. Then the set of geometric points of $Y$ is a principal homogenous space over $\ol {E^\times }\bs \wh E^\times$,
where $\ol{E^\times}$ is the closure of $E^\times$ in $\wh E^\times$.

At the archimedean place $\tau'$ of $F'$ over a place $\tau$ of $F$ as above, 
the product 
$$(X_U\times _F Y_J)_{\tau'}=X_{U, \tau} \times_\BC  Y_{J, \tau'}$$
of Shimura varieties over $\BC$ 
is defined by the reductive group $B^\times \times E^\times$ and the product of Hodge structures 
$(G\times E^\times, h\times h_\Psi)$.  
We have a natural homomorphism of reductive groups:
$$B^\times \times E^\times \lra G=B^\times \times_{F^\times} E^\times.$$
which is  compatible with  the Hodge structures. 
Thus we have a surjective morphism of Shimura curves over $F'$:
$$f: X_U\times_F Y_J\lra X''_{U''}$$
where $U''$ is the image of $U\times J$. 
Taking limits, we obtain a morphism of schemes over $F'$:
$$X\times_F Y\lra X''.$$
This morphism is compatible with the actions of $\BG_f$, $\wh E^\times$, and 
${\BG''}_f$ and induces  an isomorphism:
$$f: (X\times_F Y)/\Delta (\wh F^\times)\iso X'',$$
where $\Delta$ is the twisted diagonal map
 $$\Delta: \wh F^\times\lra \wh B^\times \times \wh E^\times, \qquad z\longmapsto (z, z^{-1}).$$
The isomorphism property of $f$  can be checked at the place $\tau'$ using uniformizations of $X, Y, X''$.

\subsubsection*{$p$-divisible groups}
Fix a prime number $p$ and a maximal  order $O_{\BB, p}$ containing $O_{E, p}$,
 we want to study certain $p$-divisible groups parametrized by $X''_{U''}$ and $Y_J$. 
 Write $\Lambda _p=O_{\BB, p}$ as a left $O_{\BB, p}$-module. 
 For any idea $\gn$ of $O_F$ dividing a power of $p$, denote by $U''_p(\gn )$ the closed subgroup of $\BG''_p$ fixing $\Lambda _p$ and acting trivially on
 $\Lambda_p /\gn  \Lambda_p $. Write $U_p''(1)=U_p'' (O_F)$.
 Then we define 
 $$X_1''=X''/U''_p(1), \qquad Y_1=Y/O _{E, p}^\times.$$
With our previous definition of $X_1$,  we have an isomorphism  
$$f_1: (X_1\times _F Y_1)/\Delta (\wh F^\times)\iso X''_1.$$

Define the $p$-divisible groups on $Y_1$ and $X''_1$ by making quotients
$$H''=\left[\BB_p/O_{\BB, p}\times X''\right]/U''_p(1),\qquad I=(E_p/O_{E, p}\times Y)/O_{E, p}^\times.$$
Here $U''_p(1)$ (resp. $O_{E, p}^\times$) acts on $\BB_p/O_{\BB, p}$  (resp. $E_p/O_{E, p}$) on the right hand side as follows:
$$x\cdot (b, e)=exb, \qquad x\in \BB_p/O_{\BB, p},\ (b, e)\in U''(1).$$
$$\left(\text{resp.}\quad y\cdot e=ey,\qquad y\in E_p/O_{E, p}, e\in O_{E, p}^\times.\right)$$
These definitions make sense since $U''(1)$ and $O_{E, p}^\times$ act freely on $X''$ and $Y$ respectively.
These groups can be defined on finite levels as in the case of $H$ over $X_1$.
We sketch the case of $H''$ as follows.  The group $H''$  is a direct limit of finite  subgroups
 $H''[p^n]$. Each $H''[p^n]$  descends to a quotient ${X''}/(U''(1)\times U''^p)$ for some compact open subgroup $U''^p$  of $({\BG''})^p$ by the formula
 $$H''_{U''_p(1)\times U^p}[p^n]=\left[p^{-n}\Lambda_p/\Lambda _p\times X''/(U''_p(p^n)\times U''^p)\right]/(U''_p(1)\times U''^p).$$
 For this we need to find $U''^p$ so that $U''_p(1)/U''_p(p^n)$ acts freely on $X''/(U''_p(p^n)\times U''^p)$. 
 This can be done by copying the argument in the proof of \cite[Corollary 1.4.1.3]{Ca}.
 It is clear that $H'=H''|_{X'}$.
The groups $H$, $H''$ and $I$ are related as follows:

\begin{pro}\label{TTT} Let $\pi_1$ and $\pi_2$ be the projections of $X_1\times _F Y_1$ to the two factors,
and $\RT (H'')$, $\RT (H)$, $\RT (I)$ be the Tate modules of the corresponding $p$-divisible groups.
There is a canonical isomorphism of \'etale sheaves on $X_1\times _F Y_1$:
$$f_1^*\RT(H'')\iso \pi_1^*\RT(H)\otimes _{O_{E, p}}\pi_2^*\RT(I ).$$
\end{pro}

\begin{proof} By  definitions, the Tate modules of these groups can be written as follows:
$$\RT(H)=(O_{\BB_p}\times X)/U(1), \qquad \RT (H'')=(O_{\BB_p}\times X'')/U''(1), \qquad \RT (I)=(O_{E, p}\times Y)/O_{E, p}^\times.$$
\end{proof}

\subsection{Integral models}\label{integral-X''}
Let $\wp'$ be a finite place of $F'$ dividing $p$, and let $\wp$ be a place of $F$ under $\wp'$.
Let $F_{\wp'}'^{\ur}$ be the completion of the maximal unramified extension of $F'_{\wp'}$, 
which is a finite extension of $F_\wp^\ur$. For simplicity, we introduce the following notations:
$K:=F_\wp^\ur$ and $ K':=F_{\wp'}'^\ur$.

Consider the following  schemes:
$$X_{1,\wp}=X_1\otimes _F K, \qquad X''_{1, \wp'}=X'_1\otimes _{F'} K', 
\qquad Y_{1, \wp'}:=Y_1\otimes _{F'}K'.$$
Then we have an isomorphism:
$$f_{\wp'}: X_{1,\wp}\times _K Y_{1, \wp'}/\Delta (\wh F^\times)\iso X''_{1, \wp'}.$$
By construction,  all geometric points of $Y_1$ are defined over $K'$.
Thus $Y_{1, \wp'}$ is a principal homogenous space of $\ol{E^\times }\bs \wh E^\times /O_{E, p}^\times $.
In this way, the integral model $\CX_{1, \wp}$  of $X_{1, \wp}$ and the model   $\Spec\, O_{K'}$ of $\Spec\, K'$ 
 induce an integral model $\CX''_{1, \wp'}$ for
$X''_{1, \wp'}$. This in turn induces an integral model $\CX'_{1, \wp'}$ by the embedding $X'_{1, \wp'}\lra X''_{1, \wp'}$.

Notice that if $\wp$ does not divide $\gd_\BB$, then $\CX_{1, \wp}$ is smooth over $O_K$. It follows that both $\CX'_{1, \wp'}$ and 
$\CX''_{1, \wp'}$ are smooth over $O_{K'}$. If $\wp$ divides $\gd_\BB$, then $\CX_{1, \wp}$ is a regular and stable Mumford curve. 
It follows that $\CX'_{\wp'}$ and $\CX''_{\wp'}$ are both stable Mumford curves. Notice that they are not regular if $\wp$ is ramified in $F'$.

Recall that we have defined a line bundle $\CN_{1, \wp}$ on
$\CX_\wp$ extending $\omega _{X_{1, \wp}}^2$. This bundle induces bundles $\CN'_{1, \wp'}$ and $\CN''_{1, \wp'}$ on $\CX'_{1, \wp'}$
and $\CX''_{1, \wp'}$ respectively.

Now we would like to extend the groups $I, H', H''$  to integral models $\CI, \CH', \CH''$ point by point using Breuil--Kisin's classification of $p$-divisible group \cite{Ki1}:  any crystalline representation of
$G_{K'}:=\Gal (\bar K'/K')$ of Hodge--Tate weights $0$ or $-1$ arises from a $p$-divisible group over $O_{K'}$. 

\begin{pro}\label{integral-IH}
 Let $L$ be a finite extension of $K'$. For each point $y\in Y(L)$ (resp. $x'\in X'(L)$, resp.  $x''\in X''(L)$) the group $I_y$ (resp. 
 $H'_{x'}$, $H''_{x''}$) over $L$ extends  uniquely to a  $p$-divisible group over $O_L$.
\end{pro}

\begin{proof}
 For $I$,  recall that the action of $G_{K'}$ on $\RT (I)\simeq O_{E, p}$ is given by  the reciprocity map for the type $(E, \Phi_1\cap \Phi_2)$. 
 Fix an isomorphism $\BC\simeq \bar\BQ_p$. Then $\RT (I)\times_{\BQ_p}\bar \BQ_p$ is a direct sum of one-dimensional spaces 
$V_\sigma $ indexed by $\sigma \in \Hom (E, \bar\BQ_p)=\Hom (E, \BC)$.
The action of $G_{K'}$ on $V_\sigma$ is trivial if $\sigma\notin \Psi$; otherwise it is given by the character:
$$G_{K'}\lra G_{F'_{\wp'}}^\ab\simeq O_{F'_{\wp'}}^\times \subset \bar\BQ_p^\times.$$
Thus $\RT (I)$ is crystalline of weight $-1$ or $0$.

For $H''_{x''}$, let $(x, y)$  be an  $L$-point of $X_1\times Y_1$ with image $x''\in X''_1(L)$.
Consider the  $p$-adic representation $\RT (H''_{x''})$. By Proposition \ref{TTT}, it is  the product $\RT (H_x)\times \RT (I_z)$. 
Both $\RT (I_y)$ and $\RT (H_x)$ are  cryslalline since  both $H_x$ and $I_y$ extend to a $p$-divisible groups over ring of integers by Proposition \ref{H-integrality}, and the above discussion. It follows that $\RT (H_{x''}'')$ is crystalline. It also has weights $0$ and $-1$.
Thus by Breuil--Kisin \cite{Ki1}, $H''_{x''}$ extends to a $p$-divisible group $\CH''_{x''}$ over $O_L$. 

The statement for $H'$ is clear as it is the restriction of $H''$ on $X''$. 
\end{proof}

\subsubsection*{Deformation theory}
Let $L$ be a finite extension of $K'$ and  let $(x, y)$  be an  $L$-point of $X_1\times Y_1$ with image $x''\in X''_1(L)$.
We have covariant Dieudonn\'e modules $\BD(\CH''_{x''})$ over $O_K$, $\BD(\CH _x)$ over $O_{K'}$, $\BD (\CI_y)$ over $O_{K'}$ and their filtrations:
$$0\lra \Omega (\CH''^t_{x''})\lra \BD (\CH ''_{x''})\lra \Omega (\CH''_{x''})^\vee\lra 0.$$
$$0\lra \Omega (\CH^t_x)\lra \BD (\CH_x)\lra \Omega (\CH _x)^\vee\lra 0,$$
$$0\lra \Omega (\CI^t_y)\lra \BD (\CI _y)\lra \Omega (\CI_y)^\vee\lra 0.$$

\begin{prop}\label{D-module-isom}
 There is 
a canonical isomorphism of filtered $O_{E, p}$-modules:
$$\BD (\CH''_{x''})\simeq \BD (\CH_x)\otimes _{O_{E, p}\otimes O_K}\BD (\CI _y).$$
\end{prop}

\begin{proof}
By Kisin \cite[Thm. 1.4.2]{Ki2} for $p\ne 2$ and by Kim \cite{Kim}, Lau \cite{La}, and  Liu \cite{Li} for $p=2$,  for a $p$-divisible group $\CG$ over $O_L$ with $L$  a finite extension of the fraction field of $W(\bar k)$ ($k:=O_\wp/\wp$),
the module $\BD (\CG)$ with its filtration depends canonically on its Tate module 
$\RT (\CG)$ as an object in the category $\Rep _{G_L}^{\criso}$ of integral  crystalline representations of $G_L:=\Gal (\bar L/L)$.
More precisely, let $\frak S=W(\bar k)[[u]]$ be the ring of power series over $W(\bar k)$ with a surjective map $\frak S\lra O_L$ by sending $u$ to a uniformizer $\pi_L$  
of $L$, then 
$$\BD (\CG)=O_L\otimes _{\frak S} \varphi^* {\frak M}\RT (\CG).$$
where  $\frak M$ is  a functor from $\Rep _{G_L}^{\criso}$ to certain category $\mathrm{Mod} _{\frak S}^\varphi$ of modules over non-commutative 
ring $\frak S [\varphi]$, defined in \cite[Theorem 1.2.1]{Ki2}.

Applying this to divisible groups $\CH''_{x''}$, $(\CH_x)_{O_{K'}}$, $\CI _y$ over $O_L=O_{K'}$, and taking care of the isomorphism in the above proposition, we obtain a canonical isomorphism of filtered $O_{E, p}$-modules:
$$\BD (\CH''_{x''})\simeq \BD (\CH_x)\otimes _{O_{E, p}\otimes O_K}\BD (\CI _y).$$
\end{proof}

Now we consider these $p$-divisible groups with actions by $O_{F, p}$. Their cohomology groups are modules over of the $O_K$-algebra 
$O_{F, p}\otimes_{\BZ_p}O_{K}$. The quotient $O_{F, p}\lra O_\wp$ induces a quotient $\tau: O_{F, p}\otimes_{\BZ_p}O_{K}\lra O_K$.
Use this $\tau$ to take quotients of cohomology groups to obtain:
$$0\lra \CW(\CH''^t_{x''})\lra \CM (\CH ''_{x''}) \lra \CW (\CH''_{x''})^\vee\lra 0.$$
$$0\lra \CW(\CH^t_x)\lra \CM (\CH_x) \lra \CW(\CH _x)^\vee\lra 0,$$
$$0\lra \CW(\CI^t_y)\lra \CM  (\CI _y) \lra \CW (\CI_y)^\vee\lra 0.$$
Notice that $\CW(\CI_y)=0$  and   $\CW (\CI_y^t)$ is a free module of rank $1$ over $O_{E, K}:=O_{E, \wp}\otimes _{O_\wp}O_K$.
Thus we have:
\begin{pro} There are canonical  isomorphisms:
$$\CW(\CH''^t_{x''})\simeq \CW(\CH^t_x)\otimes_{O_{E, K}} \CW (\CI _y^t),
\qquad \CW(\CH ''_{x''})\simeq \CW(\CH_x)\otimes_{O_{E, K}}\CW (\CI _y^t)^\vee.$$
\end{pro}

We want to apply these facts to compute the universal deformation space of $\CH ''_{x''}$ as $p$-divisible $O_{E, p}$-module:
\begin{align*}\Hom _{O_{E, p}}(
\Omega (\CH_{x''}^t), \Omega (\CH_{x''})^\vee)=&\Hom _{O_{E, p}}(
\CW(\CH_{x''}^t), \CW(\CH_{x''}'')^\vee)\\
=&\Hom _{O_{E, \wp}}(
\CW(\CH_{x}^t), \CW(\CH_{x})^\vee)\otimes _{O_K}O_{K'}\\
=&\Hom _{O_{\BB, \wp}}(
\CW(\CH_{x}^t), \CW(\CH_{x})^\vee)\otimes O_{K'}\\
=&\omega _{\CX_{1, \wp}, x}^{-1}\otimes O_{K'}\\
=&\omega _{\CX''_{1, \wp'}, x''}^{-1}
.\end{align*}
Here \begin{enumerate}[(1)]
\item the first identity follows from a consideration of types under actions by $O_{E, p}$, 
\item the second identity follows from the above proposition ,
\item the third identity  follows from   a precise  computation,
\item  the fourth identity follows from the Kodaira--Spencer map on $\CH$, 
\item the last one follows from the 
definition.
\end{enumerate}
This shows that the formal completion $\wh \CX''_{1, x''}$ of $\CX''_{1, \wp'}$ at $x''$ is indeed the universal deformation of the $p$-divisible group $\CH''_{x''}$.

Taking determinants of above isomorphism, we obtain the following identity of two  $O_K$-lattices of the module $\omega _{X''_{1, \wp'}, x''}^{-2}$.

\begin{cor}\label{n2n}
$$\CN''_{x''}=\det \CW(\CH''_{x''})\otimes \det \CW (\CH''^t_{x''}).$$
\end{cor}

\subsection{Proof of Theorem \ref{U2O}}

Let $y\in Y$ be any fixed point. Then we have an embedding $X\lra X''$. 
Recall that $P\in X^{T(\BQ)}$ is a fixed CM point by $E$. 
Let $P''\in X''$ be the image of $(P, y)$. Then $P''$ is a point fixed by $T''(\BQ)$.

\begin{lem}
There is  an embedding $X'\lra X''$ such that $P''$ is the image of a $P'$ in $X'$ fixed by $T'(\BQ)$.
\end{lem}
\begin{proof}
We fix one archimedean place $\tau'$ of $F'$ over a place $\tau$ of $F$.  This gives a nearby quaternion algebra $B=B(\tau)$.
We may assume  $P$ is represented by $(z_0, 1)\in \gh\times G(\wh \BQ)$ with $z_0\in \gh$ a fixed point by $E^\times$ in the following 
uniformization:
$$X_\tau (\BC)\simeq G(\BQ)\bs \gh^\pm\times G(\wh \BQ)/\ol{Z(\BQ)}.$$
Similarly, we may assume that $y$ is represented by $1\in \wh E^\times$. Then
$$Y_{\tau'}(\BC)=\ol{E^\times}\bs \wh E^\times.$$
In this way, the image $P''$ of  $(P, y)$ in $X''_\tau (\BC)$ is represented by $(z_0, 1)\in \gh\times G''(\wh \BQ)$:
$$X_{\tau'}''=G''(\BQ)\bs \gh ^\pm \times G''(\wh \BQ)/\ol{Z''(\BQ)}.$$
Thus $P''$ is the image of a point $P'\in X'^{T'}$. 
\end{proof}

Recall that we have fixed a maximal order $O_\BB$ of $\BB_f$ including $O_{\wh E}$, 
which defines maximal compact subgroups $U, U', U''$ of $\BG$, $\BG'$ and $\BG''$,  curves $X_U, X'_{U'}, X''_{U''}$,
and morphisms
$$X_U\lra X''_{U''}, \qquad X_{U'}'\lra X''_{U''}.$$
The image of $P, P', P''$ defines CM points $P_U, P'_{U'}, P''_{U''}$
which is compatible with above morphisms.

By Corollary \ref{Phi2Psi}, it suffices to show that for each nearby pair $(\Phi_1, \Phi_2)$ of  CM types of $E$, 
$$g\cdot h(\Phi_1, \Phi_2)=\frac 12 h_{\ol\CL_U} (P_U)-\frac 14 \log (d_\BB).$$
By Theorem \ref{N2Lfin}, the right hand side is $\frac 14 h_{\ol\CN_U}(P_U)$.

Let  $A_0$ be the corresponding abelian variety represented by $P'_{U'}$ over some finite extension $K$ of $F'(P'_{U'})$.
Then $A_0$ is isogenous to the products of CM abelian varieties  $A_1, A_2$ of CM   types  $\Phi_1, \Phi_2$.
 By Theorem \ref{specialA},
$$h(\Phi_1, \Phi_2)=\frac 12 h(A_0, \tau).$$
Thus we have reduced Theorem \ref{U2O} to the identity 
$$h(A_0, \tau)=\frac 1{2g} h_{\ol\CN_U }(P_U).$$
Since $\frac 1g h_{\ol\CN_U }(P_U)=\frac 1{[F(P_U):\BQ]}\wh\deg (\ol\CN_U|_{\bar P_U})$, it suffices to prove the following result. 

\begin{pro} \label{A-P}There is an isomorphism of hermitian line bundles over $O_K$:
$$\ol \CN(A_0, \tau)\simeq \ol \CN _{P_U}\otimes_{O_{F(P_U)} }O_{K}.$$
\end{pro}
\begin{proof}
Notice that both sides have the restriction $L _{P''_{U''}}^{\otimes 2}\otimes K$ on the  generic fiber of $X'$.
Thus two sides define two  integral and  hermtian structures on $L_{P''}^{\otimes 2}\otimes K$.
Also by Theorem \ref{N2Linf}, they has the same metric. Thus it suffices to show that they define the same lattice at each finite place of $K$.
Let $v$ be a finite place of $K$ with residue characteristic $p$. Let $O_{K, v}^\ur$ be the completion of the maximal unramified extension of $O_{K, v}$.
Then 
$$\Omega (A_0)\otimes O_{K, v}^\ur \simeq \Omega (A_0[p^\infty])\otimes O_{K, v}^\ur.$$
By Corollary \ref{n2n},
$$\CN (A_0, \tau)\otimes O_{K, v}^\ur =\CN _{P''_{U''}}\otimes O_{K, v}^\ur
=\CN_{P_U}\otimes O_{K, v}^\ur.$$
This completes the proof of the proposition.
\end{proof}

\newpage
\part{Quaternionic heights}

The goal of this part is to prove Theorem \ref{quaternion main}.
We also use the notations in our previous work \cite{YZZ}. 
We will make a specific explanation when we come to a setting different from that of \cite{YZZ}.

\section{Pseudo-theta series}

In this section, we introduce the notion of  {\em pseudo-theta series}, an important concept used in the following 
sections. We will first recall the usual theta series defined by Schwartz functions in \cite{YZZ}.
Then we define a pseudo-theta series, which looks like a theta series but is not automorphic.
We will show that it can be approximated by the difference of  two theta series associated to it. 
Finally, we will show that if a sum of pseudo-theta series is automorphic, then these pseudo-theta series can be actually 
replaced by the difference of the theta series associated to them and we get some extra identities between these theta series.

\subsection{Schwartz functions and theta series}

We first recall the notion of Schwartz functions and theta series in \cite{YZZ}, which is a variant of the standard notions. 

Let $F$ be a totally real number field, and $\BA$  the adele ring of $F$. 
Let $(V,q)$ be a positive definite quadratic space over $\RR$.
Let 
$$
\OCS(V(\adele)\times \adele^{\times})
=\otimes_v \OCS(V(F_v)\times F_v^{\times})
$$
be the space of Schwartz functions introduced in \cite[\S 4.1]{YZZ}.
We recall it in the following. 

If $v$ is non-archimedean, then
$\OCS(V(F_v)\times F_v^{\times})$ is the usual space of locally constant and compactly supported functions. 

If $v$ is archimedean, then $F_v=\RR$ and then 
$\overline\CS (V(F_v)\times \RR\cross)$ consists of
functions on $V(F_v)\times \RR^\times$ of the form 
$$\phi_v(x,u)=\left(P_1(uq(x))+\sgn (u)P_2(uq(x))\right)e^{-2\pi |u|q(x)}$$
with polynomials $P_i$ of complex coefficients. Here $\sgn (u)=u/|u|$ denotes
the sign of $u\in \RR\cross$.
The \textit{standard Schwartz function} $\phi_v\in \overline \CS(V(F_v)\times
\RR\cross)$ is the Gaussian function
$$
\phi_v(x,u)= e^{-2\pi uq(x)}\ 1_{\BR_+}(u).
$$
Here $1_{\BR_+}$ is the characteristic function of the set $\BR_+$ of positive
real numbers.
In this paper, $\phi$ is always the standard Gaussian function at archimedean places.

Assume that $\dim V$ is even in the following, which is always satisfied in our application. 
In \cite[\S 2.1.3]{YZZ}, the Weil representation on the usual space $\CS(V(\BA))$ is extended to an action of the similitude groups on $\OCS(V(\adele)\times \adele^{\times})$. 
This gives a representation of $\GL_2(\BA)\times \GO(V(\BA))$ on  
$\OCS(V(\adele)\times \adele^{\times})$. 
This extension is originally from Waldspurger \cite{Wa}. 

Take any $\phi\in  \overline \CS (V(\adele)\times \BA^\times)$.
There is the partial theta series
$$
\theta(g,u,\phi)= \sum_{x\in V} r(g)\phi(x,u), \quad g\in\gla, \ u\in \across.
$$
If $u\in F\cross$, it is invariant under the left action of $\SL_2(F)$ on $g$.
To get an automorphic form on $\gla$, we need a summation on $u$.  

There is an open compact subgroup $K\subset \GO(\adele_f)$ such that $\phi_f$ is
invariant under the action of $K$ by the Weil representation.
Denote $\mu_K=F\cross \cap K$. 
Then $\mu_K$ is a subgroup of the unit group $O_F\cross$, and thus
is a finitely generated abelian group. 
Define a theta function by  
\begin{equation*}
\theta(g, \phi)_K
=\sum_{u\in \mu_K^2\backslash F\cross} \theta(g,u, \phi)
= \sum_{u\in \mu_K^2\backslash F\cross} \sum_{x\in V} r(g)\phi(x,u),  \quad g\in\gla.
\end{equation*}
The summation is well-defined and absolutely convergent.
The result $\theta(g, \phi)_K$ is an automorphic form on $g\in \gl(\adele)$, and
$\theta(g, r(h)\phi)_K$ is an automorphic form on $(g,h)\in \gl(\adele)\times
\GO(V(\adele))$.
Furthermore, if $\phi_\infty$ is standard, then
$\theta(g, \phi)_K$
is holomorphic of parallel weight $\frac 12 \dim V$.

By choosing fundamental domains, we can rewrite the sum as
$$ \theta(g, \phi)_K
= \sum_{u\in \mu_K^2\backslash F\cross} r(g)\phi(0,u)
+w_K \sum_{(x,u)\in \mu_K \backslash ((V-\{0\})\times F\cross)}
r(g)\phi(x,u) .$$
Here the natural action of $\mu_K$ on $V\times F\cross$ is just $\alpha\circ
(x,u)\mapsto (\alpha x,\alpha^{-2}u)$.
The summation over $u$ is
well-defined since
$\phi(\alpha x, \alpha^{-2} u)=r(\alpha^{-1})\phi(x,u)=\phi(x,u)$
for any $\alpha \in\mu_K$.
The factor $w_K=|\{1,-1\}\cap K|\in \{1,2\}$.
See \cite[\S 2.1.3]{YZZ} for more details.

\subsection{Pseudo-theta series} \label{section pseudo}

Now we introduce pseudo-theta series. Let $V$ be a positive definite
quadratic space over $F$, and $V_0\subset V_1 \subset V$ be two
subspaces over $F$ with induced quadratic forms. All spaces are assumed to be even-dimensional. We allow $V_0$ to be the empty set $\emptyset$, which is not a subspace in the usual sense. Let $S$ be a finite set of non-archimedean places of $F$, and
$\phi^S \in \overline\CS(V(\adele^S)\times \adele^{S \times})$ be a Schwartz
function with standard infinite components.

A \textit{pseudo-theta series} is a series of the form
$$A^{(S)}_{\phi'}(g)=\sum_{u\in \mu^2\backslash F\cross} \sum_{x\in V_1-V_0} \phi_S'(g,x,u) r_{_V}(g)\phi^S(x,u), \ \quad g\in \gl(\adele).$$
We explain the notations as follows:
\begin{itemize}
\item The Weil representation $r_{_V}$ is not attached to  the space $V_1$
but to the space $V$;
\item $\phi_S'(g,x,u)=\prod_{v\in S} \phi'_v(g_v,x_v,u_v)$ as local product;
\item For each $v\in S$, the function
 $$\phi'_v: \gl(F_v)\times (V_1-V_0)(F_v)  \times F_v\cross \rightarrow \BC$$
 is locally constant. And it is smooth in the sense that there is an open compact subgroup $K_v$ of $\gl(F_v)$ such that
 $$\phi'_v(g\kappa,x,u)=\phi'_v(g,x,u),
 \quad \forall (g,x,u)\in \gl(F_v)\times (V_1-V_0)(F_v)\times F_v\cross,\  \kappa\in K_v.$$
\item $\mu$ is a subgroup of $O_F\cross$ with finite index such that $\phi^S(x,u)$ and $\phi'_S(g,x,u)$ are invariant under the action $\alpha: (x,u)\mapsto (\alpha x, \alpha^{-2} u)$ for any $\alpha\in \mu$. This condition makes the summation well-defined.
\item For any $v\in S$ and $g\in \gl(F_v)$, the support of $\phi'_v(g,\cdot,\cdot)$ in $(V_1-V_0)(F_v)\times F_v\cross$ is bounded. This condition makes the sum convergent.

\end{itemize}

The pseudo-theta series $A^{(S)}$ sitting on the triple $V_0\subset
V_1 \subset V$ is called \textit{non-degenerate} if $V_1=V$, and is
called \textit{non-truncated} if $V_0$ is empty. It is called
\textit{non-singular} if for each $v\in S$, the local component
$\phi'_v(1,x,u)$ can be extended to a Schwartz function on
$V_1(F_v)\times F_v\cross$.

Assume that $A^{(S)}_{\phi'}$ is non-singular. Then there are two
usual theta series associated to $A^{(S)}$. View
$\phi'_v(1,\cdot,\cdot)$ as a Schwartz function on $V_1(F_v)\times
F_v\cross$ for each $v\in S$, and $\phi_w$ as a Schwartz function on
$V_1(F_w)\times F_w\cross$ for each $w\notin S$. Then the theta
series
$$\theta_{A,1}(g)=\sum_{u\in \mu^2\backslash F\cross} \sum_{x\in V_1} r_{_{V_1}}(g)\phi_S'(1,x,u) r_{_{V_1}}(g)\phi^S(x,u)$$
is called \textit{the outer theta series associated to}
$A^{(S)}_{\phi'}$. 
Note that the Weil representation $r_{V_1}$ is based on the quadratic space $V_1$. 
Replacing the space $V_1$ by $V_0$, we get the
theta series
$$\theta_{A,0}(g)=\sum_{u\in \mu^2\backslash F\cross} \sum_{x\in V_0} r_{_{V_0}}(g)\phi_S'(1,x,u) r_{_{V_0}}(g)\phi^S(x,u).$$
We call it \textit{the inner theta series associated to}
$A^{(S)}_{\phi'}$. We set $\theta_{A,0}=0$ if $V_0$ is empty.

We introduce these theta series because the difference between
$\theta_{A,1}$ and $\theta_{A,0}$ somehow approximates $A^{(S)}$. It
will be discussed as follows.

\subsubsection*{Approximation by induced theta series} 

We start with two invariants of $\gla$ defined in terms of the Iwasawa decomposition. For $g\in \gl(\adele)$, we define $\delta(g)=\prod_v \delta_v(g_v)$ and $\rho_\infty(g)=\prod_{v|\infty} \rho_v(g_v)$. Here the local invariants are defined as follows. 

Denote by $P$ the algebraic group over $\QQ$ of upper triangular matrices.
For any place $v$, the character $\delta_v: P(F_v)\rightarrow \RR\cross$ defined by
$$\delta_v: \matrixx{a}{b}{}{d} \longmapsto \left| \frac ad \right| ^{\frac 12}$$ 
extends to a function $\delta_v: \gl(F_v)\rightarrow
\RR\cross$ by the Iwasawa decomposition. 

If $v$ is a real place, we define a function $\rho_v:
\gl(F_v)\rightarrow \BC$ by $\rho_v(g)=e^{i\theta}$ if
$$g=  \matrixx{a}{b}{}{d}  \matrixx{\cos \theta}{\sin \theta}{-\sin \theta}{\cos \theta} $$
is in the form of the Iwasawa decomposition, where we require $a>0$ so that the decomposition is unique.

Resume the notation in the last subsection. Now we consider the relation between the non-singular
pseudo-theta series $A^{(S)}_{\phi'}$ and its associated theta
series $\theta_{A,1}$ and $\theta_{A,0}$.

We first consider the non-truncated case. Then $V_0$ is empty, and
$$A^{(S)}_{\phi'}(g)=\sum_{u\in \mu^2\backslash F\cross} \sum_{x\in V_1} \phi_S'(g,x,u) r_{_V}(g)\phi^S(x,u).$$
Obviously we have $A^{(S)}_{\phi'}(1)=\theta_{A,1}(1)$, but of
course we can get more.

A simple computation using Iwasawa decomposition asserts that, if
$\phi_w$ is the standard Schwartz function on $V(F_w)\times
F_w\cross$, then for any $g\in \gl(F_v)$ and $(x,u)\in V_1(F_w)\times F_w\cross$,
$$r_{_{V}}(g)\phi_w(x,u)=\begin{cases}
\delta_w(g)^{\frac{d-d_1}{2}} r_{_{V_1}}(g)\phi_w(x,u)  &\mbox{ if } w\nmid \infty;\\
\rho_w(g)^{\frac{d-d_1}{2}}\delta_w(g)^{\frac{d-d_1}{2}}
r_{_{V_1}}(g)\phi_w(x,u)  &\mbox{ if } w\mid \infty.
\end{cases}$$
Here we write $d=\dim V$ and $d_1=\dim V_1$.

This result implies that,
$$A^{(S)}_{\phi'}(g)=\rho_\infty(g)^{\frac{d-d_1}{2}}\delta(g)^{\frac{d-d_1}{2}} \theta_{A,1}(g), \quad \
\forall \ g\in 1_{S'}\gl(\adele^{S'}). $$ Here $S'$ is a finite set
consisting non-archimedean places $v$ such that $v\in S$ or $\phi_v$ is not
standard.

Now we consider a general non-singular pseudo-theta series
$$A^{(S)}_{\phi'}(g)=\sum_{u\in \mu^2\backslash F\cross} \sum_{x\in V_1-V_0} \phi_S'(g,x,u) r_{_V}(g)\phi^S(x,u).$$
We have to compare it with the difference between the same theta
series
$$\theta_{A,1}(g)=\sum_{u\in \mu^2\backslash F\cross} \sum_{x\in V_1} r_{_{V_1}}(g)\phi_S'(1,x,u) r_{_{V_1}}(g)\phi^S(x,u)$$
and the non-truncated pseudo-theta series
$$B^{(S)}_{\phi'}(g)=\sum_{u\in \mu^2\backslash F\cross} \sum_{x\in V_0} r_{_{V_1}}(g)\phi_S'(1,x,u) r_{_{V_1}}(g)\phi^S(x,u).$$
Note that $B^{(S)}$ is just a part of $\theta_{A,1}$, where summation
is taken over the whole $V_0$ but the representation is taken over
$V_1$. By what we discussed above, we should compare $B^{(S)}$ with
the associated theta series
$$\theta_{B,0}(g)=\sum_{u\in \mu^2\backslash F\cross} \sum_{x\in V_0} r_{_{V_0}}(g)\phi_S'(1,x,u) r_{_{V_0}}(g)\phi^S(x,u).$$
But this is exactly the same as $\theta_{A,0}$. By the same
argument, there exists a finite set $S'$ of non-archimedean places such that
\begin{eqnarray*}
A^{(S)}_{\phi'}(g)
&=&\rho_\infty(g)^{\frac{d-d_1}{2}}\delta(g)^{\frac{d-d_1}{2}} (\theta_{A,1}(g)- B^{(S)}_{\phi'}(g)), \quad \ \forall \ g\in 1_{S'}\gl(\adele^{S'});\\
B^{(S)}_{\phi'}(g)
&=&\rho_\infty(g)^{\frac{d_1-d_0}{2}}\delta(g)^{\frac{d_1-d_0}{2}}
\theta_{A,0}(g), \quad \ \forall \ g\in 1_{S'}\gl(\adele^{S'}).
\end{eqnarray*}
Our conclusion is that for any $g\in 1_{S'}\gl(\adele^{S'})$,
\begin{eqnarray}
A^{(S)}_{\phi'}(g)=\rho_\infty(g)^{\frac{d-d_1}{2}}\delta(g)^{\frac{d-d_1}{2}}
\theta_{A,1}(g) -
\rho_\infty(g)^{\frac{d-d_0}{2}}\delta(g)^{\frac{d-d_0}{2}}
\theta_{A,0}(g).  \label{pseudo difference}
\end{eqnarray}
By the smoothness condition of pseudo-theta series, there exists an
open compact subgroup $K_{S'}$ of $\gl(F_{S'})$ such that the above
identity is actually true for any $g\in K_{S'}\gl(\adele^{S'})$.

\subsection{Key lemma}

Now we can state our main result for this subject.
\begin{lem}\label{pseudo}
Let $\{A_\ell^{(S_\ell)}\}_\ell$ be a finite set of non-singular
pseudo-theta series sitting on vector spaces $V_{\ell,0}\subset
V_{\ell,1}\subset V_{\ell}$. Assume that the sum $\sum_\ell
A_\ell^{(S_\ell)}(g)$ is automorphic for $g\in\gl(\adele)$. Then
\begin{enumerate}[(1)]

\item[(1)] $\displaystyle \sum_\ell A_\ell^{(S_\ell)}=\sum_{\ell\in L_{0,1}} \theta_{A_\ell,1}$,
\item[(2)] $\displaystyle \sum_{\ell\in L_{k,1}} \theta_{A_\ell,1}-\sum_{\ell\in L_{k,0}} \theta_{A_\ell,0}=0,\quad \forall k\in \ZZ_{>0}$.
\end{enumerate}
Here $L_{k,1}$ is the set of $\ell$ such that $\dim V_{\ell}-\dim
V_{\ell,1}=k$, and $L_{k,0}$ is the set of $\ell$ such that $\dim
V_{\ell}-\dim V_{\ell,0}=k$. In particular, $L_{0,1}$ is the set of
$\ell$ such that $V_{\ell,1}=V_{\ell}$.
\end{lem}

\begin{proof}
Denote $f=\sum_\ell A_\ell^{(S_\ell)}$. In the equation $f-\sum_\ell
A_\ell^{(S_\ell)}=0$, replace each $A_\ell^{(S_\ell)}$ by its
corresponding combinations of theta series on the right-hand side of
equation (\ref{pseudo difference}). After recollecting these theta
series according to the powers of $\rho_\infty(g) \delta(g)$, we
end up with an equation of the following form:
\begin{eqnarray}
\sum_{k=0}^n \rho_\infty(g)^{k}\delta(g)^{k} f_k(g)=0, \quad \
\forall g\in K_{S}\gl(\adele^{S}).  \label{vanishing}
\end{eqnarray}

Here $S$ is some finite set of non-archimedean places, $K_S$ is an open
compact subgroup of $\gl(F_S)$, and $f_0, f_1,\cdots, f_n$ are some
automorphic forms on $\gl(\adele)$ coming from combinations of $f$
and theta series. In particular, $f_0=f-\sum_{\ell\in L_{0,1}}
\theta_{A_\ell,1}$. We will show that $f_0=f_1=\cdots=f_n=0$
identically, which is exactly the result of (1) and (2).

It suffices to show $f_k(g_0)=0$ for all $g_0\in \gl(\adele_f^{S})$,
since $\gl(F)\gl(\adele_f^{S})$ is dense in $\gl(\adele)$. Fix
$g_0\in \gl(\adele_f^{S})$. For any $g \in \gl(F)\cap
K_S\gl(\adele^S)$, we have
$$\sum_{k=0}^n \rho_\infty(gg_0)^{k}\delta(gg_0)^{k} f_k(gg_0)=0,$$
and thus
$$\sum_{k=0}^n \rho_\infty(g)^{k}\delta(gg_0)^{k} f_k(g_0)=0$$
by the modularity.

These are viewed as linear equations of $f_0(g_0), f_1(g_0), \cdots,
f_n(g_0)$. To show that the solutions are zero, we only need to find
many $g$ to get plenty of independent equations. We first find some
special $g$ to simplify the equation.

The intersection $K_S\gl(\adele^S) \cap g_0 \gl(\widehat
O_F)g_0^{-1}$ is still an open compact subgroup of $\gl(\adele)$.
For any $g\in \gl(F)\cap (K_S\gl(\adele^S) \cap g_0 \gl(\widehat
O_F)g_0^{-1})$, we have
$$gg_0=g_0 \cdot g_0^{-1} g g_0\in g_0\gl(\widehat O_F).$$
Then $\delta_f(gg_0)=\delta_f(g_0)$, and our linear equation
simplifies as
$$\sum_{k=0}^n \rho_\infty(g)^{k}\delta_\infty(g)^{k} \delta_f(g_0)^{k}f_k(g_0)=0.$$

To be more explicit, consider $g_N=\matrixx{1}{}{N}{1}$ for any
$N\in \ZZ$. Then we know that $g_N\in \gl(F)\cap (K_S\gl(\adele^S)
\cap g_0 \gl(\widehat O_F)g_0^{-1})$ when $N$ is divisible by enough
integers. Explicit computation gives
$$\rho_\infty(g_N)\delta_\infty(g_N)= (1+iN)^{-n}$$
where $n=[F:\QQ]$. Then we have
$$\sum_{k=0}^n (1+iN)^{-nk} \delta_f(g_0)^{k} f_k(g_0)=0.$$
Any $n+1$ different values of $N$ imply that all $f_k(g_0)=0$ by Van
der Mond's determinant.
\end{proof}

\section{Derivative series}

The goal of this section is to study the holomorphic projection of the derivative of some mixed Eisenstein--theta series.
We will first review the construction of the series  $\Pr I'(0,g,\phi)$ treated in  \cite[Chapter 6]{YZZ}, the analytic ingredient for 
proving Theorem \ref{quaternion main}.
Then we compute the series under  some assumptions  of Schwartz functions. The final formula contains a term $L'(0, \eta)/L(0, \eta)$
which is a main ingredient of our main theorem in the paper. 
In \cite{YZZ}, this constant terms was killed under  some stronger assumptions  of Schwartz functions.

\subsection{Derivative series}

Let $F$ be a totally real field, and $E$ be a totally imaginary quadratic extension of $F$. Denote by $\BA$ the ring of adeles of $F$.
Let $\BB$ be a totally definite incoherent quaternion algebra over 
$\BA=\BA_F$ with an embedding $E_\BA\to \BB$ of $\BA$-algebras.

Fix a Schwartz function $\phi\in \ol\CS (\BB\times \BA^\times)$ invariant under $U\times U$ for some open compact subgroup $U$ of $\bfcross$. 
Start with the mixed theta-Eisenstein series
\begin{equation*} 
I(s, g, \phi)_U
 = \sumu \sum_{\gamma \in P^1(F)\bs \SL_2(F)}
\delta (\gamma g)^s \sum_{x_1\in E} r(\gamma g)\phi (x_1, u).
\end{equation*}
It was first introduced in \cite[\S 5.1.1]{YZZ}. 

The derivative series $\Pr I'(0,g,\phi)$ is the holomorphic projection of the derivative $I'(0,g,\phi)$ of $I(s,g,\phi)$. 
It has a decomposition into local components as follows.

\subsubsection*{Eisenstein series of weight one}
To illustrate the idea, we first assume that $\phi=\phi_1\otimes \phi_2$ as in 
\cite[\S 6.1]{YZZ}. 
Then 
\begin{equation*}
I(s, g, \phi)_U=\sum_{u\in\mu_U^2\bs F\cross} \theta(g, u, \phi_1)\ E(s,g, u,
\phi_2),
\end{equation*}
where for any $g\in\gla$, the theta series and the Eisenstein series are given
by
\begin{eqnarray*}
\theta(g, u, \phi_1)&=& \sum_{x_1\in E} r(g)\phi_1(x_1, u), \\
E(s,g, u, \phi_2)&=& \sum_{\gamma \in P^1(F)\bs \SL_2(F)}
\delta (\gamma g)^s r(\gamma g)\phi_2(0, u).
\end{eqnarray*}

The Eisenstein series has the standard Fourier expansion
$$E(s,g,u,\phi_2)=\delta(g)^s  r(g)\phi_2(0,u)+\sum_{a \in F} W_a(s,
g,u,\phi_2).$$
Here the Whittaker function for $a\in F, \ u\in F\cross$ is given by
\begin{eqnarray*}
W_a(s,g,u,\phi_2) &=& \int_{\adele} \delta(wn(b)g)^s \ r(wn(b)g)\phi_2(0,u)
\psi(-ab) db.
\end{eqnarray*}
We also have the constant term 
$$
E_0(s,g,u,\phi_2)=\delta(g)^s  r(g)\phi_2(0,u)+ W_0(s, g,u).
$$
For each place $v$ of $F$, we also introduce the local Whittaker function for
$a\in F_v, \ u\in F_v\cross$ by
\begin{eqnarray*}
W_{a,v}(s,g,u,\phi_{2,v}) &=& \int_{F_v} \delta(wn(b)g)^s \
r(wn(b)g)\phi_{2,v}(0,u)
\psi_v(-ab) db.
\end{eqnarray*}

For $a\in F_v\cross$, denote
\begin{eqnarray*}
W_{a,v}^{\circ}(s, g,u)
&=& \gamma_{u,v}^{-1}  W_{a,v}(s, g,u),
\end{eqnarray*}
where $\gamma_{u,v}$ is the Weil index of $(E_v\jv, uq)$. 
Normalize the intertwining part by
\begin{eqnarray*}
W_{0,v}^{\circ}(s, g,u,\phi_{2,v})
&=& \gamma_{u,v}^{-1}\frac{L(s+1,\eta_v)}{L(s,\eta_v)}
|D_v|^{-\frac{1}{2}}|d_v|^{-\frac{1}{2}} W_{0,v}(s, g,u,\phi_{2,v}).
\end{eqnarray*}
In the following we will suppress the dependence of the series on $\phi, \phi_1,
\phi_2$ and $U$.

\subsubsection*{Decomposition of non-constant part}
It is easy to have a decomposition 
$$
E'(0,g,u,\phi_2)=E_0'(0,g,u,\phi_2)-\sum_v \sum_{a\in F^\times} {W_{a,v}}'(0,g,u,\phi_2)
W_{a}^{v}(0,g,u, \phi_2),
$$
according to where the derivative is take in the Fourier expansion. 
This gives a decomposition of $I'(0,g)$. 
Eventually, \cite[\S 6.1.2]{YZZ} converts the decomposition into
\begin{align*}
I'(0,g) = -\sum_{v\ \mathrm{nonsplit}}I'(0,g)(v)+ \sum_{u \in \mu_U^2\bs F\cross}
\theta(g,u)E_0'(0,g,u), 
\end{align*}
where for any place $v$  nonsplit in $E$,
$$I'(0,g,\phi)(v) = 2 \barint_{C_U}
\CK^{(v)}_{\phi}(g,(t,t))dt.$$
Here 
$$
C_U=E^\times \bs E^\times(\BA_f)/ E^\times(\BA_f)\cap U 
$$
is a finite group and the integration is just the usual average over this finite group. 
The series
\begin{eqnarray*}
 \CK^{(v)}_{\phi}(g,(t_1,t_2))
&=&\sum_{u\in \mu_U^2\backslash F\cross} \sum_{y\in B(v)-E}
k_{r(t_1,t_2)\phi_v}(g,y,u)  r(g,(t_1,t_2)) \phi^v(y,u)
\end{eqnarray*}
is a pseudo-theta series. 
In the case $\phi_v= \phi_{1,v}\otimes \phi_{2,v}$ under the orthogonal
decomposition,
it is given by  
\begin{equation*}
k_{\phi_v}(g,y,u)= \frac{L(1,\eta_v)}{\mathrm{vol}(E_v^1)} r(g)
\phi_{1,v}(y_1,u) {W_{uq(y_2),v}^\circ}'(0,g,u,\phi_{2,v}), \quad y_2\neq 0.
\end{equation*}
Here $k_{\phi_v}(g,y,u)$ is linear in $\phi_v$, and the result extends by linearity to general $\phi$ (which are not of the form $\phi_1\otimes\phi_2$).

In \cite{YZZ}, Assumption 5.3 was put to kill the minor term 
$E_0'(0,g,u)$. In this paper, however, we will not  impose this assumption, since 
$E_0'(0,g,u)$ gives terms matching the Faltings height from the arithmetic side. 
In the following, we give a little computation about it.

\subsubsection*{Decomposition of constant term}

Now we treat the derivative of the constant term
$$
E_0(s,g,u,\phi_2)=\delta(g)^s  r(g)\phi_2(0,u)+ W_0(s, g,u).
$$
It was actually computed in the proof of \cite[Proposition 6.7]{YZZ} (before applying the degeneracy assumption). 

In fact, by definition 
\begin{align*}
W_0(s,g,u)
=&-\frac{L(s,\eta)}{L(s+1,\eta)} W_{0}^{\circ}(s,
g,u)\prod_v|D_v|^{\frac{1}{2}}|d_v|^{\frac{1}{2}}\\
=& -\frac{L(s,\eta)/L(0,\eta)}{L(s+1,\eta)/L(1,\eta)}
\prod_v W_{0,v}^{\circ}(s,g,u).
\end{align*}
We take the normalization $W_{0,v}^{\circ}(s,g,u)$ because
$$W_{0,v}^{\circ}(0, g,u)= r(g)\phi_{2,v}(0,u)$$
for all $v$, and  
$$W_{0,v}^{\circ}(s,g,u)= \delta_v(g)^{-s}  r(g)\phi_{2,v}(0,u)$$
for almost all $v$. See \cite[Proposition 6.1]{YZZ}.

So the expression gives the analytic continuation of $W_0(s,g,u)$.
Taking derivative from it, we obtain
\begin{multline*}
W_0'(0,g,u) = -\der\left(\log \frac{L(s,\eta)}{L(s+1,\eta)}\right)
r(g)\phi_2(0,u)
-  \sum_v W_{0,v}^{\circ}\, '(0,g,u) r(g)\phi_2^v(0,u).
\end{multline*}

In summary, we have
\begin{align*}
I'(0,g,\phi) = & -\sum_{v\ \mathrm{nonsplit}}I'(0,g,\phi)(v)
-c_0 \sumu \sum_{y\in E} r(g)\phi(y,u) \\
&  -\sum_v \sumu\sum_{y\in E}
c_{\phi_v}(g,y,u)\, r(g)\phi^v(y,u)
+2\log \delta(g) \sum_{u \in \mu_K^2\backslash F\cross, y\in E} r(g)\phi(y,u),
\end{align*}
where the constant 
$$ c_0=\der\left(\log\frac{L(s,\eta)}{L(s+1,\eta)}\right),$$
and 
$$c_{\phi_v}(g,y,u)=r_E(g)\phi_{1,v}(y,u) W_{0,v}^{\circ}\, '(0,g,u) 
+ \log \delta(g_v)r(g)\phi_v(y,u).$$
The term 
$$I'(0,g,\phi)(v) = 2 \barint_{C_U}
\CK^{(v)}_{\phi}(g,(t,t))dt$$
is as before.  
Both sums over $v$ have only finitely many non-zero terms. 

By the functional equation 
$$
L(1-s,\eta)=|d_E/d_F|^{s-\frac12} L(s,\eta),
$$
we obtain
$$ c_0=2\frac{L'(0,\eta)}{L(0,\eta)} +\log|d_E/d_F|.$$
Note that here $L(s,\eta)$ is the completed L-function with gamma factors.

The decomposition holds for $\phi=\phi_1\otimes\phi_2$, but it extends to 
any $\phi\in\ol\CS (\BB\times \BA^\times)$ by linearity. 
In other words, $k_{\phi_v}(g,y,u)$ and $c_{\phi_v}(g,y,u)$
are defined by linearity. We will see that we can actually have coherent integral expressions for them.

\subsubsection*{Holomorphic projection}

As in \cite[\S 6.4-6.5]{YZZ}, we are going to consider the holomorphic projection of $I'(0,g,\phi)$. 

Denote by $\CA(\gl(\adele), \omega)$ the space of
automorphic forms of central character $\omega$, and by
$\CA_0^{(2)}(\gl(\adele), \omega)$ the subspace of holomorphic
cusp forms of parallel weight two.
The holomorphic projection operator 
$$\pr: \CA(\gl(\adele), \omega)\longrightarrow \CA_0^{(2)}(\gl(\adele), \omega)$$
is just the orthogonal projection with respect to the Petersson inner 
product. 

Consider the action of the center $\BA^\times$ on $I'(0,g,\phi)$ by 
$$z:I'(0,g,\phi)\longmapsto I'(0,zg,\phi).$$ 
The action factorizes though the finite group $F^\times\bs \BA_f^\times/ U\cap \BA_f^\times$. It follows that we can decompose $I'(0,g,\phi)$ into a finite sum according to characters of this finite group. 
In other words, 
$$
I'(0,g,\phi) = \sum_\omega I'(0,g,\phi)_{\omega}, \qquad  I'(0,g,\phi)_\omega\in \CA(\gl(\adele), \omega),
$$
where the direct sum is over the finite group of characters $\omega: F^\times\bs \BA_f^\times/ U\cap \BA_f^\times\to \BC^\times$.
Hence, the holomorphic projection $\pr I'(0,g,\phi)$ is still a well-defined 
holomorphic cusp form of parallel weight two in $g\in\gla$. 

We can apply the formula in \cite[Proposition 6.12]{YZZ} to compute $\pr I'(0,g,\phi)$. Note that the formula takes the same form in all central characters, and thus can be applied directly to $\pr I'(0,g,\phi)$, if it satisfies the growth condition of the proposition. 
For the growth condition, we make the following assumption. 

\begin{assumption} \label{assumption3}
Fix a set $S_2$ consisting of $2$ non-archimedean places of $F$ which are split in $E$ and unramified over $\BQ$. Assume that for each $v\in S_2$, the open compact subgroup $U_v$ is maximal, and 
\begin{equation*}
r(g)\phi_v(0,u)=0, \quad \forall\ g\in \gl(F_v), \ u\in F_v\cross.
\end{equation*}
\end{assumption}

This assumption is exactly \cite[Assumption 5.4]{YZZ}. Under the assumption, $\pr I'(0,g,\phi)$ satisfies the growth condition of the formula for holomorphic projection. The proof is similar to that in \cite[Proposition 6.14]{YZZ}. 
Alternatively, one can expression $I'(0,g,\phi)$ as a finite sum of $I'(0,g,\chi,\phi)$ for different $\chi$.

Finally, we have the following conclusion. 
\begin{thm} \label{derivative series}
Assume that $\phi$ is standard at infinity and that Assumption \ref{assumption3} holds. Then 
\begin{align*}
\pr I'(0,g,\phi)_U
=& -\sum_{v|\infty}\overline {I'}(0,g,\phi)(v)-\sum_{v\nmid\infty \
\nonsplit}I'(0,g,\phi)(v)\\
& -c_1 \sumu \sum_{y\in E^\times} r(g)\phi(y,u) 
 -\sum_{v\nmid\infty} \sumu\sum_{y\in E^\times}
c_{\phi_v}(g,y,u)\, r(g)\phi^v(y,u)\\
&+ \sumu \sum_{y\in E^\times} (2\log \delta_f(g_f)+ \log|uq(y)|_f)\ r(g)\phi(y,u).
\end{align*}
The right-hand side is explained in the following. 
\begin{itemize}
\item[(1)]
For any archimedean $v$,
\begin{align*}
\overline {I'}(0,g,\phi)(v) &= 2 \barint_{C_U}
 \overline \CK^{(v)}_{\phi}(g,(t,t))dt, \\
\overline \CK^{(v)}_{\phi}(g,(t_1,t_2)) &= w_U \sum_{a\in F\cross}
\quasilim
\sum_{y\in \mu_U\backslash (B(v)_+\cross -E\cross)}  r(g,(t_1,t_2)) \phi(y)_a\
k_{v, s}(y), \\
k_{v, s}(y) &= \frac{\Gamma(s+1)}{2(4\pi)^{s}}
\int_1^{\infty} \frac{1}{t(1-\lambda(y)t)^{s+1}}  dt,
\end{align*}
where $\lambda(y)=q(y_2)/q(y)$ is viewed as an element of $F_v$.
\item[(2)]
For any non-archimedean $v$ which is nonsplit in $E$,
\begin{align*}
I'(0,g,\phi)(v) &= 2 \barint_{C_U} \CK^{(v)}_{\phi}(g,(t,t))dt, \\
 \CK^{(v)}_{\phi}(g,(t_1,t_2))
&= \sum_{u\in \mu_U^2\backslash F\cross} \sum_{y\in B(v)-E}
k_{r(t_1,t_2)\phi_v}(g,y,u)  r(g,(t_1,t_2)) \phi^v(y,u), \\
k_{\phi_v}(g,y,u)&= \frac{L(1,\eta_v)}{\mathrm{vol}(E_v^1)} r(g)
\phi_{1,v}(y_1,u) {W_{uq(y_2),v}^\circ}'(0,g,u,\phi_{2,v}), \quad y_2\neq 0.
\end{align*}
Here the last identity holds under the relation $\phi_v= \phi_{1,v}\otimes \phi_{2,v}$, and the definition extends by linearity to general $\phi_v$.
\item[(3)]
The constant 
$$ c_1
=2\frac{L'_f(0,\eta)}{L_f(0,\eta)} +\log|d_E/d_F|.$$
\item[(4)] 
Under the relation $\phi_v= \phi_{1,v}\otimes \phi_{2,v}$,
$$c_{\phi_v}(g,y,u)=r_E(g)\phi_{1,v}(y,u) W_{0,v}^{\circ}\ '(0,g,u) 
+ \log \delta(g_v)r(g)\phi_v(y,u).$$
The definition extends by linearity to general $\phi_v$.
\end{itemize}
\end{thm}

\begin{proof}
Apply the formula of \cite[Proposition 6.12]{YZZ} to each term of
\begin{align*}
I'(0,g,\phi) = & -\sum_{v\ \mathrm{nonsplit}}I'(0,g,\phi)(v)
-c_0 \sumu \sum_{y\in E} r(g)\phi(y,u) \\
&  -\sum_v \sumu\sum_{y\in E}
c_{\phi_v}(g,y,u)\, r(g)\phi^v(y,u) \\
& +2\log \delta(g) \sum_{u \in \mu_K^2\backslash F\cross}\sum_{y\in E} r(g)\phi(y,u).
\end{align*}
Denote by $\Pr'$ the image of each term. 
Note that the holomorphic projection of $I'(0,g,\phi)(v)$ is already computed in 
\cite[Proposition 6.15]{YZZ}. 
Furthermore, if $v$ is real, we have $c_{\phi_v}(g,y,u)=0$ by Lemma \ref{derivative of intertwining}. 

Note that $\Pr'$ does not change $I'(0,g,\phi)(v)$ for non-archimedean $v$ since it is already holomorphic of parallel weight two at infinite. 
Similarly, we have
\begin{align*}
\pr' \left(  \sumu\sum_{y\in E} r(g)
\phi(y,u) \right)
&=  \sumu\sum_{y\in E^\times} r(g)\phi(y,u),\\
\pr' \left(  \sumu\sum_{y\in E}
c_{\phi_v}(g,y,u) r(g^v)\phi^v(y,u) \right) &=
 \sumu\sum_{y\in E^\times} c_{\phi_v}(g,y,u)
r(g^v)\phi^v(y,u), \quad v\nmid\infty.
\end{align*}
The only changes are to remove the contributions of $y=0$, because the results do not have constant terms. 

It remains to take care of
$$\log \delta(g)  \sumu\sum_{y\in E^\times}
r(g)\phi(y,u)
=\frac{1}{w_U}\log \delta(g)  \sum_{(y,u) \in \mu_U \backslash
(E\cross \times F\cross)}
r(g)\phi(y,u).$$
Here $\mu_U=F^\times \cap U$, and $w_U=|\{1,-1\}\cap U|$ is equal to 1
or 2. The identity holds as in the case of usual theta series. 
Its first Fourier coefficient is just
$$
\frac{1}{w_U} \sum_{(y,u) \in \mu_U \backslash
(E\cross \times F\cross)_1} \log \delta(g) r(g)\phi(y,u). 
$$
Write 
$$
\log \delta(g) r(g)\phi(y,u)=\log \delta(g_f) r(g)\phi(y,u)+\log \delta(g_\infty) W^{(2)}(g_\infty)\cdot  r(g_f)\phi_f(y,u).
$$
Then $\pr'$ doesn't change the first sum of the right-hand side since it is holomorphic of weight two at infinity, but changes
$\log \delta(g_\infty) W^{(2)}(g_\infty)$ in the second sum to some multiple
$c_2\ W^{(2)}(g_\infty)=c_2\ r(g)\phi_\infty(y,u)$, where $c_2 $ is
some constant to be determined.
As a consequence, 
\begin{align*}
&\pr' \left(\log \delta(g) 
 \sumu\sum_{y\in E} r(g)\phi(y,u)\right)\\
=&\, \frac{1}{w_U} \sum_{a\in F\cross} \sum_{(y,u) \in \mu_U \backslash (E\cross
\times F\cross)_1}
 \log \delta_f(d^*(a)g_f) r(d^*(a)g)\phi(y,u)\\
&+ c_2  \frac{1}{w_U} \sum_{a\in F\cross} \sum_{(y,u) \in \mu_U \backslash
(E\cross \times F\cross)_1} r(d^*(a)g)\phi(y,u)\\
=&\,  \sumu\sum_{y\in E^\times}
 (\log \delta(g_f)+\log |uq(y)|_f^{\frac 12}) r(g)\phi(y,u)
+ c_2   \sumu\sum_{y\in E^\times}
r(g)\phi(y,u).
\end{align*}
As for the constant, we have
$$\frac{c_2}{[F:\QQ]}= 4\pi \lim_{s\rightarrow 0}
\int_{F_{v,+}} y^s e^{-2\pi y}\ (\log y^{\frac 12})
ye^{-2\pi y} \frac{dy}{y}
= 2\pi \int_0^{\infty} e^{-4\pi y}  \log y dy =-\frac{1}{2}
(\gamma+\log 4\pi ).$$
Here $\gamma$ is Euler's constant.
Then the combined constant 
$$
c_1=c_0-2mc_2
=2\frac{L'(0,\eta)}{L(0,\eta)} +\log|d_E/d_F|+(\gamma+\log 4\pi )m.
$$
Here $m=[F:\QQ]$. 
The gamma factor 
$$
L_\infty(s,\eta)=\left(\pi^{-\frac{s+1}{2}} \Gamma(\frac{s+1}{2})\right)^m
$$
gives 
\begin{align*}
\frac{L_\infty'(0,\eta)}{L_\infty(0,\eta)}
= - \frac{1}{2}m(\gamma+\log 4\pi ).
\end{align*}
Thus
$$
c_1
=2\frac{L_f'(0,\eta)}{L_f(0,\eta)} +\log|d_E/d_F|.
$$
\end{proof}

\subsection{Choice of the Schwartz function} \label{choices}

To make further explicit local computations, we need to specify the Schwartz function. 

Start with the setup of Theorem \ref{quaternion main}.
Let $F$ be a totally real field, and $E$ be a totally imaginary quadratic extension of $F$. Let $\BB$ be a totally definite incoherent quaternion algebra over $\BA=\BA_F$ with an embedding $E_\BA\to \BB$ of $\BA$-algebras.
Let $U=\prod_{v\nmid\infty} U_v$ be a maximal open compact subgroup of $\BB_f^\times$ containing (the image of) $\wh O_E^\times=\prod_{v\nmid\infty} O_{E_v}^\times$. As in Theorem \ref{quaternion main}, assume that there is no non-archimedean place of $F$ ramified in $E$ and $\BB$ simultaneously. 

Note that we have already assumed that $U_v$ is maximal at any $v\nmid\infty$. Denote by $O_{\bv}$ the $O_{F_v}$-subalgebra of $\bv$ generated by 
$U_v$.  Then $O_{\bv}$ is a maximal order of $\bv$, and $U_v=O_{\bv}^\times$ is the group of invertible elements. Furthermore, the inclusion $O_{E_v}^\times\subset U_v$ induces $O_{E_v}\subset O_{\bv}$.

As for the Schwartz function $\phi=\otimes_v \phi_v$, we make the following choices:
\begin{itemize}
\item[(1)] If $v$ is archimedean, set $\phi_v$ be the standard Gaussian. 
\item[(2)] If $v$ is non-archimedean, nonsplit in $E$ and split in $\BB$, set $\phi_v$ to be the standard characteristic function $1_{O_\bv\times\ofv\cross}$. 
\item[(3)] If $v$ is nonsplit in $\BB$, set $\phi_v$ to be $1_{O_\bv^\times\times\ofv\cross}$ (instead of $1_{O_\bv\times\ofv\cross}$). 
\item[(4)] There is a set $S_2$ consisting of two (non-archimedean) places of $F$ split in $E$ and unramified over $\BQ$ such that 
$$\phi_v=1_{O_\bv^\times\times \ofv\cross}-\frac{1}{1+N_v+N_v^2}1_{\varpi_v^{-1}(O_\bv)_{2}\times\ofv\cross}, \quad \forall v\in S_2.$$
Here $\varpi_v$ denotes a uniformizer of $\ofv$, and
$$
(O_\bv)_{2}=\{x\in O_\bv: v(q(x))=2\}.
$$ 
\item[(5)] If $v$ is split in $E$ and $v\notin S_2$, set $\phi_v$ to be the standard characteristic function $1_{O_\bv}\otimes 1_{\ofv\cross}$. 
\end{itemize}
By definition, $\phi$ is invariant under both the left action and the right action of $U$. 

Note that (4) seems least natural in the choices. However, it is made to meet Assumption \ref{assumption3}. 
In fact, as in the proof of \cite[Proposition 5.15]{YZZ}, any function of the form 
$$
L\phi_0-\deg(L)\phi_0, \quad \phi_0\in \OCS(\BB_v\times F_v\cross), \
L\in C_c^\infty(\bb_v^1O_{\bb_v}^\times)
$$
satisfies the assumption. 
The choice of (4) comes from $\phi_0=1_{O_\bv\cross} \otimes  1_{\ofv\cross}$
and $L=1_{(O_\bv)_{2}}$. 
It is classical that $\deg((O_\bv)_{2})=|(O_\bv)_{2}/O_\bv^\times|=1+N_v+N_v^2$.

For any $v\nmid\infty$, fix an element $\jv\in O_{\bb_v}$ orthogonal to $E_v$ such that $v(q(\mathfrak{j}_v))$ is non-negative and minimal; i.e.,
$v(q(\mathfrak{j}_v))\in\{0, 1\}$, and such that $v(q(\mathfrak{j}_v))=1$ if and only if $\bv$ is nonsplit (and thus $E_v/F_v$ is inert by assumption). 
We check the existence of $\jv$ in the following. 

If $v$ is nonsplit in $\BB$ (and inert in $E$), then $O_{\bb_v}$ is the unique maximal order of $\bv$. It is easy to see the existence of $\jv$. We have $v(q(\mathfrak{j}_v))=1$ and an orthogonal decomposition $O_\bv=\oev+\oev \mathfrak{j}_v$.

If $v$ is split in $\BB$, start with an isomorphism $O_\bv \to M_2(O_{F_v})$. By this isomorphism, $O_\bv$ acts on $M=O_{F_v}^2$, and thus the subalgebra $O_{E_v}$ also acts on $M$. Fix a nonzero element $m_0\in M$. 
We have an isomorphism $O_{E_v}\to M$ of $O_{F_v}$-modules by $t\mapsto t\circ m_0$. Thus it induces an $O_{F_v}$-linear action of $O_\bv$ on $\oev$, which is compatible with the multiplication action of $\oev$ on itself. 
Set $\jv\in O_\bv$ to be the unique element which acts on $\oev$ as the nontrivial element of $\Gal(E_v/F_v)$. Then $\jv^2=1$ and $\jv t \jv =\bar t$ for any $t\in \oev$. It follows that $\jv$ is orthogonal to $E_v$, and $q(\jv)=-1$ satisfies the requirement. 

For any non-archimedean place $v$ nonsplit in $E$, let $B(v)$ be the nearby quaternion algebra. Fix an embedding $E\to B(v)$ and isomorphisms $B(v)_{v'}\simeq \bb_{v'}$ for any $v'\neq v$, which are assumed to be compatible with the embedding $E_{\BA}\to \BB$.
At $v$, we also take an element $j_v\in B(v)_v$ orthogonal to $E_v$, such that $v(q(j_v))$ is non-negative and minimal as above. 
We remark that this set $\{\mathfrak j_{v'}:v'\neq v\} \cup \{j_v\}$ is not required to be the localizations of a single element of $B(v)$.

\begin{lem} \label{orders}
Let $v$ be a non-archimedean place of $F$ and $D_v\subset \ofv$ be the relative discriminant of $E_v/F_v$. 
Then in the above setting,
$$D_vO_\bv\ \subset\ \oev+\oev \mathfrak{j}_v\ \subset\ O_\bv.$$
Furthermore, 
$O_\bv=\oev+\oev \mathfrak{j}_v$ if and only if $v$ is unramified in $E$. 
\end{lem}
\begin{proof}
This is classical. 
Assume that $v$ is split in $\BB$, since the nonsplit case is easy. 
For any (full) lattice $M$ of $\BB_v$, the discriminant $d_M$ is the fraction ideal of $F_v$ generated by $\det(\tr(x_i\bar x_j))$, where $x_1,\cdots, x_4$ is an $\ofv$-basis of $M$. In particular, if $M'\subset M$ is a sub-lattice, then 
$[d_M:d_{M'}]=[M:M']^2$. 
Direct computation gives $d_{O_\bv}=1$
and $d_{\oev+\oev \mathfrak{j}_v}=D_v^2$. 
The statement follows. 
\end{proof}

\subsection{Explicit local derivatives} \label{explicit derivative}

Let $(U,\phi, \jv, j_v)$ be as in \S \ref{choices}. 
The goal of this subsection is to compute $k_{\phi_v}(1,y,u)$ and $c_{\phi_v}(1,y,u)$.
The computations are quite involved, though the result are not so complicated eventually.
The readers may skip this subsection for the first time and come back when 
the results are used in the comparison with the height series.

Throughout this subsection, $v$ is non-archimedean.
For $y\in B(v)_v$, write $y=y_1+y_2$ with respect to the orthogonal decomposition $B(v)_v=E_v+E_vj_v$. 
By Lemma \ref{orders}, if $v\notin S_2$ and $v$ is unramified in $E$, we have a decomposition 
$\phi_{v}=\phi_{1,v}\otimes \phi_{2,v}$ with $\phi_{2,v}=1_{O_{E_v}\jv\times O_{F_v}^\times}$. 
Here $\phi_{1,v}=1_{O_{E_v} \times O_{F_v}^\times}$
if $v$ is split in $\BB$,  and
$\phi_{1,v}=1_{O_{E_v}^\times \times O_{F_v}^\times}$
if $v$ is nonsplit in $\BB$.

All Haar measures are normalized as in \cite[\S 1.6]{YZZ}, unless otherwise described.

\subsubsection*{Derivative of Whittaker function I}

\begin{lem}\label{derivative of Whittaker function}
\begin{itemize}
\item[(1)]
Let $v$ be a non-archimedean place inert in $E$. Then the difference
$$
k_{\phi_v}(1,y,u)- \phi_{v}(y_1,u)\cdot 1_{O_{E_v}j_v}(y_2)\cdot \frac{1}{2}(v(q(y_2)/q(\jv))+1) \log N_v
$$
extends to a Schwartz function on $B(v)_v\times F_v^\times$ whose restriction to $E_v\times F_v^\times$ is equal to 
\begin{align*}
 \phi_{v}(y,u)\cdot \frac{|d_vq(\jv)|-1}{(1+N_v^{-1})(1-N_v)} \log N_v.
\end{align*}

\item[(2)]
Let $v$ be a non-archimedean place ramified in $E$. Then the difference
$$
k_{\phi_v}(1,y,u)- \phi_{v}(y_1,u)\cdot 1_{O_{E_v}j_v}(y_2)\cdot \frac{1}{2}(v(q(y_2))+1) \log N_v
$$
extends to a Schwartz function on $B(v)_v\times F_v^\times$ whose restriction to $E_v\times F_v^\times$ is equal to 
\begin{align*}
\phi_{v}(y,u)\cdot
 \left(\frac{|d_v|-1}{2(1-N_v)} +\frac{1}{2}(v(D_v)-1) \right) \log N_v
+\frac{1}{2}\alpha_v(y,u),
\end{align*}
where 
$$
\alpha_v(y,u)
=\frac{\log N_v}{|D_v|^{\frac{1}{2}}}  \cdot 1_{D_v^{-1}O_{E_v}-O_{E_v}}(y)
 \sum_{n=0}^{v(d_v)-1}
N_v^{n} \int_{D_n} \phi_{v}(y+x_2,u) dx_2.
$$

\end{itemize}

\end{lem}

The result allows more ramifications of $v$ in $E$ or $\BB$ than its counterpart in \cite[Corollary 6.8(1)]{YZZ}. The computation follows a similar strategy, but it is  more complicated due to these ramifications. 

Recall that if $\phi_{v}=\phi_{1,v}\otimes \phi_{2,v}$, then
\begin{equation*}
k_{\phi_v}(1,y,u)= \frac{L(1,\eta_v)}{\mathrm{vol}(E_v^1)} 
\phi_{1,v}(y_1,u) {W_{uq(y_2),v}^\circ}'(0,1,u, \phi_{2,v}).  
\end{equation*}
Here $\mathrm{vol}(E_v^1)$ is given in \cite[\S 1.6.2]{YZZ}.
By \cite[Proposition 6.10]{YZZ},
\begin{eqnarray*}
{W_{a,v}^\circ}(s,1,u,\phi_{2,v})
= |d_v|^{\frac{1}{2}}  (1-N_v^{-s}) \sum_{n=0}^{\infty}
N_v^{-ns+n} \int_{D_n(a)} \phi_{2,v}(x_2,u) dx_2,
\end{eqnarray*}
where 
$$D_n(a)=\{x_2 \in E_v\jv:  uq(x_2)-a \in p_v^nd_v^{-1}\},$$
and $dx_2$ is the self-dual measure for $(E_v\jv, uq)$, which gives
$\vol(\oev\jv)=|D_v|^{\frac 12} |d_vuq(\jv)|$.
In the following, we will always denote $a=uq(y_2)$ for simplicity. 

We can also obtain a coherent expression of $k_{\phi_v}(1,y,u)$ which does not require $\phi_v$ to be of the form $\phi_{1,v}\otimes \phi_{2,v}$.
In fact, in the case $\phi_v=\phi_{1,v}\otimes \phi_{2,v}$ (and $v$ is nonsplit in $E$), the above gives 
\begin{eqnarray*}
k_{\phi_v}(1,y,u)
&=&\frac{L(1,\eta_v)}{\mathrm{vol}(E_v^1)}  \phi_{1,v}(y_1,u)\cdot \frac{d}{ds}  |_{s=0}  \left(
|d_v|^{\frac{1}{2}}  (1-N_v^{-s}) \sum_{n=0}^{\infty}
N_v^{-ns+n} \int_{D_n(a)} \phi_{2,v}(x_2,u) dx_2
\right)\\
&=& \frac{L(1,\eta_v)}{\mathrm{vol}(E_v^1)} \cdot \frac{d}{ds}  |_{s=0}  \left(
 |d_v|^{\frac{1}{2}}
  (1-N_v^{-s}) \sum_{n=0}^{\infty}
N_v^{-ns+n} \int_{D_n(a)} \phi_{v}(y_1+x_2,u) dx_2
\right).
\end{eqnarray*}
The last expression is actually valid for any $\phi_v$.
It is nonzero only if $u\in O_{F_v}^\times$, which we will always assume in the following. 

The computation relies on a detailed description of $D_n(a)$. For example, we will see that $D_n(a)$ is empty if $n$ is sufficiently large, so the summation for $k_{\phi_v}(1,y,u)$ has only finitely many non-zero terms. Then the derivative commutes with the sum. 

In the following lemma, $v$ is a non-archimedean place nonsplit in $E$.
Consider 
$$D_n(a)=\{x_2 \in E_v\jv:  uq(x_2)-a \in p_v^nd_v^{-1}\}, \quad
u\in O_{F_v}^\times, \ a\in uq(E_v^\times j_v)$$
and
$$
D_n=\{x_2 \in E_v\jv:  uq(x_2) \in p_v^nd_v^{-1}\}, \quad u\in O_{F_v}^\times.
$$

\begin{lem} \label{domain}
\begin{itemize}
\item[(1)]
If $v$ is inert in $E$, then
$$
D_n(a)
=\begin{cases}
D_n & \mathrm{if\ } n\leq v(ad_v);\\
\emptyset & \mathrm{if\ } n>v(ad_v).
\end{cases}
$$

\item[(2)]
If $v$ is ramified in $E$, then
$$
D_n(a)
=\begin{cases}
D_n & \mathrm{if\ } n\leq v(ad_v);\\
\emptyset & \mathrm{if\ } n>v(ad_v)+v(D_v)-1.
\end{cases}
$$
If $v(ad_v) <n \leq v(ad_v)+v(D_v)-1$, then 
$$
\vol(D_n(a))
= |D_v|^{\frac12} \cdot |d_v| \cdot  |a|_v \cdot N_v^{v(ad_v)-n}.$$
Here the volume is taken with respect to the self-dual measure for $(E_v\jv, uq)$, which gives
$\vol(\oev\jv)=|D_v|^{\frac 12} |d_v|$.
\end{itemize}

\end{lem}

\begin{proof}
The key property is that $a$ is not represented by $(E_v\jv, uq)$, since it is represented by $(E_vj_v, uq)$.

We first consider (1), so $v$ is inert in $E$.
Then $v(a)\neq v(uq(x_2))$ for any $x_2\in E_v\jv$ since $a$ is not represented by $(E_v\jv, uq)$. 
It follows that 
$$v(uq(x_2)-a)= \min\{v(a),
v(uq(x_2))\}.$$ 
The result follows. 

Now we consider (2), so $v$ is ramified in $E$. 
If $n\leq v(ad_v)$, the result is trivial. 
Assume that $n> v(ad_v)$ in the following. 
Let $e_v$ be the smallest integer such that 
$1+p_v^{e_v}\subset q(E_v^\times).$
By the class field theory, we have $e_v=v(D_v)$. 

The condition $x_2\in D_n(a)$ gives
$$a^{-1}uq(x_2) \in 1+p_v^{n-v(ad_v)}.$$
By $a=uq(y_2)$ with $y_2\in E_v^\times j_v$, the condition becomes
$$q(x_2)/q(y_2) \in 1+p_v^{n-v(ad_v)}.$$
Note that $q(E_v^\times\jv)$ and $q(E_v^\times j_v)$ are exactly the two cosets of $F_v^\times$ under the subgroup $q(E_v\cross)$ of index 2. 
Then $q(x_2)/q(y_2)$ always lies in the non-identity coset. 
Hence, $D_n(a)$ is empty if $n-v(ad_v)\geq e_v$ by the definition of $e_v$. 

It remains to compute $\vol(D_n(a))$ for 
$v(ad_v) <n \leq v(ad_v)+e_v-1$.
Write $m=n-v(ad_v)$, which satisfies $1\leq m\leq e_v-1$. 
The above condition on $x_2$ is just $a^{-1}uq(x_2) \in (1+p_v^{m})$.
We need to consider the intersection $(1+p_v^{m})\cap a^{-1}uq(E_v^\times\jv)$.
By the definition of $e_v$, we see that 
$(1+p_v^{m})$ is not completely contained in either 
$q(E_v^\times\jv)$ or
$q(E_v^\times j_v)$. 
Thus $(1+p_v^{m})$ is partitioned into two cosets 
$q(E_v^\times\jv)\cap (1+p_v^{m})$ and
$q(E_v^\times j_v)\cap (1+p_v^{m})$. 
In particular, $(1+p_v^{m})\cap a^{-1}uq(E_v^\times\jv)$ is one of the cosets. 
Therefore, 
$$
\vol((1+p_v^{m})\cap a^{-1}uq(E_v^\times\jv), d^\times x)
=\frac12 \vol(1+p_v^{m}, d^\times x)
=\frac{\vol(O_{F_v}^\times, d^\times x)}{2(N_v-1)N_v^{m-1}} 
=\frac{|d_v|^{\frac12}}{2(N_v-1)N_v^{m-1}}.
$$
Here the volumes are under the multiplicative measure $d^\times x=\zeta_{F_v}(1)|x|_v^{-1} d x$, but we will convert it back to $dx$. 
Similar measures $dx$ and $d^\times x$
are defined on $E_v$ as in \cite[\S 1.6.1-1.6.2]{YZZ}. Both measures are transferred to $E_v\jv$ by the identification $E_v\jv\to E_v$ sending $\jv$ to 1. 
The induced measure $dx$ on $E_v\jv$ is compatible with the self-dual measure with respect to the quadratic form $uq$.

Therefore,
$$
\vol(D_n(a), d^\times x)
=\vol(E_v^1)\cdot
\vol((1+p_v^{m})\cap a^{-1}uq(E_v^\times\jv), d^\times x)
=\frac{ |D_v|^{\frac12}|d_v| }{(N_v-1)N_v^{m-1}} .
$$
The additive volume is just 
$$
\vol(D_n(a), d x)
=\frac{|a|_v}{\zeta_{E_v}(1)}\vol(D_n(a), d^\times x)
=\frac{|a|_v\cdot |D_v|^{\frac12} \cdot |d_v|}{N_v^m}.
$$
\end{proof}

\subsubsection*{Derivative of Whittaker function II}

The goal of this subsection is to prove Lemma \ref{derivative of Whittaker function}.

\begin{proof}[Proof of Lemma \ref{derivative of Whittaker function}]

We first consider (1), so we assume that $v$ is inert in $E$.
We will take advantage of the decomposition
$\phi_{v}=\phi_{1,v}\otimes \phi_{2,v}$, which simplifies the computation slightly.
It amounts to computing the derivative of 
\begin{eqnarray*}
{W_{a,v}^\circ}(s,1,u)
= |d_v|^{\frac{1}{2}}  (1-N_v^{-s}) \sum_{n=0}^{\infty}
N_v^{-ns+n} \int_{D_n(a)} \phi_{2,v}(x_2,u) dx_2.
\end{eqnarray*}
Note that we always write $a=uq(y_2)$. 
By Lemma \ref{domain}, 
\begin{eqnarray*}
{W_{a,v}^\circ}'(0,1,u) &=& |d_v|^{\frac{1}{2}}  \log N_v \sum_{n=0}^{v(ad_v)}
N_v^n \int_{D_n} \phi_{2,v}(x_2,u) dx_2.
\end{eqnarray*}
It is nonzero only if $v(a)\geq -v(d_v)$. 

We first consider the case $-v(d_v) \leq v(a) <0$. 
In this case, we always have $\oev\jv\subset D_n$ for all $0\leq n \leq v(ad_v)$. 
It follows that 
\begin{eqnarray*}
{W_{a,v}^\circ}'(0,1,u) 
= |d_v|^{\frac{1}{2}}  \log N_v \sum_{n=0}^{v(ad_v)}
N_v^n\ \vol(\oev\jv)
= |d_v|^{\frac{1}{2}} |q(\jv)|
   \frac{|d_v|-N_v|a|^{-1}}{1-N_v} \log N_v.
\end{eqnarray*}
Note that this part does not affect the behavior as $a\to 0$.

Now we assume that $v(a)\geq 0$ (still for part (1)). 
If $n<v(d_vq(\jv))$, then $\oev\jv\subset D_n$; if
$n\geq v(d_vq(\jv))$, then $D_n\subset \oev\jv$. 
It follows that
\begin{align*}
 {W_{a,v}^\circ}'(0,1,u) 
=&  |d_v|^{\frac{1}{2}} \log N_v
  \left( \sum_{n=0}^{v(d_vq(\jv))-1} N_v^n\ \vol(\oev\jv)
 + \sum_{n=v(d_vq(\jv))}^{v(ad_v)} N_v^n\ \vol(D_n) \right) \\
=&  |d_v|^{\frac{1}{2}}  \log N_v
  \left( \frac{|d_vq(\jv)|-1}{1-N_v}
 +   \sum_{n=v(d_vq(\jv))}^{v(ad_v)} N_v^n\ \vol(D_n) \right).
\end{align*}
Note that
$$
D_n = p_v^{[\frac{n-v(d_vq(\jv))+1}{2}]}O_{E_v}\jv,
$$
so
$$
N_v^n\ \vol(D_n) = N_v^{n-v(d_vq(\jv))-2 [\frac{n-v(d_vq(\jv))+1}{2}]}=
\begin{cases}
1 & \mathrm{if\ } 2\mid (n-v(d_vq(\jv));\\
N_v^{-1} & \mathrm{if\ } 2\nmid (n-v(d_vq(\jv)).
\end{cases}
$$
Since $v(q(\jv))$ and $v(a)$ always have different parities in this inert case, we have
\begin{align*}
k_{\phi_v}(1,y,u)
=& \frac{\log N_v}{1+N_v^{-1}} \left( \frac{|d_vq(\jv)|-1}{1-N_v}
 + \frac{v(q(y_2))-v(q(\jv))+1}{2} (1+N_v^{-1}) \right).
\end{align*}
This finishes the proof of (1).

Now we prove (2), so $v$ is ramified in $E$. 
We need to compute
\begin{eqnarray*}
k_{\phi_v}(1,y,u)
&=&  \frac{1}{2|D_v|^{\frac{1}{2}}} \cdot \frac{d}{ds}  |_{s=0}  \left(
  (1-N_v^{-s}) \sum_{n=0}^{\infty}
N_v^{-ns+n} \int_{D_n(a)} \phi_{v}(y_1+x_2,u) dx_2
\right).
\end{eqnarray*}
We first use Lemma \ref{domain} to write
\begin{eqnarray*}
k_{\phi_v}(1,y,u)
&=&  \frac{\log N_v}{2|D_v|^{\frac{1}{2}}}    \sum_{n=0}^{v(ad_v)+v(D_v)-1}
N_v^{n} \int_{D_n(a)} \phi_{v}(y_1+x_2,u) dx_2.
\end{eqnarray*}
It is zero if $v(a)$ is too small, so $k_{\phi_v}(1,y,u)$ is compactly supported. 

In this ramified case, the first complication is that $k_{\phi_v}(1,y,u)$ can be nonzero for some
$y_1\notin O_{E_v}$. 
Write
$$
k_{\phi_v}(1,y,u)=
k_{\phi_v}(1,y,u)\cdot 1_{O_{E_v}}(y_1)
+k_{\phi_v}(1,y,u)\cdot 1_{E_v-O_{E_v}}(y_1).
$$ 
We first treat the second term on the right-hand side, so we assume that
$y_1\in E_v-O_{E_v}$.

We claim that $k_{\phi_v}(1,y,u)\cdot 1_{E_v-O_{E_v}}(y_1)$ is naturally a Schwartz function on $B(v)_v\times F_v^\times$. 
In fact, by Lemma \ref{orders}, in order to make $\phi_{v}(y_1+x_2,u)$ nonzero in the formula of
$k_{\phi_v}(1,y,u)$,  we have
$$
y_1\in D_v^{-1}O_{E_v}-O_{E_v}, \quad
x_2\in D_v^{-1}O_{E_v}\jv-O_{E_v}\jv.
$$
Then both $v(q(y_1))$ and $v(q(x_2))$ are bounded from the above and the below. 
Consider the behavior when $a=uq(y_2)$ approaches 0. 
By Lemma \ref{domain}, $x_2\in D_n(a)$ only if 
$$
n\leq v(q(x_2))+v(d_v)+v(D_v)-1 \leq v(d_v D_v^3)-1.$$
The second bound is independent of $a$. 
Hence, if $v(a)$ is sufficiently large, then
$D_n(a)=D_n$ independent of $a$. 
So $k_{\phi_v}(1,y,u)\cdot 1_{E_v-O_{E_v}}(y_1)$ is a Schwartz function on $B(v)_v\times F_v^\times$. 

For the restriction to $E_v\times F_v^\times$, set $y_2\to 0$. 
The above discussion already gives 
\begin{eqnarray} \label{111}
k_{\phi_v}(1,y_1,u)\cdot 1_{E_v-O_{E_v}}(y_1)
=  \frac{\log N_v}{2|D_v|^{\frac{1}{2}}}  \cdot 1_{E_v-O_{E_v}}(y_1)
 \sum_{n=0}^{v(d_v D_v^3)-1}
N_v^{n} \int_{D_n} \phi_{v}(y_1+x_2,u) dx_2.
\end{eqnarray}
We can further change the bounds of $n$ in the summation from $[0,v(d_v D_v^3)-1]$ to $[0,v(d_v)]$, because $x_2\in D_n$ implies 
$$
n\leq v(d_v)+v(q(x_2))\leq v(d_v). 
$$
Then the expression is exactly the function $\frac12\alpha_v$ in the lemma. 

It remains to treat $k_{\phi_v}(1,y,u)\cdot 1_{O_{E_v}}(y_1)$. 
Assume that $y_1\in O_{E_v}$.
Then 
\begin{eqnarray*}
k_{\phi_v}(1,y,u)
&=&  \frac{\log N_v}{2|D_v|^{\frac{1}{2}}}    \sum_{n=0}^{v(ad_v)+v(D_v)-1}
N_v^{n} \ \vol(D_n(a)\cap O_{E_v}\jv).
\end{eqnarray*}
The sum is nonzero only if $v(a)\geq -v(d_v)-v(D_v)+1$. 
The behavior of $k_{\phi_v}(1,y,u)$ when $-v(d_v)-v(D_v)+1\leq v(a)<0$ does affect our final result. So we assume that $v(a)\geq 0$ in the following.  

The computation is similar to the inert case.
Recall that $\vol(\oev\jv)=|D_v|^{\frac 12} |d_v|$
and
$$D_n(a)=\{x_2 \in E_v\jv:  uq(x_2)-a \in p_v^nd_v^{-1}\}.$$
Split the summation as
$$
\sum_{n=0}^{\infty}
=\sum_{n=0}^{v(d_v)-1}  
+ \sum_{n=v(d_v)}^{v(ad_v)}
+ \sum_{n=v(ad_v)+1}^{v(ad_v)+v(D_v)-1}.
$$
The first sum gives 
\begin{eqnarray} \label{222}
&& \frac{\log N_v}{2|D_v|^{\frac{1}{2}}} \sum_{n=0}^{v(d_v)-1}
N_v^{n} \ \vol(O_{E_v}\jv)
= \frac{|d_v|-1}{2(1-N_v)} \log N_v.
\end{eqnarray}
The second sum gives 
\begin{eqnarray} 
  \frac{\log N_v}{2|D_v|^{\frac{1}{2}}} \sum_{n=v(d_v)}^{v(ad_v)}
N_v^{n} \ \vol(D_n)
&=&
\frac{\log N_v}{2|D_v|^{\frac{1}{2}}}
 \sum_{n=v(d_v)}^{v(ad_v)} N_v^n \cdot N_v^{-(n-v(d_v))} 
 |D_v|^{\frac 12} |d_v|  \nonumber\\ 
&=& \frac{1}{2} (v(a)+1)  \log N_v. \label{333}
\end{eqnarray}
By Lemma \ref{domain}, the third sum gives 
\begin{eqnarray}
&&  \frac{\log N_v}{2|D_v|^{\frac{1}{2}}} 
\sum_{n=v(ad_v)+1}^{v(ad_v)+v(D_v)-1}
N_v^{n} \ \vol(D_n)  \nonumber\\
&=& \frac{\log N_v}{2|D_v|^{\frac{1}{2}}} 
\sum_{n=v(ad_v)+1}^{v(ad_v)+v(D_v)-1}
N_v^{n} \cdot 
|D_v|^{\frac12} \cdot |d_v| \cdot  |a|_v \cdot  N_v^{v(ad_v)-n}
\nonumber\\
&=& \frac{1}{2} (v(D_v)-1)  \log N_v. \label{444}
\end{eqnarray}
Combining equations (\ref{111})-(\ref{444}), we obtain the result for ramified $v$. 
The proof of Lemma \ref{derivative of Whittaker function} is complete. 
\end{proof}

\subsubsection*{Derivative of intertwining operator}

Recall that if $\phi_{v}=\phi_{1,v}\otimes \phi_{2,v}$ for a place $v$, then
$$c_{\phi_v}(g,y,u)=\phi_{1,v}(y,u) W_{0,v}^{\circ}\ '(0,g,u, \phi_{2,v})+ \log \delta(g_v)r(g)\phi_v(y,u),$$
where the normalization
\begin{eqnarray*}
W_{0,v}^{\circ}(s, g,u,\phi_{2,v})
&=& \gamma_{u,v}^{-1}|D_v|^{-\frac{1}{2}}|d_v|^{-\frac{1}{2}}\frac{L(s+1,\eta_v)}{L(s,\eta_v)}
 W_{0,v}(s, g,u,\phi_{2,v}).
\end{eqnarray*}

\begin{lem} \label{derivative of intertwining}
\begin{itemize}
\item[(1)] For any archimedean place $v$,
$$c_{\phi_v}(g,y,u)  = 0, \quad \ g\in \gl(\BR), \ (y,u)\in E_v\times F_v^\times.$$
\item[(2)]
For any non-archimedean place $v$ and any $(y,u)\in E_v\times F_v^\times$,
$$
c_{\phi_v}(1,y,u)  = 
\phi_{v}(y,u)\cdot \log |d_vq(\jv)|\ +\
\begin{cases}
\displaystyle \phi_{v}(y,u)\cdot
  \frac{2(|d_vq(\jv)|-1)}{(1+N_v^{-1})(1-N_v)} \log N_v,
&\quad \mbox{ if } E_v/F_v \mbox{ inert};\\
\displaystyle \phi_{v}(y,u)\cdot
  \frac{|d_vq(\jv)|-1 }{1-N_v} \log N_v+\alpha_v(y,u),
&\quad \mbox{ if } E_v/F_v \mbox{ ramified};\\
\displaystyle 0,
&\quad \mbox{ if } E_v/F_v \mbox{ split}.
\end{cases}
$$
Here 
$$
\alpha_v(y,u)
=\frac{\log N_v}{|D_v|^{\frac{1}{2}}}  \cdot 1_{D_v^{-1}O_{E_v}-O_{E_v}}(y)
 \sum_{n=0}^{v(d_v)-1}
N_v^{n} \int_{D_n} \phi_{v}(y+x_2,u) dx_2
$$
as in Lemma \ref{derivative of Whittaker function}. 
\end{itemize}

\end{lem}

\begin{proof}
If $v$ is archimedean, it suffices to check that 
$$
W_{0,v}^{\circ}(s,g,u) = \delta(g)^{-s} r(g)\phi_{2,v}(0,u), \quad g\in \gl(F_v). 
$$
The behaviors of the intertwining operator $W_{0,v}^{\circ}(s,g,u)$ under the left action of $P(\RR)$ and the right action of 
$\SO(2, \BR)$ are the same as those of $\delta(g)^{-s} r(g)\phi_{2,v}(0,u)$. 
It follows that two sides are equal up to a constant possibly depending on $s$. 
To determine the constant, it suffices to check $W_{0,v}^{\circ}(s,1,u)=1$. 
By a change of variable, we can assume that $u=1$.
At the end of the proof of \cite[Proposition 2.11]{YZZ}, there is a formula for $W_{0,v}(s,1,u)$ in terms of gamma functions, which implies the result we need here. 

Assume that $v$ is non-archimedean in the following. 
The proof is similar to that of Lemma \ref{derivative of Whittaker function}. We first introduce some formulas for $c_{\phi_v}(1,y,u)$.
Note that the statement of \cite[Proposition 6.10(1)]{YZZ} is only correct for $a\in F_v^\times$ due to the different normalizing factor defining 
$W_{0,v}^{\circ}(0,1,u,\phi_{2,v})$. However, its proof actually gives
\begin{eqnarray*}
W_{0,v}(s,1,u, \phi_{2,v})
&=&\gamma_{u,v} |d_v|^{\frac{1}{2}} (1-N_v^{-s}) \sum_{n=0}^{\infty} N_v^{-ns+n}
\int_{D_n} \phi_{2,v}(x_2,u) d_ux_2,
\end{eqnarray*}
where
$$
D_n =\{x_2 \in E_v\jv:  uq_2(x_2) \in p_v^nd_v^{-1}\}
$$
and the measure $d_ux_2$ gives $\vol(\oev\jv)=|D_v|^{\frac 12} |d_vuq(\jv)|$.
Putting these together, we have 
\begin{eqnarray*}
c_{\phi_v}(1,y,u)
&=& \phi_{1,v}(y,u)\cdot \frac{d}{ds}  |_{s=0}  \left(
|D_v|^{-\frac{1}{2}} \frac{L(s+1,\eta_v)}{L(s,\eta_v)}
(1-N_v^{-s})   \sum_{n=0}^{\infty} N_v^{-ns+n}
\int_{D_n} \phi_{2,v}(x_2,u) d_ux_2 \right)\\
&=&  \frac{d}{ds}  |_{s=0}  \left(
|D_v|^{-\frac{1}{2}} \frac{L(s+1,\eta_v)}{L(s,\eta_v)}
(1-N_v^{-s})   \sum_{n=0}^{\infty} N_v^{-ns+n}
\int_{D_n} \phi_{v}(y+x_2,u) d_ux_2 \right).
\end{eqnarray*}
The last expression actually works for any $\phi_v$ (not necessarily of the form 
$\phi_{1,v}\otimes \phi_{2,v}$).

For convenience, denote 
\begin{eqnarray*}
\tilde c_{\phi_v}(s)
=  |D_v|^{-\frac{1}{2}} \frac{L(s+1,\eta_v)}{L(s,\eta_v)}
(1-N_v^{-s})   \sum_{n=0}^{\infty} N_v^{-ns+n}
\int_{D_n} \phi_{v}(y+x_2,u) d_ux_2,
\end{eqnarray*}
so that 
$$
c_{\phi_v}(1,y,u)=\tilde c_{\phi_v}'(0).
$$
Note that $\tilde c_{\phi_v}(s)$ or $c_{\phi_v}(1,y,u)$ is nonzero only if $u\in O_{F_v}^\times$, which we assume in the following. 
We will check the lemma case by case. 

First, assume that $v$ is inert in $E$.
Then $\phi_v= \phi_{1,v}\otimes \phi_{2,v}$ with $\phi_{2,v}=1_{O_{E_v}\jv\times \ofv\cross}$, and 
$$c_{\phi_v}(1,y,u)=\phi_{1,v}(y,u) W_{0,v}^{\circ}\ '(0,1,u).$$
Split the sum in
\begin{eqnarray*}
W_{0,v}(s,1,u)
&=&\gamma_{u,v} |d_v|^{\frac{1}{2}} (1-N_v^{-s}) \sum_{n=0}^{\infty} N_v^{-ns+n}
\int_{D_n} \phi_{2,v}(x_2,u) dx_2
\end{eqnarray*}
into two parts: $n<v(d_vq(\jv))$ and $n\geq v(d_vq(\jv))$.
Denote $n=m+ v(d_vq(\jv))$ in the second case, and note $D_{m+ v(d_vq(\jv))} =p_v^{[\frac {m+1}{2}]}\oev\jv$.
We have
\begin{eqnarray*}
&&W_{0,v}(s,1,u)  \\
&=&\gamma_{u,v}  |d_v|^{\frac{1}{2}}  (1-N_v^{-s})\left(
 \sum_{n=0}^{v(d_vq(\jv))-1} N_v^{-n(s-1)} \vol(\oev\jv)
+\sum_{m=0}^{\infty} N_v^{-(m+ v(d_vq(\jv)))(s-1)} \vol(D_{m+ v(d_vq(\jv))})  \right)\\
&=&\gamma_{u,v}  |d_v|^{\frac{1}{2}} |D_v|^{\frac 12}(1-N_v^{-s})\left(
 \frac{|d_vq(\jv)|-|d_vq(\jv)|^{s} }{1-N_v^{-(s-1)} }
+ |d_vq(\jv)|^{s} \frac {1+N_v^{-(s+1)} }{1-N_v^{-2s}} \right).
\end{eqnarray*}
Then
\begin{eqnarray*}
W_{0,v}^{\circ}(s,1,u)
= (1-N_v^{-s})\frac {1+N_v^{-s} }{1+N_v^{-(s+1)}} \frac{|d_vq(\jv)|-|d_vq(\jv)|^{s} }{1-N_v^{-(s-1)} }
+ |d_vq(\jv)|^{s}  .
\end{eqnarray*}
We get
\begin{eqnarray*}
W_{0,v}^{\circ}\ '(0,1,u)
= \log |d_vq(\jv)|+  \frac{2(|d_vq(\jv)|-1) }{(1+N_v^{-1})(1-N_v)} \log N_v.
\end{eqnarray*}
This finishes the inert case.

Second, assume that $v$ is ramified in $E$.
Consider 
\begin{eqnarray*}
\tilde c_{\phi_v}(s)
=  |D_v|^{-\frac{1}{2}} 
(1-N_v^{-s})   \sum_{n=0}^{\infty} N_v^{-ns+n}
\int_{D_n} \phi_{v}(y+x_2,u) dx_2.
\end{eqnarray*}
As in the proof of Lemma \ref{derivative of Whittaker function}, the first complication of this ramified case is that $\tilde c_{\phi_v}(s)$ can be nonzero for some $y\notin O_{E_v}$, but it can be treated similarly. 

In fact, assume that $y\notin O_{E_v}$ and $\tilde c_{\phi_v}(s)\neq 0$. 
In order to  make $\phi_{v}(y+x_2,u)$ nonzero in the formula of
$\tilde c_{\phi_v}(s)$,  we have
$$
y\in D_v^{-1}O_{E_v}-O_{E_v}, \quad
x_2\in D_v^{-1}O_{E_v}\jv-O_{E_v}\jv.
$$
Then $x_2\in D_n$ gives
$$
n\leq v(q(x_2))+v(d_v) \leq v(d_v).$$
Then the summation for $\tilde c_{\phi_v}(s)$ is a finite sum. 
We have 
\begin{eqnarray*}
c_{\phi_v}(1,y,u)=\tilde c_{\phi_v}'(0)
=  |D_v|^{-\frac{1}{2}} 
(\log N_v)   \sum_{n=0}^{v(d_v)} N_v^{n}
\int_{D_n} \phi_{v}(y+x_2,u) dx_2.
\end{eqnarray*}
This is exactly the function $\alpha_v$ in the lemma. 

Now we assume that $y\in O_{E_v}$. Then 
\begin{eqnarray*}
\tilde c_{\phi_v}(s)
=  |D_v|^{-\frac{1}{2}} 
(1-N_v^{-s})   \sum_{n=0}^{\infty} N_v^{-ns+n}
\ \vol(D_n\cap O_{E_v}\jv).
\end{eqnarray*}
The computation is similar to the inert case. 
Split the sum into two parts: $n<v(d_v)$ and $n\geq v(d_v)$.
Denote $n=m+ v(d_v)$ in the second case, and note $D_{m+ v(d_v)} =p_v^{\frac {m}{2}}\oev\jv$.
We have
\begin{eqnarray*}
&& \tilde c_{\phi_v}(s) \\
&=&  |D_v|^{-\frac{1}{2}} 
(1-N_v^{-s})   \left(
 \sum_{n=0}^{v(d_v)-1} N_v^{-n(s-1)} \vol(\oev\jv)
+\sum_{m=0}^{\infty} N_v^{-(m+ v(d_v))(s-1)} \vol(D_{m+ v(d_v)})  \right)\\
&=& (1-N_v^{-s})\frac{|d_v|-|d_v|^{s} }{1-N_v^{-(s-1)} }+ |d_v|^{s} .
\end{eqnarray*}
Thus
\begin{eqnarray*}
c_{\phi_v}(1,y,u)=\tilde c_{\phi_v}'(0)
=\log |d_v|+  \frac{|d_v|-1 }{1-N_v} \log N_v.
\end{eqnarray*}

Third, consider the case that $E_v/F_v$ is split and $v\notin S_2$. Then $|q(\jv)|=|D_v|=1$ and we use it to relieve the notation burden.
We compute
$$c_{\phi_v}(1,y,u)=\phi_{1,v}(y,u) W_{0,v}^{\circ}\ '(0,1,u).$$
As before, split the sum into $n<v(d_v)$ and $n\geq v(d_v)$ and write $n=m+ v(d_v)$ in the second case.
We have
\begin{eqnarray*}
&&W_{0,v}(s,1,u)  \\
&=&\gamma_{u,v}  |d_v|^{\frac{1}{2}} (1-N_v^{-s})\left(
 \frac{|d_v|-|d_v|^{s} }{1-N_v^{-(s-1)} }
+ |d_v|^{s} \sum_{m=0}^{\infty} N_v^{-m(s-1)} \frac{\vol(D_{m+ v(d_v)}\cap \oev\jv)}{\vol(\oev\jv)}  \right).
\end{eqnarray*}
Identify $E_v=F_v\oplus F_v$ and $\oev=\ofv\oplus \ofv$. For simplicity, we identify $E_v\jv$ with $E_v$ by sending $\jv$ to 1.
Then
$$
D_{m+ v(d_v)}\cap \oev =\{(z_1,z_2) \in \ofv\oplus \ofv:  z_1z_2 \in p_v^m\}
=\oev-\{(z_1,z_2) \in \ofv\oplus \ofv:  v(z_1)+v(z_2) \leq m-1\}.
$$
Thus
\begin{align*}
\vol(D_{m+ v(d_v)}\cap \oev)
= & \vol(\oev)-\sum_{k=0}^{m-1}\vol(\varpi_v^k\ofv\cross)\vol(\ofv-p_v^{m-k})\\
= & \vol(\oev)-\vol(\oev) \sum_{k=0}^{m-1} N_v^{-k}(1-N_v^{-1}) (1-N_v^{-(m-k)})\\
= & \vol(\oev)(N_v^{-m}+(1-N_v^{-1})mN_v^{-m}).
\end{align*}
Therefore,
\begin{eqnarray*}
W_{0,v}(s,1,u)
&=&\gamma_{u,v}  |d_v|^{\frac{1}{2}} (1-N_v^{-s})\left(
 \frac{|d_v|-|d_v|^{s} }{1-N_v^{-(s-1)} }
+ |d_v|^{s} \sum_{m=0}^{\infty} N_v^{-m(s-1)}  (N_v^{-m}+(1-N_v^{-1})mN_v^{-m})  \right)\\
&=&\gamma_{u,v}  |d_v|^{\frac{1}{2}} (1-N_v^{-s})\left(
 \frac{|d_v|-|d_v|^{s} }{1-N_v^{-(s-1)} }
+ |d_v|^{s} \frac{ 1-N_v^{-(s+1)} }{(1-N_v^{-s})^2}  \right).
\end{eqnarray*}
Hence,
\begin{align*}
W_{0,v}^{\circ}(s,1,u)
=\gamma_{u,v}^{-1} \frac{1-N_v^{-s}}{ 1-N_v^{-(s+1)} } |d_v|^{-\frac{1}{2}} W_{0,v}(s,1,u)
=  \frac{(1-N_v^{-s})^2}{ 1-N_v^{-(s+1)} } \frac{ |d_v|-|d_v|^{s} }{1-N_v^{-(s-1)}}
+ |d_v|^{s}.
\end{align*}
The first term has a double zero and no contribution to the derivative, so
\begin{align*}
W_{0,v}^{\circ}\ '(0,1,u)
=  \log |d_v|.
\end{align*}
This finishes the case that $E_v/F_v$ is split and $v\notin S_2$.

Fourth, we treat the case $v\in S_2$, which is the last case. 
Then $v$ is split in $E$, and 
$$\phi_v=1_{O_\bv^\times\times \ofv\cross}-\frac{1}{1+N_v+N_v^2}1_{\varpi_v^{-1}(O_\bv)_{2}\times\ofv\cross}.$$
Note that $|q(\jv)|=|d_v|=1$ by assumption, so the result to prove is exactly $c_{\phi_v}(1,y,u)=0$. 
Recall that $c_{\phi_v}(1,y,u)$ is the derivative of
\begin{eqnarray*}
\tilde c_{\phi_v}(s)
=  \frac{(1-N_v^{-s})^2}{ 1-N_v^{-(s+1)} }   \sum_{n=0}^{\infty} N_v^{-ns+n}
\int_{D_n} \phi_{v}(y+x_2,u) dx_2.
\end{eqnarray*}

We will make separate computations for 
$$
\psi_1=1_{O_\bv^\times\times \ofv\cross}, \quad
\psi_2=1_{\varpi_v^{-1}(O_\bv)_{2}\times\ofv\cross}.
$$
The results will be 0 for both functions.
Make identifications $E_v\jv\simeq E_v\simeq F_v\oplus F_v$ as above.

Start with 
\begin{eqnarray*}
\tilde c_{\psi_1}(s)
= \frac{(1-N_v^{-s})^2}{ 1-N_v^{-(s+1)} }   \sum_{n=0}^{\infty} N_v^{-ns+n}
\int_{D_n} \psi_{1}(y+x_2,u) dx_2.
\end{eqnarray*}
It is nonzero only if $y\in O_{E_v}$, which we assume. For the integral, write $x_2=(z_1,z_2)\in F_v\oplus F_v$. 
Then we have 
\begin{eqnarray*}
\tilde c_{\psi_1}(s)
&= & \frac{(1-N_v^{-s})^2}{ 1-N_v^{-(s+1)} }   
\sum_{n=0}^{\infty} N_v^{-ns+n}\\
&&\qquad \cdot
\vol \{(z_1, z_2)\in O_{F_v}\oplus O_{F_v}: 
z_1z_2\in p_v^n, \ q(y)-z_1z_2\in O_{F_v}^\times \}.
\end{eqnarray*}

If $q(y)\in p_v$,  in order for the volume to be nonzero, we have to have $z_1z_2\in O_{F_v}^\times$ and $n=0$. The summation has a single nonzero term equal to 1. 
Then $c_{\psi_1}(1,y,u)=0.$

If $q(y)\in O_{F_v}^\times$,  we can neglect the term with $n=0$, since a single term does not change the derivative due to the double zero of the factor $(1-N_v^{-s})^2$. Then the remaining terms give 
\begin{eqnarray*}
&& \frac{(1-N_v^{-s})^2}{ 1-N_v^{-(s+1)} }   \sum_{n=1}^{\infty} N_v^{-ns+n}\cdot
\vol \{(z_1, z_2)\in O_{F_v}\oplus O_{F_v}: 
z_1z_2\in p_v^n \}.
\end{eqnarray*}
A similar summation has just been computed above, and the eventual result is still $c_{\psi_1}(1,y,u)=0.$ (Note that $d_v=1$ in the current case.)

Now we treat 
\begin{eqnarray*}
\tilde c_{\psi_2}(s)
&=& \frac{(1-N_v^{-s})^2}{ 1-N_v^{-(s+1)} }   \sum_{n=0}^{\infty} N_v^{-ns+n}
\int_{D_n} \psi_{2}(y+x_2,u) dx_2 \\
&= & \frac{(1-N_v^{-s})^2}{ 1-N_v^{-(s+1)} }   
\sum_{n=0}^{\infty} N_v^{-ns+n}\\
&&\qquad \cdot
\vol \{(z_1, z_2)\in p_v^{-1}\oplus p_v^{-1}: 
z_1z_2\in p_v^n, \ q(y)-z_1z_2\in O_{F_v}^\times \}.
\end{eqnarray*}
Here we have assumed $u\in O_{F_v}^\times$ and will assume $y\in \varpi_v^{-1}O_{E_v}$ in order to make the situation nontrivial. 
It is similar to the case $\psi_1$. 

If $q(y)\notin \ofv$, the summation has no nonzero term and thus $c_{\psi_1}(1,y,u)=0$.

If $q(y)\in p_v$, the summation has a single nonzero term coming from $n=0$. 
Then $c_{\psi_1}(1,y,u)=0$ again.

If $q(y)\in O_{F_v}^\times$, we can neglect the term with $n=0$ again. The remaining terms give 
\begin{eqnarray*}
&& \frac{(1-N_v^{-s})^2}{ 1-N_v^{-(s+1)} }   \sum_{n=1}^{\infty} N_v^{-ns+n}\cdot
\vol \{(z_1, z_2)\in p_v^{-1}\oplus p_v^{-1}: 
z_1z_2\in p_v^n \}\\
&=& \frac{(1-N_v^{-s})^2}{ 1-N_v^{-(s+1)} }   \sum_{n=1}^{\infty} N_v^{-ns+n}\cdot
N_v^2\cdot \vol \{(z_1', z_2')\in O_{F_v}\oplus O_{F_v}: 
z_1'z_2'\in p_v^{n+2} \}.
\end{eqnarray*}
Here we have used the substitution $z_i=\varpi_v^{-1}z_i'$. 
Then it is similar to the computation above and still gives $c_{\psi_1}(1,y,u)=0.$
This finishes the case $v\in S_2$. 
\end{proof}

\begin{remark}
It is not surprising that some (complicated and un-wanted) terms in the result of Lemma \ref{derivative of Whittaker function} appear in that of Lemma \ref{derivative of intertwining}. In fact, it just reflexes that the identity 
$$\lim_{a\to 0} W_{a,v}'(0,1,u)= W_{0,v}'(0,1,u),$$
which fails due to convergence issues, actually holds for some pieces of the two sides. 
Eventually we need these terms to cancel each other in order to get a neat Proposition \ref{combination}. 
\end{remark}

\section{Height series}
In this section, we study the intersection series of CM points, the main geometric ingredient for proving Theorem \ref{quaternion main}.
 We will first review the construction of the  series $Z(g, (t_1, t_2), \phi)$
in \cite{YZZ}.
Then we will compute this series under some assumption of Schwartz functions. In particular, we will obtain a term for the self-intersection of 
CM points which contributes a main term for the identity in Theorem \ref{quaternion main}.  In \cite{YZZ}, this term was killed under a stronger 
assumption of Schwartz functions.

\subsection{Height series}
Let $F$ be a totally real number field, and $\BB$ be a totally definite incoherent quaternion algebra over $F$ with ramification set $\Sigma$. 
To avoid complication of cusps, we assume that $|\Sigma|>1$.  
For any open compact subgroup $U$ of $\bfcross$, 
we have a Shimura curve $X_U$, which is a projective and smooth curve over 
$F$. For any embedding $\tau: F\hookrightarrow \BC$, 
it has the usual uniformization
$$X_{U, \tau}(\BC) = B(\tau)^\times \bs \gh^\pm\times
\bb_f^\times / U.$$
Here $B(\tau)$ denotes the nearby quaternion algebra, i.e., the unique quaternion algebra over $F$ with ramification set $\Sigma \setminus \{\tau\}$.

For any $x\in \bfcross$, we have a correspondence
$Z(x)_U$ defined as the image of the morphism
$$(\pi_{U_x, U}, \pi_{U_x, U}\circ \RT_x):
\quad X_{U_x}\lra X_U\times X_U.$$
Here $U_x=U\cap xUx^{-1}$, $\pi_{U_x, U}$ denotes the natural projection, and $\RT_x$ denotes the right multiplication by $x$.
In terms of the complex uniformization, the push-forward action gives
$$Z(x)_U:[z, \beta]_U\longmapsto \sum_{y\in UxU/U}[z, \beta y]_U.$$

\subsubsection*{Generating series}

We first recall the generating series in \cite[\S3.4.5]{YZZ}. 
For any $\phi\in \ol\CS (\BB\times \BA^\times)$ invariant under $K=U\times U$,
form a generating series
\begin{equation*}  
Z(g,\phi)_U=Z_0(g,\phi)_U+Z_*(g,\phi)_U, \quad g\in \gla,
\end{equation*}
where
\begin{eqnarray*}
Z_0(g,\phi)_U&=& -  \sum_{\alpha\in F_+\cross\bs \afcross/q(U)} \sum_{u\in
\mu_U^2 \bs F\cross}
  E_0(\alpha^{-1}u,r(g)\phi)\ L_{K,\alpha},\\
  Z_*(g,\phi)_U&=& w_U \sum_{a\in F^\times} \sum_{x\in U\bs \bb_f^\times/U} r(g)\phi(x,aq(x)^{-1})\ Z(x)_U.
\end{eqnarray*}
Here $\mu_U=F^\times \cap U$, and $w_U=|\{1,-1\}\cap U|$ is equal to 1
or 2.
We often abbreviate 
$$Z(g,\phi)_U, \ Z_0(g,\phi)_U, \ Z_*(g,\phi)_U$$
as 
$$Z(g)_U,\ Z_0(g)_U,\ Z_*(g)_U.$$ 

For our purpose on the height series, we will see that the constant term $Z_0(g,\phi)_U$ can be neglected in our consideration, since its contribution is always zero.

\begin{thm}\cite[Theorem 3.17]{YZZ} \label{modularity}
The series $Z(g, \phi)_U$ is absolutely convergent and defines an
automorphic form on $g\in\GL_2(\BA)$ with coefficients in $\Pic(X_U\times
X_U)_\BC$.
\end{thm}

\subsubsection*{Height series}

Let $E/F$ be a totally imaginary quadratic extension, with a fixed embedding
$E_\BA\hookrightarrow \BB$ over $\BA$. 
In \cite{YZZ}, we consider a CM point $P\in X^{E^\times}(E^\ab)$ on the limit of the Shimura curves.
In this paper, we only consider the point $P_U\in X_U(E^\ab)$ 
for fixed $U$. 
For a more precise description, fixing an embedding $\tau: F\hookrightarrow \BC$, take $P_U=[z_0,1]_U$
based on the uniformization
$$ X_{U, \tau}(\BC) =B(\tau)^\times \bs \gh^\pm\times
\bb_f^\times / U,$$
where $z_0\in\gh$ is the unique fixed point of $E\cross$ in $\gh$ via the action
induced by the embedding $E\hookrightarrow B(\tau)$. 
For simplicity, we write $P$ for $P_U$. 

In terms of the uniformization, there are two sets of CM points in $X_U(E^\ab)$ for our purpose:
$$
C_U=\{ [z_0,t]_U: t \in E^\times(\af)\}, \qquad
\CMU=\{ [z_0,\beta]_U: \beta \in \bfcross\}.
$$
It is easy to see canonical bijections
$$
C_U\cong E^\times \bs E^\times(\af)/(E^\times(\af)\cap U), \qquad
\CMU\cong E^\times \bs \bfcross/U.
$$
We will abbreviate $[z_0,\beta]_U$ as $[\beta]_U$, $[\beta]$ or just $\beta$.

For any $t\in E^\times(\BA)$, denote by 
$$[t]=[t]_U=[z_0, t_f]_U$$ 
the CM point of $X_{U, \tau}(\BC)$, viewed as an algebraic point of $X_U$.
Denote by
$$t^\circ=[t]_U^\circ=[t]_U-\xi_{U,t}$$ 
the degree-zero divisor on $X_U$, where $\xi_t=\xi_{U,t}$
is the normalized Hodge class of degree one on the connected component of $[t]_U$.

Recall from \cite[\S3.5.1, \S5.1.2]{YZZ} that we have a height series
\begin{eqnarray*}
Z(g, (t_1, t_2), \phi)_U
&=& \pair{Z(g, \phi)_U\ t_1^{\circ},\ t_2^{\circ} }_{\NT},
\quad t_1,t_2\in E^\times(\af).
\end{eqnarray*}
Here $Z(g, \phi)_U$ acts on $t_1^{\circ}$ as correspondences, and the pairing is the N\'eron--Tate height pairing
$$\pair{\cdot, \cdot}_{\NT}: J_U(\overline F)_\BC\times J_U(\overline F)_\BC
\lra \BC$$
on the Jacobian variety $J_U$ of $X_U$ over $F$. 

By linearity, $Z(g, (t_1, t_2), \phi)_U$ is an automorphic form in $g\in \gla$. 
By \cite[Lemma 3.19]{YZZ}, it is actually a cusp form. 
In particular, the constant term $Z_0(g,\phi)$ of the generating function plays no role here.

\subsubsection*{Decomposition of the height series}

By the theory of \cite[\S7.1]{YZZ}, we are going to decompose the height series into local pairings and some global terms. We will use (possibly) different integral models to do the decomposition. 

Assume that $(\BB,E,U)$ satisfies the assumptions of \S \ref{choices} in the following. 
In particular, $U$ is maximal at every place, and there is no non-archimedean place of $F$ ramified in both $E$ and $\BB$. 

Let $\CX_U$ be the integral model of $X_U$ over $O_F$ introduced before Corollary \ref{system2}, and let $\ol\CL_U$ be the arithmetic Hodge bundle introduced in Theorem \ref{system3}.
We are going to use $(\CX_U, \ol\CL_U)$ to decompose the Neron--Tate height pairing. 

Note that every point of $\CMU$ is defined over a finite extension $H$ of $F$ that is unramified above $\Sigma(\BB_f)$. The composite of two such extensions still satisfies the same property. By Corollary \ref{system2}, the base change 
$\CX_{U,O_H}$ is $\QQ$-factorial for such $H$. 
Then arithmetic intersection numbers of Arakelov divisors are well-defined on 
$\CX_{U,O_H}$.
Take the integral model $\CY_U$ used in \cite[\S 7.2.1]{YZZ} to be 
$\CX_{U,O_H}$ (without any desingularization).
We get a decomposition of $Z(g, (t_1, t_2))_U$ by the process of \cite[\S 7.2.2]{YZZ}. 

We do not know whether $\CX_U$ is regular everywhere or smooth above any prime of $F$ split in $\BB$. If both are true, then $\CX_{U,O_H}$ is already regular, and the decomposition here is the same as that in \cite{YZZ}.

\subsubsection*{Vanishing of the pairing with Hodge class}

Now we use freely the notations of \cite[\S 7.1-7.2]{YZZ}.
For the height series, the linearity gives a decomposition  
$$Z(g, (t_1, t_2))_U
=\pair{Z_*(g)_U t_1,  t_2}-\pair{Z_*(g)_U t_1, \xi_{t_2}}
-\pair{Z_*(g)_U \xi_{t_1}, t_2} +\pair{Z_*(g)_U\xi_{t_1}, \xi_{t_2}}.$$
Here 
$Z_*(g)_U=Z_*(g,\phi)_U$, and
the pairings on the right-hand side are arithmetic intersection numbers in terms of admissible extensions, as introduced in \cite[\S7.1.6]{YZZ}.

Now we resume the degeneracy assumption in \ref{assumption3}, which mainly requires that there is a set $S_2$ consisting of $2$ non-archimedean places of $F$ split in $E$ and unramified over $\BQ$ such that 
\begin{equation*}
r(g)\phi_v(0,u)=0, \quad \forall\ g\in \gl(F_v), \ u\in F_v\cross, \ v\in S_2.
\end{equation*}
By \cite[Proposition 7.5]{YZZ}, the assumption kills the last three terms on the right-hand side and gives the simplification
$$Z(g, (t_1, t_2))_U
=\pair{Z_*(g)_U t_1,  t_2}.$$

As in \cite[Proposition 7.5]{YZZ}, we have a decomposition 
$$Z(g, (t_1, t_2))_U
=-i(Z_*(g)_U t_1,  t_2)-j(Z_*(g)_U t_1,  t_2).$$
Here the $i$-part is essentially the arithmetic intersection number of horizontal parts, and the $j$-part is the contribution from vertical parts. 

Now we have a decomposition to local intersection numbers by
\begin{eqnarray*}
j(Z_*(g) t_1,  t_2) =\sum_v j_v(Z_*(g) t_1,  t_2)\log N_v.
\end{eqnarray*}
The sum is over all places of $F$, and we take the convention $\log N_v=1$ if $v$ is real. 
Decomposing the local intersection number in terms of Galois orbits, we further have
\begin{eqnarray*}
j_v(Z_*(g) t_1,  t_2)&=& \barint_{C_U} 
j_{\bar
v}(Z_*(g) tt_1,  tt_2)dt.
\end{eqnarray*}
Here the pairing $j_{\bar v}$ is introduced in \cite[\S 7.1.7]{YZZ}, and 
$$
C_U=E^\times \bs E^\times(\BA_f)/ E^\times(\BA_f)\cap U 
$$
is a finite group and the integration is just the usual average over this finite group. 

Unlike the $j$-part, the decomposition of the $i$-part into local intersection numbers is complicated due to the occurrence of self-intersections.
We have to isolate the self-intersections before the decomposition.  
Such a complication is diminished in \cite{YZZ} by Assumption 5.3 in it, but we cannot impose this assumption here. 
In fact, the assumption kills all possible self-intersections, but the purpose of this paper is to compute these self-intersections!

\subsubsection*{Self-intersection}

The self-intersection in $\pair{Z_*(g) t_1,  t_2}$ comes from the 
multiplicity of $[t_2]_U$ in $Z_*(g) t_1$. 
By definition, 
\begin{eqnarray*}
Z_*(g)t_1 = w_U  \sum_{a\in F\cross}\sum_{x \in \bb_f^\times /U} r(g)\phi(x)_a
[t_1x].
\end{eqnarray*}
Here $r(g)\phi(x)_a=r(g)\phi(x, a/q(x))$. 
See also \cite[\S 4.3.1]{YZZ} for this formula. 

Note that $[t_1 x]=[t_2]$ as CM points on $X_U$ if and only if $x\in
t_1^{-1}t_2 E\cross U$.
It follows that the coefficient of $[t_2]_U$ in $Z_*(g) t_1$ is equal to
\begin{eqnarray*}
w_U \sum_{a\in F\cross}\sum_{x \in t_1^{-1}t_2 E\cross U/U} r(g)\phi(x)_a
&=& w_U  \sum_{a\in F\cross}\sum_{y \in E\cross/(E\cross\cap U)}
r(g)\phi(t_1^{-1}t_2y)_a.
\end{eqnarray*}
Note that  $\mu_U=F\cross\cap U$ has finite index in
$E\cross\cap U$, the above becomes
\begin{eqnarray*}
&& \frac{w_U}{[E\cross\cap U:\mu_U]} \sum_{a\in F\cross}\sum_{y \in
E\cross/\mu_U} r(g,(t_1,t_2))\phi(y)_a\\
&=&\frac{1}{[E\cross\cap U:\mu_U]} \sumu \sum_{y \in
E\cross } r(g,(t_1,t_2))\phi(y,u) .
\end{eqnarray*}
The last double sum already appeared in the derivative series, and will continue to appear in local heights. So, we introduce the notation 
$$
\Omega_\phi(g,(t_1,t_2))
=\sumu \sum_{y \in
E\cross } r(g,(t_1,t_2))\phi(y,u) .
$$

Finally, we can write 
\begin{eqnarray*}
i(Z_*(g) t_1,  t_2) 
= i(Z_*(g) t_1,  t_2)_{\rm proper}
 + \frac{\Omega_\phi(g,(t_1,t_2))}{[E\cross\cap U:\mu_U]} i(t_2,t_2).\end{eqnarray*}
Here
$$
i(Z_*(g) t_1,  t_2)_{\rm proper}
=i\left(Z_*(g) t_1-\frac{\Omega_\phi(g,(t_1,t_2))}{[E\cross\cap U:\mu_U]}t_2, \  t_2\right) 
$$
is a proper intersection. 
The proper intersection has decompositions 
\begin{eqnarray*}
i(Z_*(g) t_1,  t_2)_{\rm proper}&=&
\sum_v i_v(Z_*(g) t_1,  t_2)_{\rm proper} \log N_v,
\\
i_v(Z_*(g) t_1,  t_2)_{\rm proper}&=& \barint_{C_U} 
i_{\bar v}(Z_*(g) tt_1,  tt_2)_{\rm proper} dt.
\end{eqnarray*}

We further have an identity $i(t_2,t_2)=i(1,1)$ since $[1]$ and $[t]$ are Galois conjugate CM points.

\subsection{Local heights as pseudo-theta series}
\label{local height}

Now we are going to express the local heights 
$i_{\bar v}(Z_*(g) t_1,  t_2)_{\rm proper}$ and $j_{\bar v}(Z_*(g) t_1,  t_2)$
in terms of multiplicity functions on local models of the Shimura curve. 
The idea is similar to \cite[Chapter 8]{YZZ}, with extra effort to take care of the self-intersections. Note that in \cite{YZZ}, self-intersections vanish due to a degeneracy assumption, which we cannot put here.

\subsubsection*{Archimedean case}

Let $v$ be an archimedean place.
Fix an identification $B(\af)=\bb_f$, and write $B=B(v)$.
The formula is based on the uniformization
$$X_{U,v}(\BC)=B_+\cross\bs \gh\times B\cross(\af)/U.$$
 
Resume the notations in \cite[\S 8.1]{YZZ}. 
In particular, we have the local multiplicity function
$$m_s(\gamma)=Q_s(1-2\lambda(\gamma)), \qquad \gamma\in B_v\cross-E_v\cross.$$
Here
$$Q_s(t)=\int_0^{\infty} \left(t+\sqrt{t^2-1}\cosh u\right)^{-1-s}du $$
is the Legendre function of the second kind. 
For any two distinct CM points $[\beta_1]_U, [\beta_2]_U\in
\mathrm{CM}_U$, denote
$$
g_s(\beta_1, \beta_2) =\sum_{\gamma\in \mu_U \backslash
(B_+\cross-E\cross)}  m_s(\gamma)\
 1_{U}(\beta_1^{-1}\gamma\beta_2),
$$
Then the local height has the expression 
$$i_{\bar v}(\beta_1, \beta_2)=\quasilim  g_s(\beta_1, \beta_2).$$
Here $\quasilim$ denotes the constant term at $s=0$ of $g_s((z_1,
\beta_1), (z_2, \beta_2))$, which converges for $\Re(s)>0$ and has
meromorphic continuation to $s=0$ with a simple pole.

In \cite{YZZ}, the formula works for distinct points $[\beta_1]_U$ and 
$[\beta_2]_U$. In this paper, we extend it formally to any two points. Namely, for any $\beta_1, \beta_2\in \mathrm{CM}_U$, we denote
$$
g_s(\beta_1, \beta_2) =\sum_{\gamma\in \mu_U \backslash
(B_+\cross-E\cross)}  m_s(\gamma)\
 1_{U}(\beta_1^{-1}\gamma\beta_2),
$$
and define
$$i_{\bar v}(\beta_1, \beta_2)=\quasilim  g_s(\beta_1, \beta_2).$$

With the extra new notation, we have the following result. 
\begin{pro}\label{archimedean height}
For any $t_1,t_2\in C_U$,
\begin{eqnarray*}
i_{\bar v}(Z_*(g,\phi)t_1,t_2)_{\rm proper}
= \CM^{(v)}_{\phi}(g,(t_1,t_2))-
\frac{i_{\bar v}(t_2,t_2)}{[E\cross\cap U:\mu_U]} \Omega_\phi(g,(t_1,t_2))
\end{eqnarray*}
where
\begin{eqnarray*}
\Omega_\phi(g,(t_1,t_2))
&=& \sumu \sum_{y \in
E\cross } r(g,(t_1,t_2))\phi(y,u), \\
\CM^{(v)}_{\phi}(g,(t_1,t_2))
&=&w_U \sum_{a\in F\cross}  \quasilim \sum_{y \in \mu_U \backslash (B_+\cross-E\cross)}
r(g,(t_1,t_2))\phi(y)_a   m_s(y).
\end{eqnarray*}
\end{pro}
\begin{proof}
By definition, 
$$
i_{\bar v}(Z_*(g) t_1,  t_2)_{\rm proper}
=i_{\bar v}(Z_*(g) t_1,  t_2)
-i_{\bar v}\left(\frac{\Omega_\phi(g,(t_1,t_2))}{[E\cross\cap U:\mu_U]}t_2, \  t_2\right).
$$
Here the first term on the right-hand side makes sense by the extended definition of $i_{\bar v}$ to self-intersections. 
The rest of the proof is the same as \cite[Proposition 8.1]{YZZ}.
\end{proof}

\subsubsection*{Supersingular case and superspecial case}

Let $v$ be a non-archimedean place of $F$ non-split in $E$. Let $B=B(v)$ be the nearby quaternion algebra over $F$. 
We will write the local pairing $i_{\bar v}$ as a sum of pseudo-theta series following the idea \cite{YZZ}.  The situation is more complicated by the self-intersections here. 
Note that $v$ can be either split or non-split in $\BB$, but the exposition here are the same (before going to explicit computations).

Recall from \cite[Lemma 8.2]{YZZ} that for any two distinct CM-points 
$[\beta_1]_U\in\mathrm{CM}_U$ and $[t_2]_U\in C_U$,
their local height is given by
$$i_{\bar v}(\beta_1, t_2)
=\sum _{\gamma \in \mu_U\bs B^{\times}} m(\gamma t_{2,v},
\beta_{1v}^{-1}) 1_{U^v}((\beta_1^v)^{-1}\gamma t_2^v).$$
Here the multiplicity function $m$ is defined everywhere on
$$\gh_{U_v}= B^{\times}_v\times_{E^\times_v}\bb_v^\times/U_v$$
except at the image of $(1,1)$.
It satisfies the symmetry $m(b^{-1}, \beta^{-1})=m(b, \beta)$. 

The summation is only well-defined for $[\beta_1]_U\neq [t_2]_U$. Otherwise, we can find $\gamma\in E\cross$ such that $\beta_1^{-1}\gamma
t_2\in U$, and the term at $\gamma$ is not well-defined. 
Hence, we extend the definition to any two CM-points 
$[\beta_1]_U\in\mathrm{CM}_U$ and $[t_2]_U\in C_U$
by
\begin{eqnarray*}
&&i_{\bar v}(\beta_1, t_2)\\
&=& \sum _{\gamma \in \mu_U\bs (B^{\times}-E^{\times}\cap \beta_1 U
t_2^{-1})}
m(\gamma t_{2v}, \beta_{1v}^{-1}) 1_{U^v}((\beta_1^v)^{-1}\gamma t_2^v)\\
&=& \sum _{\gamma \in \mu_U\bs (B^{\times}-E^{\times})} m(\gamma
t_{2v}, \beta_{1v}^{-1}) 1_{U^v}((\beta_1^v)^{-1}\gamma t_2^v) +\sum
_{\gamma \in \mu_U\bs (E^{\times}-\beta_1 U t_2^{-1} )}
m(\gamma t_{2v}, \beta_{1v}^{-1}) 1_{U^v}((\beta_1^v)^{-1}\gamma t_2^v)\\
&=& \sum _{\gamma \in \mu_U\bs (B^{\times}-E^{\times})} m(\gamma
t_{2v}, \beta_{1v}^{-1}) 1_{U^v}((\beta_1^v)^{-1}\gamma t_2^v) +\sum
_{\gamma \in \mu_U\bs (E^{\times}-\beta_1 U_v t_2^{-1} )} m(\gamma
t_{2v}, \beta_{1v}^{-1}) 1_{U^v}((\beta_1^v)^{-1}\gamma t_2^v).
\end{eqnarray*}
The definition is equal to the previous one if $[\beta_1]_U\neq [t_2]_U$.
In Lemma \ref{self-int}, we will see that $i_{\bar v}(t_2, t_2)$ can be realized as a proper intersection number via pull-back to $X_{U'}$ for sufficiently small $U'$ with $U'_v=U_v$. 

With the extended definition, our conclusion is as follows. 

\begin{pro}\label{supersingular height}
For any $t_1,t_2\in C_U$,
\begin{eqnarray*}
i_{\bar v}(Z_*(g,\phi)t_1,t_2)_{\rm proper}
= \CM^{(v)}_{\phi}(g,(t_1,t_2))+\CN^{(v)}_{\phi}(g,(t_1,t_2))
-\frac{i_{\bar v}(t_2,t_2)}{[E\cross\cap U:\mu_U]} \Omega_\phi(g,(t_1,t_2)),
\end{eqnarray*}
where
\begin{eqnarray*}
\Omega_\phi(g,(t_1,t_2))
&=& \sumu \sum_{y \in
E\cross } r(g,(t_1,t_2))\phi(y,u),\\
\CM^{(v)}_\phi(g,(t_1,t_2))
&=&\sum_{u\in \mu_U^2\bs F\cross}\sum _{y\in B-E}  r(g,(t_1,t_2))\phi^v (y, u)\  m_{r(g,(t_1,t_2))\phi _v}(y, u),\\
\CN^{(v)}_\phi(g,(t_1,t_2))
&=&\sum_{u\in \mu_U^2\bs F\cross}\sum _{y \in E\cross}  r(g,(t_1,t_2))\phi^v (y, u)\ r(t_1,t_2)n_{r(g)\phi _v}(y, u),
\end{eqnarray*}
and
\begin{align*}
m_{\phi _v}(y, u) =& \sum_{x\in \bb^{\times}_v/U_v} m(y, x^{-1}) \phi_v(x, uq(y)/q(x)),
\quad (y, u)\in (B_v-E_v) \times F_v ^\times; \\
n_{\phi _v}(y, u) =& \sum_{x\in (\bb^{\times}_v-y U_v)/U_v} m(y, x^{-1}) \phi_v(x, uq(y)/q(x)),
\quad (y, u)\in E^{\times}_v \times F_v ^\times.
\end{align*}
\end{pro}

\begin{proof}
By the extended definition of $i_{\bar v}$, it suffices to prove 
\begin{eqnarray*}
i_{\bar v}(Z_*(g,\phi)t_1,t_2)
= \CM^{(v)}_{\phi}(g,(t_1,t_2))+\CN^{(v)}_{\phi}(g,(t_1,t_2)).
\end{eqnarray*} 
The left-hand side is equal to 
\begin{eqnarray*}
&&  w_U\sum_{a\in F\cross}\sum_{x \in
\bb_f^\times /U} r(g)\phi(x)_a \sum _{\gamma \in \mu_U\bs
(B^{\times}-E^{\times})}
m(\gamma t_2, x^{-1}t_1^{-1}) 1_{U^v}(x^{-1}t_1^{-1}\gamma t_2)\\
&+&  w_U\sum_{a\in F\cross}\sum_{x \in \bb_f^\times /U} r(g)\phi(x)_a
\sum_{\gamma \in \mu_U\bs (E^{\times}-t_1 xU_v t_2^{-1} )} m(\gamma
t_2, x^{-1}t_1^{-1}) 1_{U^v}(x^{-1}t_1^{-1}\gamma t_2).
\end{eqnarray*}
The first triple sum is converted to $\CM^{(v)}_\phi(g,(t_1,t_2))$ as in  \cite[Proposition 8.4]{YZZ}, and the second  triple sum is converted to 
$\CN^{(v)}_\phi(g,(t_1,t_2))$ similarly. 
\end{proof}

Here we use the convention
$$r(t_1,t_2) n_{r(g)\phi _v}(y, u)=n_{r(g)\phi _v}(t_1^{-1}yt_2, q(t_1t_2^{-1})u).$$
Note that in the above series, we write the dependence on
$(t_1,t_2)$ in different manners for $m_{\phi _v}$ and
$n_{\phi_v}$. This is because $m_{\phi _v}(y, u)$ translates well
under the action of $P(F_v)\times (E_v\cross \times E_v\cross)$,
but $n_{\phi _v}(y, u)$ only translates well under the action of
$P(F_v)$.

\subsubsection*{Ordinary case}

Assume that $v$ is a non-archimedean place of $F$ split in $E$.
Then $\bb_v$ is split because of the embedding $E_v\rightarrow
\bb_v$. In this case, the treatment of \cite[\S 8.4]{YZZ} is not sufficient for our current purpose, so we write more details here. 

Let $\nu_1$ and $\nu_2$ be the two primes of $E$ lying over $v$.
Fix an identification $\bb_v \cong \mathrm{M}_2(F_v)$ under which
$E_v=\matrixx{F_v}{}{}{F_v}$. Assume that $\nu_1$ corresponds to the
ideal $\matrixx{F_v}{}{}{0}$ and $\nu_2$ corresponds to
$\matrixx{0}{}{}{F_v}$ of $E_v$.

We will make use of results of \cite{Zh}. The reduction map of
CM-points to ordinary points above $\bar\nu_1$ is given by
$$E\cross\backslash \bb_f\cross / U  \longrightarrow
E\cross\backslash (N(F_v)\backslash \gl(F_v)) \times \bb_f^{v\times} / U.$$
The intersection multiplicity is a function 
$$m_{\bar \nu_1}:\gl(F_v)/U_v\lra \QQ$$
supported on $N(F_v)U_v/U_v$ explicitly as follows. 
If $U_v=(1+p_v^r O_{\bv})^\times$ for some $r\geq 0$, then \cite[Lemma 5.5.1]{Zh} gives
$$
m_{\bar \nu_1}
\matrixx{1}{b}{}{1}=
\frac{1}{N_v^{r-v(b)-1}(N_v-1)}
$$
for $b\in F_v$ with $v(b)\leq r-1$. 
Note that the case $v(b)\geq r$ corresponds to self-intersection and is thus not well-defined. 

\begin{lem} \label{ordinary multiplicity}
The local height pairing of two distinct CM points
$[\beta_1]_U\in\CMU$ and $[t_2]_U\in C_U$ is
given by
$$i_{\bar \nu_1}(\beta_1, t_2)= \sum_{\gamma \in \mu_U\bs E\cross} m_{\bar\nu_1}(t_2^{-1}\gamma^{-1}\beta_1) 1_{U^v}(\beta_1^{-1}\gamma t_2). $$
\end{lem}
\begin{proof}
Denote the right-hand side by $i_{\bar \nu_1}(\beta_1, t_2)'$.
We first prove that $i_{\bar \nu_1}(\beta_1, t_2)=i_{\bar \nu_1}(\beta_1, t_2)'$
if $U^v$ is sufficiently small. 
In that case, by the local moduli of \cite{Zh}, 
 $i_{\bar \nu_1}(\beta_1, t_2)$ is nonzero only if there is $\gamma_0\in E\cross$ such that $\gamma_0 t_2^v U^v=\beta_1^v U^v$ and 
$t_2^{-1}\gamma_0^{-1}\beta_1\in N(F_v)U_v$. In this case, $i_{\bar \nu_1}(\beta_1, t_2)$ is equal to $m_{\bar\nu_1}(t_2^{-1}\gamma_0^{-1}\beta_1)$. 
Then it suffices to check that in the expression of $i_{\bar \nu_1}(\beta_1, t_2)'$, the summation has only one nonzero term which is exactly given by $\gamma=\gamma_0$. 
In fact, assume that $\gamma\in E^\times$ satisfies 
$$m_{\bar\nu_1}(t_2^{-1}\gamma^{-1}\beta_1) 1_{U^v}(\beta_1^{-1}\gamma t_2)\neq 0.$$
Write $\gamma=\gamma'\gamma_0$. Then the condition becomes
$$m_{\bar\nu_1}(\gamma'^{-1}t_2^{-1}\gamma_0^{-1}\beta_1) 1_{U^v}(\beta_1^{-1}\gamma_0 t_2\gamma')\neq 0.$$
It gives $\gamma'^{-1} N(F_v)U_v \subset N(F_v)U_v$ at $v$ and $\gamma'\in U^v$ outside $v$. 
The former actually implies $\gamma'\in U_v$. 
Then we have $\gamma'\in U\cap E\cross$. 
The condition that $U$ is sufficiently small implies that $U\cap E\cross=\mu_U$. 
In fact, $[U\cap E\cross:\mu_U]$ is exactly the ramification index of $[t_2]_U$. 
Hence, $\gamma=\gamma_0$ in $\mu_U\bs E\cross$. 
This proves the case that $U$ is sufficiently small. 

Now we extend the result to general $U$. 
Let $U'=U_vU'^v$ be an open compact subgroup of $\bbf$ with 
$U'^v\subset U^v$ normal. 
Assume that $U'^v$ is sufficiently small so that the lemma holds for $X_{U'}$. 
Consider the projection $\pi: X_{U'}\to X_{U}$. 
By the projection formula, we have
$$
i_{\bar\nu_1}([\beta_1]_U, [t_2]_U)
=i_{\bar\nu_1}(\pi^{-1}([\beta_1]_U), [t_2]_{U'}).
$$
To compute the right-hand side, we need to examine $\pi: X_{U'}\to X_{U}$ more carefully. 
By the right multiplication of $U$ on $X_{U'}$, it is easy to see that the Galois group of $X_{U'}\to X_U$ is isomorphic to $U/(U'\mu_U)$. 
It follows that 
$$
\pi^{-1}([\beta_1]_U)=\sum_{u \in U/(U'\mu_U)}[\beta_1 u]_{U'}
=\frac{1}{[\mu_U:\mu_{U'}]}\sum_{u \in U/U'}[\beta_1 u]_{U'}.
$$
We can further change the summation to $u \in U^v/U'^v$. 
Then
\begin{eqnarray*}
i_{\bar\nu_1}([\beta_1]_U, [t_2]_U)
&=& i_{\bar\nu_1}(\pi^{-1}([\beta_1]_U), [t_2]_{U'})\\
&=&\frac{1}{[\mu_U:\mu_{U'}]} \sum_{u\in U/U'} i_{\bar\nu_1}([\beta_1 u]_{U'}, [t_2]_{U'}) \\
&=& \frac{1}{[\mu_U:\mu_{U'}]} \sum_{u \in U^v/U'^v}  
\sum _{\gamma \in \mu_{U'}\bs E^{\times}}
m_{\bar\nu_1}(t_2^{-1}\gamma^{-1}\beta_1) 1_{U'^v}(u^{-1}\beta_1^{-1}\gamma t_2) 
\\
&=& \frac{1}{[\mu_U:\mu_{U'}]}  
\sum _{\gamma \in \mu_{U'}\bs E^{\times}}
m_{\bar\nu_1}(t_2^{-1}\gamma^{-1}\beta_1) 1_{U^v}(\beta_1^{-1}\gamma t_2) 
\\
&=& 
\sum _{\gamma \in \mu_{U}\bs E^{\times}}
m_{\bar\nu_1}(t_2^{-1}\gamma^{-1}\beta_1) 1_{U^v}(\beta_1^{-1}\gamma t_2).
\end{eqnarray*}
This finishes the general case. 
\end{proof}

Just like the other cases, the above summation is only well-defined
for $[\beta_1]_U\neq [ t_2]_U$. But we extend the definition to any 
$[\beta_1]_U$ and $[ t_2]_U$ by
\begin{eqnarray*}
i_{\bar \nu_1}(\beta_1,  t_2)
&=& \sum_{\gamma \in \mu_U\bs (E\cross-\beta_1 U  t_2^{-1})}
m_{\bar\nu_1}(t_2^{-1}\gamma^{-1}\beta_1) 1_{U^v}(\beta_1^{-1}\gamma  t_2)\\
&=& \sum_{\gamma \in \mu_U\bs (E\cross-\beta_1 U_v  t_2^{-1})}
m_{\bar\nu_1}(t_2^{-1}\gamma^{-1}\beta_1) 1_{U^v}(\beta_1^{-1}\gamma
 t_2). 
\end{eqnarray*}
It is equal to the original pairing if $[\beta_1]_U\neq [ t_2]_U$.

If $[\beta_1]_U= [t_2]_U$, then we can assume that $\beta_1=t_2$, a simple calculation taking advantage of the commutativity of $E^\times$ simply gives 
$$
i_{\bar \nu_1}(t_2,  t_2)=0, \quad \forall\ [t_2]_U\in C_U.
$$
So in this case, the definition does not give anything new. 

The results hold for $\nu_2$ by changing upper triangular matrices to lower triangular matrices. For example, 
the intersection multiplicity  $m_{\bar \nu_2}:\gl(F_v)/U_v\lra \QQ$
is supported on $N^t(F_v)U_v/U_v$ and given by 
$$
m_{\bar \nu_1}
\matrixx{1}{}{b}{1}=
\frac{1}{N_v^{r-v(b)-1}(N_v-1)}
$$
for $b\in F_v$ with $v(b)\leq r-1$. 
Then we also have a similar extension for $i_{\bar \nu_1}(\beta_1,  t_2)$. 

Passing to $\bar v$, we have 
$$
m_{\bar v}=\frac12(m_{\bar\nu_2}+m_{\bar\nu_2}), \quad
i_{\bar v}=\frac12(i_{\bar \nu_1}+i_{\bar \nu_2}). 
$$
Now we have the following result. 

\begin{pro} \label{ordinary height}
For any $t_1,t_2\in C_U$,
\begin{eqnarray*}
i_{\bar v}(Z_*(g,\phi)t_1,t_2)_{\rm proper}
=  \CN^{(v)}_{\phi}(g,(t_1,t_2)),
\end{eqnarray*}
where
\begin{eqnarray*}
\CN^{(v)}_\phi(g,(t_1,t_2))
&=&\sum_{u\in \mu_U^2\bs F\cross}\sum _{y \in E\cross}  r(g,(t_1,t_2))\phi^v (y, u)\ r(t_1,t_2)n_{r(g)\phi _v}(y, u),
\end{eqnarray*}
and
\begin{align*}
n_{\phi _v}(y, u)
=&\ \frac{1}{2} \sum_{x_v \in (N(F_v)U_v-U_v)/U_v} \phi_v(y x_v,u) \  m_{\bar \nu_1}(x) \\
&+ \frac{1}{2} \sum_{x_v \in (N^t(F_v)U_v-U_v)/U_v} \phi_v(y x_v,u) \  m_{\bar \nu_2}(x)
\end{align*}
for any $(y,u)\in E_v\cross\times F_v\cross$.
\end{pro}
\begin{proof}
Note that the extended intersection number $i_{\bar v}(t_2,t_2)=0$ automatically. 
It suffices to check
\begin{eqnarray*}
i_{\bar v}(Z_*(g,\phi)t_1,t_2)
=  \CN^{(v)}_{\phi}(g,(t_1,t_2)).
\end{eqnarray*}
The left-hand side is equal to
\begin{align*}
w_U \sum _{a\in F\cross}  \sum_{x \in \bb_f^\times /U} r(g)\phi(x)_a
\sum_{\gamma \in \mu_U\bs (E\cross-t_1x U_v t_2^{-1})} 
m_{\bar v}(t_2^{-1}\gamma^{-1}t_1x) 
1_{U^v}(x^{-1}t_1^{-1}\gamma t_2).
\end{align*}
By $1_{U^v}(x^{-1}t_1^{-1}\gamma t_2)=1$, we have $x^v \in t_1^{-1}\gamma t_2 U^v$; by
$\gamma \notin t_1x U_v t_2^{-1}$, we have $x_v\notin t_1^{-1}\gamma t_2 U_v$. Thus it becomes
\begin{align*}
 w_U  \sum _{a\in F\cross} \sum_{\gamma \in \mu_U\bs E\cross}  r(g)\phi^v(t_1^{-1}\gamma t_2)_a
\sum_{x_v \in (\bb_v^{\times}-t_1^{-1}\gamma t_2 U_v)/U_v} r(g)\phi_v(x_v)_a \  m_{\bar v}(t_2^{-1}\gamma^{-1}t_1x) .
\end{align*}
It remains to convert the last sum to the desired form, which is reduced to 
similar results for $\nu_1$ and $\nu_2$. 
We have
\begin{align*}
&\sum_{x_v \in (\bb_v^{\times}-t_1^{-1}\gamma t_2 U_v)/U_v} r(g)\phi_v(x_v)_a \  m_{\bar \nu_1}(t_2^{-1}\gamma^{-1}t_1x)\\
=&\sum_{x_v \in (t_1^{-1}\gamma t_2 N(F_v)U_v-t_1^{-1}\gamma t_2 U_v)/U_v} r(g)\phi_v(x_v)_a \  m_{\bar \nu_1}(t_2^{-1}\gamma^{-1}t_1x)\\
=&\sum_{x_v \in (N(F_v)U_v-U_v)/U_v} r(g)\phi_v(t_1^{-1}\gamma t_2 x_v)_a \  m_{\bar \nu_1}(x).
\end{align*}
A similar result holds for $\nu_2$.  
\end{proof}

\subsubsection*{Decomposition of the height series}

Finally, we end up with the following summary. 
\begin{thm} \label{height series}
Assume that Assumption \ref{assumption3} holds. 
Then for any $t_1,t_2\in C_U$,
\begin{align*}
Z(g, (t_1,t_2),\phi))_U
=& -\sum_{v\ \nonsplit} (\log N_v) \barint_{C_U}
\CM^{(v)}_{\phi}(g,(tt_1,tt_2)) dt\\
& -\sum_{v\nmid\infty} \CN^{(v)}_{\phi}(g,(t_1,t_2)) \log N_v-\sum_{v\nmid\infty} j_v(Z_*(g,\phi)t_1,t_2)\log N_v\\
& -\frac{i_0(t_2,t_2)}{[E\cross\cap U:\mu_U]} \Omega_\phi(g,(t_1,t_2)).
\end{align*}
The right-hand side is explained in the following. 
\begin{itemize}

\item[(1)]
The modified arithmetic self-intersection number
$$i_0(t_2,t_2)=i(t_2,t_2)- \sum_{v} i_v(t_2,t_2) \log N_v,$$
where the local term 
$$
i_v(t_2,t_2)=\barint_{C_U}  i_{\bar v}(tt_2,tt_2) dt
$$
uses the extended definition of $i_{\bar v}$.
\item[(2)]
The pseudo-theta series
\begin{eqnarray*}
\Omega_\phi(g,(t_1,t_2))
= \sumu \sum_{y \in
E\cross } r(g,(t_1,t_2))\phi(y,u).
\end{eqnarray*}

\item[(3)]
For any place $v$ non-split in $E$,
\begin{eqnarray*}
\CM^{(v)}_{\phi}(g,(t_1,t_2))
&=&w_U \sum_{a\in F\cross}  \quasilim \sum_{y \in \mu_U \backslash (B_+\cross-E\cross)}
r(g,(t_1,t_2))\phi(y)_a   m_s(y), \quad v|\infty, \\
\CM^{(v)}_\phi(g,(t_1,t_2))
&=&\sum_{u\in \mu_U^2\bs F\cross}\sum _{y\in B-E}  r(g,(t_1,t_2))\phi^v (y, u)\  m_{r(g,(t_1,t_2))\phi _v}(y, u),\quad v\nmid \infty.
\end{eqnarray*}

\item[(4)]
For any non-archimedean $v$,
\begin{eqnarray*}
\CN^{(v)}_\phi(g,(t_1,t_2))
&=&\sum_{u\in \mu_U^2\bs F\cross}\sum _{y \in E\cross}  r(g,(t_1,t_2))\phi^v (y, u)\ r(t_1,t_2)n_{r(g)\phi _v}(y, u),
\end{eqnarray*}
\end{itemize}
\end{thm}

The only new information used above is the identity 
$$
\barint_{C_U}
\CN^{(v)}_{\phi}(g,(tt_1,tt_2)) dt
=\CN^{(v)}_{\phi}(g,(t_1,t_2)).
$$
This follows from the invariance 
$$
\CN^{(v)}_{\phi}(g,(tt_1,tt_2))
=\CN^{(v)}_{\phi}(g,(t_1,t_2)),
$$
which in turn follows from the special situation that the summation only involves $y\in E^\times$  in the definition of $\CN^{(v)}_{\phi}$.

\subsection{Explicit local heights}

Let $(U,\phi, \jv, j_v)$ be as in \S \ref{choices}. 
The goal of this subsection is to compute $m_{\phi_v}(y,u)$ and $n_{\phi_v}(y,u)$, and treat $j_v(Z_*(g,\phi)t_1,t_2)$.
The results are parallel to those in \S \ref{explicit derivative}. 

\subsubsection*{Local intersection numbers}

\begin{lem}    \label{local intersection}
\begin{itemize}
\item[(1)]
Let $v$ be a non-archimedean place nonsplit in $E$.
For any $(y,u)\in (B(v)_v-E_v)\times \fvcross$,
\begin{align*}
& m_{\phi _v}(y, u)=
 \begin{cases}
\displaystyle  \phi_{v}(y_1,u) 1_{\oev j_v}(y_2)
\cdot \frac{1}{2} (v(q(y_2))+1),
  \quad & \BB_v \mathrm{\ split},\ E_v \mathrm{\ inert};\\
  \displaystyle  \phi_{v}(y_1,u) 1_{\oev j_v}(y_2)
\cdot \frac{1}{2} (v(q(y_2))+v(D_v)),
  \quad & \BB_v \mathrm{\ split},\ E_v \mathrm{\ ramified};\\
\displaystyle  \phi_{v}(y_1,u) 1_{\oev j_v}(y_2)
\cdot \frac{1}{2} v(q(y_2)),
  \quad & \BB_v \mathrm{\ nonsplit}.
  \end{cases}
\end{align*}

\item[(2)]
Let $v$ be a non-archimedean place of $F$.
For any $(y,u)\in E_v\cross\times \fvcross$,
\begin{align*}
 n_{\phi _v}(y, u)=  \phi_{v}(y,u)\cdot \frac{1}{2} v(q(y)).
\end{align*}

\end{itemize}
\end{lem}
\begin{proof}
If $v$ is nonsplit in $E$, by Proposition \ref{supersingular height}, 
\begin{align*}
m_{\phi _v}(y, u) =& \sum_{x\in \bb^{\times}_v/U_v} m(y, x^{-1}) \phi_v(x, uq(y)/q(x)),
\quad (y, u)\in (B_v-E_v) \times F_v ^\times; \\
n_{\phi _v}(y, u) =& \sum_{x\in (\bb^{\times}_v-y U_v)/U_v} m(y, x^{-1}) \phi_v(x, uq(y)/q(x)),
\quad (y, u)\in E^{\times}_v \times F_v ^\times.
\end{align*}

If $v$ is nonsplit in $E$ and split in $\BB$, then (1) is computed in \cite[Proposition 8.7]{YZZ}, except that there is a mistake in the case that $E_v$ is wildly ramified over $F_v$. 
The mistake came from \cite[Lemma 8.6]{YZZ}, which was in turn caused by the wrong formula of \cite[Lemma 5.5.2]{Zh}. 
As a digression, we remark that the mistake did not impact the main result of \cite{YZZ} because the result in this case was not used in the book elsewhere. 

The correct version of \cite[Lemma 8.6]{YZZ} is as follows. 
The multiplicity function $m(b,\beta)\neq 0$ only if $q(b)q(\beta)\in
\ofv\cross$. In this case, assume that
$\beta \in E_v^\times h_c\GL_2(O_{F_v})$. Then
$$m(b, \beta)=\begin{cases}
\frac 12(v(\lambda(b))+1) &\text{if $c=0$, $E_v/F_v$ is unramified;} \\
\frac 12 v(D_v\lambda(b)) &\text{if $c=0$, $E_v/F_v$ is ramified;} \\
N_v^{1-c}(N_v+1)^{-1}&\text{if $c>0$, $E_v/F_v$ is unramified;} \\
\frac 12 N_v^{-c}&\text{if $c>0$, $E_v/F_v$ is ramified.}
\end{cases}
$$
Only the second case is different, and it can be verified
by going back to the canonical lifting of Gross \cite{Gr1}. 
Then it is easy to have the correct formula (1) of the current case. 

If $v$ is nonsplit in $E$ and split in $\BB$, then (2) can be verified by the same method as in \cite[Proposition 8.7]{YZZ}, where the only difference is that 
$$n_{\phi_v}(y, u)=\sum _{c=1}^\infty  m(y^{-1}, h_c)\vol (E_v^\times h_c\GL_2(O_{F_v})\cap M_2(O_{F_v})_n)$$
is a sum omitting $c=0$.

If $v$ is inert in $E$ and nonsplit in $\BB$, by Lemma \ref{superspecial multiplicity}, 
$$m(y, x^{-1})
=\frac 12 v(\lambda(y))\ 1_{E_v^\times  (1+O _{E_v}\varpi_v j_v)}(y)1_0(v(q(x)/q(y))).$$
It follows that
\begin{align*}
m_{\phi _v}(y, u)
=& \frac 12 v(\lambda(y))\ 1_{E_v^\times (1+O _{E_v}\varpi_v j_v)}(y) \sum_{x\in \bb^{\times}_v/U_v} 1_0(v(q(x)/q(y))) \phi_v(x, uq(y)/q(x)).
\end{align*}
Note that $\bb^{\times}_v/U_v\cong \ZZ$. It is easy to get (1). 
For (2), since the conditions $x\notin yU_v$ and $1_0(v(q(x)/q(y)))$ are contradictory, we get $n_{\phi _v}(y, u)=0$.

If $v$ is split in $E$, in the setting of Proposition \ref{ordinary height},
\begin{align*}
n_{\phi _v}(y, u)
=&\ \frac{1}{2} \sum_{x_v \in (N(F_v)-N(O_{F_v}))/N(O_{F_v})} \phi_v(y x_v,u) \  m_{\bar \nu_1}(x) \\
&+ \frac{1}{2} \sum_{x_v \in (N^t(F_v)-N^t(O_{F_v}))/N^t(O_{F_v})} \phi_v(y x_v,u) \  m_{\bar \nu_2}(x).
\end{align*}
We first consider the case $v\notin S_2$. 
Then $\phi_v$ is the standard characteristic function. 
Write $y=\matrixx{a}{}{}{d}$. 
The summations are nonzero only if $a,d\in O_{F_v}$ and $u\in \ofv^\times$, which we assume.
For the first sum, write $x_v=\matrixx{1}{b}{}{1}$. 
Then we need $ab\in \ofv$. Eventually, the first sum becomes
\begin{align*}
 \sum_{b \in (a^{-1}O_{F_v}-O_{F_v})/O_{F_v}} 
\frac{1}{N_v^{-v(b)-1}(N_v-1)}  
= \sum_{i=1}^{v(a)} 
\frac{|(p_v^{-i}-p_v^{-i+1})/O_{F_v}|}{N_v^{i-1}(N_v-1)}  
=v(a). 
\end{align*}
Similarly, the second sum equals $v(d)$. 
Then 
$$
n_{\phi _v}(y, u)
= \frac{1}{2}(v(a)+v(d))=\frac{1}{2}v(q(y)).
$$
This finishes the proof for $v\notin S_2$. 
If $v\in S_2$, the computation is similar, and we will get everywhere 0. 
\end{proof}

\subsubsection*{Multiplicity function: superspecial case}

Let $v$ be non-archimedean place nonsplit in $\BB$ and inert in $E$. 
Recall that the multiplicity function $m$ is defined on
$$\gh_{U_v}= B(v)^{\times}_v\times_{E^\times_v}\bb_v^\times/U_v.$$
Note that we have assumed that $U_v$ is maximal. The following result does not need any restriction on $U^v$.

\begin{lem} \label{superspecial multiplicity}
For any $(\gamma, \beta)\in B(v)^{\times}_v\times_{E^\times_v}\bb_v^\times$, we have $m(\gamma, \beta)\neq 0$ only if $q(\gamma )q(\beta)\in \ofv\cross$ and $\gamma\in E_v^\times \cdot (1+O _{E_v}\varpi_v j_v)$. In this case, 
$$
m(\gamma, \beta)
=\frac 12v(\lambda(\gamma)).
$$
Here $\lambda(\gamma)=q(\gamma_2)/q(\gamma)$, where $\gamma=\gamma_1+\gamma_2$ is the decomposition according to $B_v=E_v+E_vj_v$. 
\end{lem}

Instead of deformation theory, our proof uses directly the theorem of $p$-adic uniformization of \v Cerednik \cite{Ce}. See also \cite{BC}.

Write $B=B(v)$ for simplicity. 
Denote by $F_v^\ur$ the completion of the maximal unramified extension of $F_v$, and $\BC_v$   the completion of the algebraic closure of $F_v$.
The $p$-adic uniformization in terms of rigid-analytic space is
$$X_{U}^\an\times_{F_v} F_v^\ur =B^{\times}\bs (\Omega \times_{F_v} F_v^\ur)\times \BB_f\cross/U.$$
Here $\Omega$ is the Drinfe'ld (rigid-analytic) upper half plane over $F_v$, which gives 
$\Omega (\BC _v) =\BC _v-F_v$. 
The group $B_v\cross\cong \GL
_2(F_v)$ acts on $\Omega$ by the fractional linear transformation, and on 
$\bb_v^\times/U_v\cong \ZZ$ via translation by $v\circ q=v\circ\det$.

To study the intersection multiplicity, we need the integral version of the uniformization.  The uniformization theory also gives a canonical integral model $\wh \Omega$ of $\Omega$. It is a formal scheme over $O_{F_v}$ obtained from successive blowing-up of rational points on the special fiber of $\BP _{O_{F_v}}$ constructed by Deligne. 
The uniformization takes the form:
$$\wh\CX_{U}\times_{\Spf\, O_{F_v}} \Spf\, O_{F_v^\ur} =B^{\times}\bs 
(\wh\Omega \times_{\Spf\, O_{F_v}} \Spf\, O_{F_v^\ur})\times \BB_f\cross/U.$$
Here $\CX_U$ is the canonical integral model over $O_F$, which is semistable at $v$, and $\wh\CX_U$ denotes the formal completion along the special fiber above $v$. 

The special fiber of $\wh \Omega$, or equivalently the underlying topological space of $\wh \Omega$, is a union of $\BP^1$'s indexed by scalar equivalence
classes of $O_{F_v}$-lattices of $F_v^2$.  
Then its irreducible
components are indexed by
$$\GL _2(F_v)/F_v^\times \GL _2(O_{F_v}).$$
It follows that the irreducible components of the special fiber of $\CX_{U}$ above $v$ are indexed by
$$B^{\times}\bs (\GL _2(F_v)/F_v^\times \GL_2(O_{F_v})) 
\times \BB_f^{\times}/U.$$

Consider the set 
$$\mathrm{CM}_U=E^\times \bs B\cross(\BA_f)/U=B^{\times}\bs (B^{\times}\times _{E^\times }\bb_v\cross/U_v)
\times B^{\times}(\BA_f^v)/U^v.$$ 
The natural embedding $\mathrm{CM}_U\to X_U(\BC_v)$ is given by the embedding
$$B^{\times}\times _{E^\times }\bb_v\cross/U_v\lra \Omega \times \BZ, \quad (\gamma, \beta)\longmapsto (\gamma z_0, v(q(\gamma )q(\beta))),$$
where $z_0\in \Omega (E_v)$ is the unique point in $\Omega (\BC_v)$
fixed by $E_v^\times$. Thus the CM-points on $\Omega$ are given by
$$\gh_{U_v}^\circ=\left\{(\gamma , \beta)\in B_v\cross\times _{E_v\cross}\bb_v\cross/U_v: v(q(\gamma )q(\beta))=0\right\}.$$
As $U_v$ is maximal, the class of $(\gamma , \beta)$ in $\gh _{U_v}^\circ$ it determined by $\gamma $. Thus $\gh_{U_v}^\circ$ can be identified with
$$B_v\cross/E_v^\times =B_v\cross z_0.$$ 
Then we have a multiplicity function $m$ on $B_v\cross/E_v^\times$ such that
$$m(\gamma , \beta)=m(\gamma )1_0(v(q(\gamma )q(\beta))), \qquad \gamma \in B_v^{\times},\ \beta\in \bv\cross.$$
The problem is reduced to compute $m(\gamma)$, which is the intersection number of $z_0$ with $\gamma z_0$ on the special fiber. 

The intersection number is on 
$\wh\Omega \times_{\Spf\, O_{F_v}} \Spf\, O_{F_v^\ur}$. 
Since the irreducible components of its special fiber are indexed by
$\GL_2(F_v)/F_v^\times \GL_2(O_{F_v})$,
we see that $m(\gamma)$ is nonzero only if $\gamma$ lies in $\GL_2(F_v)/F_v^\times \GL_2(O_{F_v})$. 
Then we can assume that $\gamma \in \GL_2(O_{F_v})$, since the center acts trivially on $z_0$. 

By the assumption, $z_0$ and $\gamma z_0$ reduce to the same irreducible component on the special fiber of $\wh\Omega\times_{\Spf\, O_{F_v}} \Spf\, O_{F_v^\ur}$. 
Remove the other irreducible components of $\wh\Omega\times_{\Spf\, O_{F_v}} \Spf\, O_{F_v^\ur}$. We obtain a formal scheme, which is just the formal completion of  
$\BP^1_{O_{F_v^\ur}}-\BP^1(k_v)$ along the special fiber. 
Here $k_v$ denotes the residue field of $O_{F_v}$, and the $k_v$-points on the special fiber are removed.
Now the problem is elementary: $z_0$ and $\gamma z_0$ are points of 
$\BP^1_{O_{F_v^\ur}}$, and the goal is to find their intersection number on the special fiber. 
We further replace $\BP^1_{O_{F_v^\ur}}$ by $\BP^1_{O_{E_v}}$, which does not change the intersection number. 

The point $z_0\in \BP^1(O_{E_v})$ corresponds to an $O_{F_v}$-linear isomorphism $\ell_0: O_{F_v}^2\to O _{E_v}$, which is determined by $z_0$ up to 
$O_{E_v}^\times$-action. 
Then $\gamma z_0$ corresponds to the isomorphism 
$\ell_0\circ \gamma: O_{F_v}^2\to O _{E_v}$.
We need to find the maximal integer $n$ such that $\ell_0$ and $\ell_0\circ \gamma$ reduce to the same point in $\BP^1(O_{E_v}/p_v^n)$.
Identify $E_v$ with $F_v^2$ by $\ell_0$, so that $M_2(F_v)$ acts on $E_v$.
The action is compatible with the embedding $E\hookrightarrow B(v)$ we specify at the very beginning because $z_0$ is the fixed point of $E_v^\times$. 
Then the problem becomes finding the maximal integer $m$ such that the image of $\gamma$ in $\gl(O_{F_v}/p_v^n)$ actually lies in $(O_{E_v}/p_v^n)^\times$. 

Write $\gamma=a+bj_v$ according to the orthogonal decomposition $M_2(F_v)=E_v+E_vj_v$. Here $q(j_v)\in O_{F_v}^\times$ by assumption.
Some $O_{E_v}^\times$-multiple of $j_v$ acts on $E_v$ by the nontrivial element of $\Gal(E_v/F_v)$. 
Hence, $m(\gamma)\neq 0$ only if $a\in O_{E_v}^\times$ and $b\in p_vO_{E_v}$. In that case, $m(\gamma)=v(b)$. 

Go back to an arbitrary $\gamma\in \gl(F_v)$. We have 
$m(\gamma)\neq 0$ only if 
$\gamma\in E_v^\times \cdot (1+O _{E_v}\varpi_v j_v)$.
In that case, 
$m (\gamma)=v(\lambda(\gamma))/2.$

\subsubsection*{The $j$-part}

If $v$ is a non-archimedean place of $F$ split in $\BB$, then the $j$-part $j_{v}(Z_*(g,\phi)t_1,t_2)=0$ automatically. 
This is a trivial consequence of the fact that the special fiber of $X_U$ at $v$ is a disjoint union of irreducible curves. 
For the fact, in the construction of $\CX_U$ before Corollary \ref{system2}, we can take the prime $p$ to be coprime to $v$, then $\CX_{U'}$ is smooth at $v$. 
The special fiber of $\CX_{U'}$ at $v$ is a disjoint union of irreducible curves, and
the quotient $\CX_U$ has the same property since it is also a quotient of the underlying topological space. 

In the following, assume that $v$ is a non-archimedean place nonsplit in $\BB$ and inert in $E$.  
Note that $U_v$ is maximal and $\phi_v=1_{O_\bv^\times\times\ofv\cross}$. 
It is proved that the $j$-part $j_{v}(Z_*(g)t_1,t_2)$ is a non-singular pseudo-theta series in \cite{YZZ} under \cite[Assumption 5.3]{YZZ}. 
The result is also true in the current situation. 
Recall that
$$j_{v}(Z_*(g)t_1,t_2)=\barint_{C_U}  j_{\bar v}(Z_*(g) tt_1,  tt_2)dt.$$
The integration is a finite sum, so it suffices to prove the same result for 
$j_{\bar v}(Z_*(g)t_1,t_2)$.

\begin{lem} \label{superspecial j}
Let $v$ be a non-archimedean place nonsplit in $\BB$ and inert in $E$.  
The $j$-part $j_{\bar v}(Z_*(g,\phi)_Ut_1,t_2)$ is a non-singular pseudo-theta series of the form 
$$
\sum_{u\in \mu_U^2\backslash F\cross} \sum_{y\in B(v)-\{0\}}
  r(g) \phi^v(y,u)\ l_{r(g)\phi_v}(y,u).
$$
\end{lem}

\begin{proof}
Resume the notations of Lemma \ref{superspecial multiplicity}.
As above, denote by $F_v^\ur$ the completion of the maximal unramified extension of $F_v$. As all CM points of $\CMU$ are defined over $F_v^\ur$, the intersection number $j_{\overline v}(Z_*(g)t_1 , t_2)$ can be computed on the integral model $\CX_{U,O_{F_v^\ur}}$.
By the definition in \cite[\S7.1.7]{YZZ},
$$j_{\overline v}(Z_*(g)t_1, t_2)= \overline{Z_*(g)t_1}\cdot
V_{t_2}.$$
Here $\overline{Z_*(g)t_1}$ is the Zariski closure in $\CX_{U,O_{F_v^\ur}}$, and $V_{t_2}$ is a vertical divisor on $\CX_{U,O_{F_v^\ur}}$, i.e., a  linear combination of irreducible components in the special fibers of $\CX_{U,O_{F_v^\ur}}$ which gives the $\hat \xi$-admissible arithmetic extension of $t_2$.

We still use the $p$-adic uniformization
$$\wh\CX_{U}\times_{\Spf\, O_{F_v}} \Spf\, O_{F_v^\ur} =B^{\times}\bs 
(\wh\Omega \times_{\Spf\, O_{F_v}} \Spf\, O_{F_v^\ur})\times \BB_f\cross/U.$$
Here $B=B(v)$ as before. 
The map from $\wh\CX_{U}\times_{\Spf\, O_{F_v}} \Spf\, O_{F_v^\ur}$ 
to its set of connected components is exactly the natural composition
$$
B^{\times}\bs 
(\wh\Omega \times_{\Spf\, O_{F_v}} \Spf\, O_{F_v^\ur})\times \BB_f\cross/U 
\lra B^{\times}\bs \BB_f\cross/U
\stackrel{q}{\lra} F_+^\times\bs \afcross/ q(U).
$$

For the case $t_2=1$, write $V_{1}=\sum_i a_i W_i$, where $\{W_i\}_i$ is the set of irreducible components of the special fiber of $\CX_{U,O_{F_v^\ur}}$ lying in the same connected component as $1$. 
Let $\wt W_i$ be an irreducible component of the special fiber of $\wh\Omega \times_{\Spf\, O_{F_v}} \Spf\, O_{F_v^\ur}$ lifting $W_i$.
Note that the choice of $\wt W_i$ is not unique, but we fix such choice. 
Write $\wt V=\sum_i a_i \wt W_i$, viewed as a vertical divisor of 
$\wh\Omega \times_{\Spf\, O_{F_v}} \Spf\, O_{F_v^\ur}$. 
The vertical divisor $(\wt V,1)=\sum_i a_i (\wt W_i, 1)$ of $(\wh\Omega \times_{\Spf\, O_{F_v}} \Spf\, O_{F_v^\ur})\times \BB_f\cross/U$
is a lifting of the vertical divisor $V_{1}=\sum_i a_i W_i$. 

For general $t_2\in \afcross$, the vertical divisor $(\wt V, t_2)=\sum_i a_i (\wt W_i, t_2)$ of $(\wh\Omega \times_{\Spf\, O_{F_v}} \Spf\, O_{F_v^\ur})\times \BB_f\cross$ is a lifting of the vertical divisor $V_{t_2}$. In fact, by the projection formula, it suffices to verify the intersection number of $(\wt V, t_2)$ with any
$B^\times$-invariant vertical divisors of 
$(\wh\Omega \times_{\Spf\, O_{F_v}} \Spf\, O_{F_v^\ur})\times \BB_f\cross/U$
are the expected ones. 
But these intersection numbers are given by the corresponding ones from the case $t_2=1$. 

For any point $\beta\in \CMU$, the projection formula gives
$$\overline{\beta}\cdot V_{t_2}
=\sum_{\gamma \in \mu_U\bs B^\times}
(\gamma^{-1} z_0\cdot \wt V)1_{O_{F_v}^\times}(q(\gamma)q(t_2)/q(\beta)) 1_{U^v}(t_2^{-1}\gamma^{-1} \beta).$$
Here $z_0\in \wh\Omega(O_{F_v^\ur})$ is the unique fixed section of $E_v^\times$, and 
the intersection $(\gamma^{-1} z_0\cdot \wt V)$  is taken on 
$\wh\Omega \times_{\Spf\, O_{F_v}} \Spf\, O_{F_v^\ur}$.

Hence, as in all the previous cases of local heights, we have
\begin{align*}
\overline{Z_*(g)t_1}\cdot V_{t_2}
=& w_U \sum_{a\in F\cross}\sum_{x \in\bb_f^\times /U} r(g)\phi(x)_a  \sum_{\gamma \in \mu_U\bs B^\times}
(\gamma^{-1} z_0\cdot \wt V)1_{O_{F_v}^\times}(q(\gamma)q(t_2)/q(t_1x)) 1_{U^v}(t_2^{-1}\gamma^{-1} t_1x)\\
=& w_U \sum_{a\in F\cross} \sum_{\gamma \in \mu_U\bs B^\times}
 r(g)\phi^v(t_1^{-1}\gamma t_2)_a  
\sum_{x \in\bb_v^\times /U_v} r(g)\phi_v(x)_a
(\gamma^{-1} z_0\cdot \wt V)
1_{O_{F_v}^\times}(q(t_1^{-1}\gamma t_2)/q(x)) 
\\
=& \sumu \sum_{\gamma \in B^\times}
 r(g,(t_1,t_2))\phi^v(\gamma, u)\ r(t_1,t_2) l_{r(g)\phi_v}(\gamma,u),
\end{align*}
where
$$l_{\phi_v}(\gamma,u)=
\sum_{x \in\bb_v^\times /U_v} \phi_v(x, uq(\gamma)/q(x))
1_{O_{F_v}^\times}(q(x)/q(\gamma))\ (\gamma^{-1} z_0\cdot \wt V).
$$
Here we have used $(t_2^{-1}\gamma^{-1} t_1 z_0\cdot \wt V)=(\gamma^{-1} z_0\cdot \wt V)$, which is explained as follows.
In fact, $t_1z_0=z_0$ by definition. For $t_2$, since the intersection number is invariant under the action of $B_v^\times$, we have
$(t_2^{-1}\gamma^{-1} z_0\cdot \wt V)=(\gamma^{-1} z_0\cdot t_2\wt V)$.
But then $t_2\wt V=\wt V$ since $t_2\in F_v^\times \gl(\ofv)$ fixes every irreducible component of the special fiber of 
$\wh\Omega \times_{\Spf\, O_{F_v}} \Spf\, O_{F_v^\ur}$.

Hence, the intersection number $j_{\overline v}(Z_*(g)t_1, t_2)$ is a pseudo-theta series. 
It remains to prove that the function 
$$l_{\phi_v}(\gamma,u)=
\sum_{x \in\bb_v^\times /U_v} \phi_v(x, uq(\gamma)/q(x))
1_{O_{F_v}^\times}(q(x)/q(\gamma))\ (\gamma^{-1} z_0\cdot \wt V), \quad (\gamma,u)\in B_v^\times\times F_v^\times
$$
extends to a Schwartz function of $B_v\times F_v^\times$.
The function is locally constant on $B_v^\times\times F_v^\times$, and we need to prove that its support is actually compactly supported in $B_v^\times\times F_v^\times$. 
In order for the contribution of $x \in\bb_v^\times /U_v$ to the summation to be nonzero, we need 
$$x\in O_{\bb_v}^\times, \quad
uq(\gamma)/q(x)\in O_{F_v}^\times, \quad
q(x)/q(\gamma) \in O_{F_v}^\times.
$$
It follows that 
$$l_{\phi_v}(\gamma,u)= 
(\gamma^{-1} z_0\cdot \wt V)\cdot 1_{O_{F_v}^\times}(q(\gamma)) \cdot 1_{O_{F_v}^\times}(u).$$
In particular, it is already compactly supported in $u$. 

To get extra information on $\gamma$, go back to the uniformization. 
Note that the irreducible components of the special fiber of  
$\wh\Omega \times_{\Spf\, O_{F_v}} \Spf\, O_{F_v^\ur}$ are indexed by 
$$
\gl(F_v)/F_v^\times \gl(\ofv).
$$
Denote by $\alpha_i F_v^\times \gl(\ofv)$ the coset representing the component $W_i$ of $\wt V=\sum_i a_iW_i$.
Then we simply have  
$$
\gamma^{-1} z_0\cdot \wt V
=\sum_{i} a_i1_{\alpha_iF_v^\times \gl(\ofv)}(\gamma^{-1}).
$$
Combining with 
$q(\gamma)\in O_{F_v}^\times$, 
we conclude that the support of $\gamma$ in $l_{\phi_v}(1,\gamma,u)$ is the union of finitely many cosets of $\gl(\ofv)$. 
This finishes the proof. 
\end{proof}

\begin{remark}
As we can see from the proof, the result holds under the more general condition that $\phi_v(0,u)=0$. This condition is weaker than \cite[Assumption 5.3]{YZZ}. 
\end{remark}

\section{Quaternionic height}

In this section, we will combine results in the last two sections to prove Theorem \ref{quaternion main}.
We will prove a formula for the modified self-intersection $i_0(1,1)$ by applying  Lemma \ref{pseudo} (2) to the difference
$$\mathcal D(g,\phi)=\pr I'(0, g, \phi)_U- 2 Z(g,(1,1))_U.  $$
Then we will connect $i_0(1,1)$ to the height of CM points defined by arithmetic Hodge bundles 
by proving an adjunction formula.

\subsection{Derivative series vs. height series}

Let $(F,E,\BB, U,\phi)$ be as in \S \ref{choices}. 
By comparing the height series and the derivative series, 
 we will show a formula of the modified self-intersection
$$i_0(P,P)=i_0(1,1)=i(1,1)- \sum_{v} i_v(1,1) \log N_v.$$
Here $i(1,1)$ represents the horizontal arithmetic intersection of the CM point $[1]_U\in C_U$ with itself, while the local term 
$$
i_v(1,1)=\barint_{C_U}  i_{\bar v}(t,t) dt
$$
uses the extended definition of $i_{\bar v}(t,t)$
introduced in \S \ref{local height} case by case. 

The following is the main theorem of this section. 

\begin{thm}\label{result of comparison}
$$\frac{1}{[O_E\cross:O_F\cross]}i_0(P,P)
=\frac{L_f'(0,\eta)}{L_f(0,\eta)} +\frac 12 \log(\frac{d_{E/F}}{d_\BB}).
$$
\end{thm}

The theorem is already very close to Theorem \ref{quaternion main}. 
The bridge between these two theorems is the arithmetic adjunction formula in 
Theorem \ref{adjunction all places}.

\subsubsection*{The comparison}

Let $(\BB, U,\phi)$ be as in \S \ref{choices}. 
Go back to 
$$\mathcal D(g,\phi)=\pr I'(0, g, \phi)_U- 2 Z(g,(1,1))_U.  $$
By Theorem \ref{derivative series}, 
\begin{align*}
\pr I'(0,g,\phi)_U
=& -\sum_{v|\infty}2\barint_{C_U}
\overline\CK^{(v)}_{\phi}(g,(t,t)) dt-\sum_{v\nmid\infty \
\nonsplit}2\barint_{C_U}
\CK^{(v)}_{\phi}(g,(t,t)) dt\\
&+ \sumu \sum_{y\in E^\times} (2\log \delta_f(g_f)+ \log|uq(y)|_f)\ r(g)\phi(y,u)\\
& -\sum_{v\nmid\infty} \sumu\sum_{y\in E^\times}
c_{\phi_v}(g,y,u)\, r(g)\phi^v(y,u)\\
& -c_1 \Omega_\phi(g).
\end{align*}
Here
$$ c_1=2\frac{L_f'(0,\eta)}{L_f(0,\eta)} +\log\frac{d_E}{d_F}$$
and
\begin{eqnarray*}
\Omega_\phi(g)
= \sumu \sum_{y \in
E\cross } r(g)\phi(y,u).
\end{eqnarray*}

By Theorem \ref{height series},
\begin{align*}
Z(g, (1,1),\phi))_U
=& -\sum_{v\ \nonsplit} (\log N_v) \barint_{C_U}
\CM^{(v)}_{\phi}(g,(t,t)) dt\\
& -\sum_{v\nmid\infty}  
\CN^{(v)}_{\phi}(g,(1,1)) \log N_v -\sum_{v\nmid\infty} j_v(Z_*(g,\phi)_U1,1)\log N_v\\
& -\frac{1}{e} i_0(1,1)\, \Omega_\phi(g).
\end{align*}
Here we write $e=[O_E\cross:O_F\cross]$ for simplicity.
We already know that $j_v(Z_*(g,\phi)_U1,1)\neq 0$ only if $v$ is nonsplit in 
$\BB$. 

Group the terms in the difference as follows:
\begin{align*}
\mathcal D(g,\phi)
=& -2\sum_{v|\infty} \barint_{C_U}
(\overline\CK^{(v)}_{\phi}(g,(t,t))-\CM^{(v)}_{\phi}(g,(t,t))) dt\\
& -2 \sum_{v\nmid\infty \ \nonsplit} \barint_{C_U}
(\CK^{(v)}_{\phi}(g,(t,t))-\CM^{(v)}_{\phi}(g,(t,t))\log N_v) dt\\
&+2\sum_{v\in \Sigma_f}  j_{v}(Z_*(g,\phi)_U 1,  1)\log N_v \\
&+ \sum_{v\nmid\infty} \sumu \sum_{y\in E^\times} 
d_{\phi_v}(g,y,u)\ r(g)\phi^v(y,u)\\
& +(\frac{2}{e} i_0(1,1)-c_1) \Omega_\phi(g).
\end{align*}
Here
\begin{multline*}
d_{\phi_v}(g,y,u)=2n_{\phi_v}(g,y,u)\log N_v-c_{\phi_v}(g,y,u)+(2\log \delta(g)+ \log|uq(y)|_v)r(g)\phi_v(y,u), \\ 
\forall\ g\in\gl(F_v),\ (y,u)\in E_v^\times \times F_v^\times, \quad
v\nmid\infty.
\end{multline*}
The key term for us is the coefficient of 
$\Omega_\phi(g)$.

Every term in the expression of $\mathcal D(g,\phi)$ is a pseudo-theta series, and each summation over $v$ is just a finite sum. 
In fact, we have the following itemized result:
\begin{itemize}
\item[(1)] If $v|\infty$, then
$$\overline\CK^{(v)}_{\phi}(g,(t,t))-\CM^{(v)}_{\phi}(g,(t,t))=0.$$
This follows from \cite[Proposition 8.1]{YZZ}.
In the following cases, we assume that $v$ is non-archimedean. 
\item[(2)] If $v$ is nonsplit in $E$, then
$$k_{\phi_v}(1,y,u)-m_{\phi_v}(y,u)\log N_v$$
extends to a Schwartz function on $B(v)_v\times \fvcross$.
Furthermore, for all but finitely many such $v$,
$$
k_{\phi_v}(g,y,u)-m_{r(g)\phi_v}(y,u)\log N_v=0
$$
identically and thus
$$
\overline\CK^{(v)}_{\phi}(g,(t,t))-\CM^{(v)}_{\phi}(g,(t,t))=0.
$$
The second statement is just \cite[Proposition 8.8]{YZZ}.
The first statement is a consequence of Lemma 
\ref{derivative of Whittaker function} and Lemma \ref{local intersection}. 

\item[(3)] For any $v\nmid \infty$, the function
$$d_{\phi_v}(1,y,u)=2n_{\phi_v}(1,y,u)\log N_v-c_{\phi_v}(1,y,u)+ \log|uq(y)|_v\ \phi_v(y,u)$$
extends to a Schwartz function on $E_v\times \fvcross$.
Furthermore, for all but finitely many $v$,
$$
d_{\phi_v}(g,y,u)=0
$$
identically.
The first statement is a consequence of Lemma \ref{derivative of intertwining} and Lemma \ref{local intersection}.
From them, we see that $d_{\phi_v}(1,y,u)=0$ for all but finitely many $v$. The vanishing result extends to $d_{\phi_v}(g,y,u)$ by considering Iwasawa decompositions as in \cite[Proposition 8.8]{YZZ}.

\item[(4)] For any $v$ nonsplit in $\BB$, the $j$-part $j_{v}(Z_*(g,\phi)_U1,1)$ is a non-singular pseudo-theta series of the form 
$$
\sum_{u\in \mu_U^2\backslash F\cross} \sum_{y\in B(v)-\{0\}}
l_{\phi_v}(g,y,u)  r(g) \phi^v(y,u).
$$
This is Lemma \ref{superspecial j}. 
\end{itemize}

With these results, every term on the right-hand side of $\mathcal D(g,\phi)
$ is a non-singular pseudo-theta series. 
Therefore, we are finally ready to apply Lemma \ref{pseudo} (2).

The outer theta series associated to the pseudo-theta series
$$\Omega_\phi(g)=\sumu\sum_{y\in E\cross} r(g)\phi(y,u)$$
is exactly the weight-one theta series
$$\theta_{\Omega,1}(g)=\sumu\sum_{y\in E} r_E(g)\phi(y,u).$$
By Lemma \ref{pseudo} (2), there is a unique identity including this theta series, and we are going to write down this identity explicitly. 
This identity will be a sum of theta series of weight one. 
We look at the contribution of every term in the expression. 

The contribution of 
$$\CK^{(v)}_{\phi}(g,(t,t))-\CM^{(v)}_{\phi}(g,(t,t))\log N_v$$
to the weight-one identity comes from its inner theta series
$$\sumu\sum_{y\in E} r_E(g)\phi_v(y,u)\ r_E(g)(k_{\phi_v}(1,y,u)-m_{\phi_v}(y,u)\log N_v).$$
This sum does not change after averaging on $C_U$. 
The term $j_v(Z_*(g,\phi)_U1,1)$ does not contribute to the identity we want. 
The term 
$$
\sumu \sum_{y\in E^\times} 
d_{\phi_v}(g,y,u)\ r(g)\phi^v(y,u)
$$
contributes by its outer theta series
$$
 \sumu \sum_{y\in E} 
r_E(g)\phi^v(y,u)\ r_E(g)d_{\phi_v}(1,y,u).$$
Hence, we obtain the following identity
\begin{align*}
0=&\ 2 \sum_{v\nmid\infty \ \nonsplit} \sumu\sum_{y\in E} r_E(g)\phi_v(y,u)\ r_E(g)(k_{\phi_v}(1,y,u)-m_{\phi_v}(y,u)\log N_v)\\
&+ \sum_{v\nmid\infty} \sumu \sum_{y\in E} 
r_E(g)\phi^v(y,u)\ r_E(g)d_{\phi_v}(1,y,u) \\
& +(\frac{2}{e} i_0(1,1)-c_1) \sumu\sum_{y\in E} r_E(g)\phi(y,u).
\end{align*}

Now we need the following explicit local results. 
\begin{pro}\label{combination}
Let $v$ be a non-archimedean place and $(y,u)\in E_v\times \fvcross$. 
\begin{itemize}
\item[(1)] If $v$ is nonsplit in $E$, then
$$
2k_{\phi_v}(1,y,u)-2m_{\phi_v}(y,u)\log N_v+d_{\phi_v}(1,y,u)=-\log |d_vq(\jv)|_v\phi_v(y,u).
$$
\item[(2)] If $v$ is split in $E$, then
$$
d_{\phi_v}(1,y,u)
=-\log |d_vq(\jv)|_v \phi_v(y,u).
$$
\end{itemize}
\end{pro}
\begin{proof}
Recall that 
$$d_{\phi_v}(1,y,u)=2n_{\phi_v}(1,y,u)\log N_v-c_{\phi_v}(1,y,u)+ \log|uq(y)|_v\ \phi_v(y,u).$$
The proposition is just a combination of 
Lemma \ref{derivative of Whittaker function}, Lemma \ref{derivative of intertwining} and Lemma \ref{local intersection}. 
\end{proof}

Therefore, the identity gives exactly
\begin{align*}
0= \left(\sum_{v\nmid \infty }  -\log |d_vq(\jv)|_v  +\frac{2}{e}i_0(1,1)-c_1\right) \sumu\sum_{y\in E} r_E(g)\phi(y,u),
\end{align*}
which is just 
\begin{align*}
0= \left(\log |d_Fd_\BB |  +\frac{2}{e}i_0(1,1)-c_1\right) \theta_{\Omega, 1}(g).
\end{align*}
We claim that $\theta_{\Omega, 1}(g)$ is not identically zero. 
Then we get  
\begin{align*}
\log |d_Fd_\BB | +\frac{2}{e}i_0(1,1)-c_1=0,
\end{align*}
which proves Theorem \ref{result of comparison}.

It remains to check that the theta series
$$\theta_{\Omega,1}(g)=\sumu\sum_{y\in E} r_E(g)\phi(y,u)$$
is not identically zero. 
It suffices to check that the constant term 
$$\sumu r_E(g)\phi(0,u)$$
is not identically zero. 
For that, assume that for $v\in \Sigma_f$ or $v\in S_2$
$$g_v=\matrixx{}{1}{-1}{},$$
and $g_v=1$ at any other place $v$. 
By local computation, $r_E(g)\phi(0,1)> 0$ and $r_E(g)\phi(0,u)\geq 0$ for all $u\in F^\times$. Then the (finite) sum over $u$ is strictly positive. 
This shows that the theta series is nonzero.

\subsection{Arithmetic Adjunction Formula}   \label{section arithmetic adjunction formula}
Now we are going to relate
$$i_0(P,P)=i_0(1,1)=i(1,1)- \sum_{v} i_v(1,1) \log N_v$$
to the Faltings height. Here $i_v(1,1)=0$ if $v$ is split in $E$. 
It is essentially an arithmetic adjunction formula.
The main result of this subsection is:

\begin{thm}[Arithmetic adjunction formula]
\label{adjunction all places}
$$\frac{1}{[O_E\cross: O_F\cross]}i_0(P,P)=-h_{\overline\CL_U}(P).$$
\end{thm}

The theorem and Theorem \ref{result of comparison} implies  Theorem \ref{quaternion main}. The goal of this subsection is to prove the theorem. 

Denote by $H$ the Hilbert class field of $E$. 
Then $P=[1]_U$ is defined over $H$, and we view it as a rational point of $X_{U,H}$. 
By assumption, $E$ is unramified at any $v\in \Sigma_f$. 
By Corollary \ref{system2}, $\CX_{U,O_H}$ is $\BQ$-factorial. 
We will consider arithmetic intersections over $\CX_{U,O_H}$. 
We will suppress the symbol $U$ from the subscripts. For example, $\CX_{U,O_H}$ is written as $\CX_{O_H}$. 

Denote by $\CP$ the Zariski closure of $P$ in $\CX_{O_H}$. 
Then we have an arithmetic divisor 
$$
\bar\CP=(\CP, g_P),
$$
where the Green function $g_P=\{g_{P,w}\}_{w:H\to \BC}$ is the admissible Green function as in \cite[\S 7.1.5]{YZZ}. 
Denote by $\CO(\bar\CP)$ the corresponding hermitian line bundle. 
By definition, 
$$i(1,1)=\frac{1}{[H:F]}\pair{\bar\CP,\CP}
=\frac{1}{[H:F]}\wh\deg (\CO(\bar\CP)|_\CP).$$

Denote by $\bar\CL_{O_H}$ the base change of the arithmetic Hodge class $\bar\CL_U=\bar\CL$ from $\CX$ to $\CX_{O_H}$. 
It follows that 
$$
h_{\hat\omega}(1)= \frac{1}{[H:F]}\wh\deg (\bar\CL_{O_H}|_\CP).
$$
So the goal is to prove 
$$
\frac{1}{e}\wh\deg (\CO(\bar\CP)|_\CP)
+\wh\deg (\bar\CL_{O_H}|_\CP)
=[H:F]\frac{1}{e}\sum_{v} i_v(1,1) \log N_v.
$$
Here we denote $e=[O_E\cross: O_F\cross]$ for simplicity, which is also the ramification index $e_P$ of $P$. 
Rewriting the right-hand side according to places $w$ of $H$, the equality becomes
$$
\wh\deg \left(\bar\CM|_\CP\right)
=\frac{1}{e}\sum_{w} i_w(1,1) \log N_w.
$$
Here 
$$\bar\CM= \bar\CL_{O_H}\otimes \CO(e^{-1}\bar\CP)$$
is a hermitian $\QQ$-line bundle on $\CX_{O_H}$. 

Denote by $\CM$ and $M$ the finite part and the generic fiber of $\bar\CM$. 
We first claim that there is canonical isomorphism
$$
\Res_P: M|_P\lra H.
$$
In fact, by definition, 
$$
L= \omega_{X/F} \otimes \bigotimes_{Q\in X(\overline F)}  \CO_X((1-e_Q^{-1})Q).
$$
Then 
$$
M=L_{H}\otimes \CO(e^{-1}P)
= \omega_{\CX_{O_H}/O_H} \otimes \CO(P) 
\otimes \bigotimes_{Q\in X(\overline F), \ Q\neq P}  \CO_X((1-e_Q^{-1})Q).
$$
It follows that we have canonical isomorphisms
$$
M|_P \lra  (\omega_{\CX_{O_H}/O_H} \otimes \CO(P))|_P
\lra H.
$$
Here the second map is just the residue map 
$$
\frac{d u}{u}\otimes 1_P  \longmapsto 1,
$$
where $u$ is any local coordinate of $P$ in $X_H$, and $1_P$ denotes the section $1$ of $\CO(P)$.
The map does not depend on the choice of $u$. 

By the residue map
$\Res_P: M|_P\to H,$
we have an induced hermitian line bundle $\bar\CN=(\CN, \|\cdot\|)$ on $\Spec(O_H)$. 
Here $\CN$ denotes the image of $\CM|_\CP$ in $H$, which is a fractional ideal of $H$, and the metric on $\CN$ is determined by
$$
\|1\|_w= \left\|\frac{d u}{u}\otimes 1_P\right\|_w(P), \ \quad w:H\to \BC.
$$
Then we have
$$
\wh\deg \left(\bar\CM|_\CP\right)
=\wh\deg(\bar\CN)
=-\sum_{w:H\to \BC} \log \|1\|_w
+\sum_{w\nmid\infty} \dim_{k_w}(\CN_w/O_{H_w}) \log N_w.
$$
Here the second summation is over all non-archimedean places $w$ of $H$, $k_w$ denotes the residue field of $w$, and
$\dim_{k_w}(\CN_w/O_{H_w})$ means $-\dim_{k_w}(O_{H_w}/\CN_w)$ if 
$\CN_w$ is contained in $O_{H_w}$.
However, we will see that $\CN_w$ always contains $O_{H_w}$.

The theorem is reduced to the local identities
$$
-\log \|1\|_w=\frac{1}{e}  i_w(P,P), \quad w:H\to \BC,
$$
and 
$$
\dim_{k_w}(\CN_w/O_{H_w})=\frac{1}{e}  i_w(P,P), \quad w\nmid\infty.
$$
We will see that the ideas in different case are very similar even though the reductions are completely different.

\subsubsection*{Archimedean case}
We first check the local identity for archimedean case, so $w$ is an embedding $H\to \BC$. It restricts to an embedding $v: F\to \BC$.
We have a uniformization
$$X_v(\BC)= B_+\cross\bs \gh\times B\cross(\af)/U.$$
Here $B=B(v)$ is the nearby quaternion algebra.
Under the uniformization, the point $P$ is represented by $(z_0,t)$ for some 
$t\in E^\times(\af)$.
The metric $\|\cdot\|_w$ of $O(\bar\CP)$ is given by
$$-\log\|1_P\|_{w}([z,\beta])=i_{\bar v}([z,\beta],[z_0, t])$$
for any other point $[z,\beta]\in X_v(\BC)$ not equal to $[z_0,t]$. 
Here we recall from \cite[\S8.1]{YZZ} that 
$$i_{\bar v}([z, \beta], [z_0,t])=\quasilim
\sum_{\gamma\in \mu_U \backslash B_+\cross} m_s(z_0,\gamma z)  1_{U}(t^{-1}\gamma\beta),$$
where
$$m_s(z_0,z)=Q_s\left( 1+\frac{|z-z_0|^2}{2\Im(z_0)\Im(z)} \right).$$

Consider the covering map
$$\pi: \gh\times B\cross(\af)/U \longrightarrow X_v(\BC).$$
Here the left-hand side is just a countable disjoint union of 
$\gh$.
Denote by $\tilde P$ the point $(z_0,t)$ in this space, which is a lifting of $P=[z_0,t]$. 
By the construction of the Hodge bundle,  
$\pi^*L$ is canonically isomorphic to the sheaf $\Omega^1$ of holomorphic 1-forms on 
$\gh\times B\cross(\af)/U$. 
As a consequence, we have canonical isomorphisms
$$
(M|_P)\otimes_w\BC \lra (\pi^*M)|_{\tilde P}
= (\pi^*L_{H}\otimes \pi^*\CO(e^{-1}P))|_{\tilde P}
\lra (\Omega^1\otimes \CO(\tilde P))|_{\tilde P}
\lra \BC.
$$
Here the last map is a residue map again, and the whole composition is exactly the base change to $\BC$ of the original residue map $\Res_P:M|_P\to H$. 

Let $\tilde Q=(z_1,t)$ be a point of $\gh\times B\cross(\af)/U$, and $Q=[z_1,t]$ be its image in the quotient $X_v(\BC)$.
Consider the behavior as $z_1$ approaches $z_0$, which also means $\tilde Q\to \tilde P$ or $Q\to P$ in the complex topology. 
Let $z$ be the usual coordinate of $\gh\subset \BC$, so that $z-z_0$ gives a local coordinate at $\tilde P$ in $\gh\times B\cross(\af)/U$.
Then the second residue map gives
\begin{align*}
 \|1\|_w
= \lim_{\tilde Q\to \tilde P}   \left( \left\|\frac{d z}{z-z_0}\right\|_{\Pet}(\tilde Q)
\cdot \|1_P(Q)\|^{\frac1{e}} \right).
\end{align*}
Recall that the Petersson metric gives
$$
\left\|\frac{dz}{z-z_0}\right\|_{\Pet}(\tilde Q)=  \frac{2\,\Im(z_1)}{|z_1-z_0|}.
$$
On the other hand, the Green function
\begin{align*}
& -\log\|1_P\|_{w}(Q)\\
=&i_{\bar v}([z_1,t],[z_0, t])\\
=&\quasilim \sum_{\gamma\in \mu_U \backslash B_+\cross} m_s(z_0,\gamma z_1)  1_{U}(t^{-1}\gamma t)\\
=& e\cdot m_0(z_0,z_1)+\quasilim
\sum_{\gamma\in \mu_U \backslash (B_+\cross-E\cross)} m_s(z_0,\gamma z_1)  1_{U}(t^{-1}\gamma t).
\end{align*}
The definition has been extended to self-intersection as
$$i_{\bar v}([z_0, t], [z_0,t])=
\quasilim
\sum_{\gamma\in \mu_U \backslash (B_+\cross-E\cross)} m_s(z_0,\gamma z_0)  1_{U}(t^{-1}\gamma t).$$
Hence,
\begin{align*}
-\log \|1\|_w
=& \lim_{z_1\to z_0}\left(m_0(z_0,z_1)-\log \frac{2\,\Im(z_1)}{|z_1-z_0|}\right)
+\frac1e i_{\bar v}(P, P).
\end{align*}
It remains to check that the limit on the right-hand side is exactly zero. 

Note that 
$$m_0(z_0,z_1)=Q_0\left( 1+\frac{|z_1-z_0|^2}{2\Im(z_0)\Im(z_1)} \right).$$
By \cite[II, (2.6)]{GZ},
$$Q_0(t)= \frac{1}{t+1}F(1,1,2,  \frac{2}{t+1})= \frac12 \log\frac{t+1}{t-1}.$$
It follows that
$$
m_0(z_0,z_1)-\log \frac{2\,\Im(z_1)}{|z_1-z_0|}
=
\frac 12 \log \left(1+\frac{|z_1-z_0|^2}{4\Im(z_0)\Im(z_1)} \right)
- \frac 12\log \frac{\Im(z_1)}{\Im(z_0)},
$$
which converges to $0$ as $z_1\rightarrow z_0$.
This finishes the archimedean case.

\subsubsection*{Non-archimedean case}

Let $w$ be a non-archimedean place of $H$. 
Let $v$ be the restriction of $w$ to $F$. 
To prove the arithmetic adjunction formula, 
the key is the following geometric interpretation of the extended intersection 
$i_w(P,P)=i_{\bar v}(P,P)$. 
For convenience, denote by $R=O_{H_w^\ur}$ the integer ring of the completion $H_w^\ur$ of the maximal unramified extension of $H_w$. 

\begin{lem} \label{self-int}
Let $U'=U_vU'^v$ be an open compact subgroup of $\bbf$ with 
$U'^v\subset U^v$ normal. 
Consider the projection $\pi: \CX_{U', R}\to \CX_{U, R}$. 
Denote by $\CP'$ an irreducible component of the divisor $\pi^{-1}\CP_R$ 
on $\CX_{U', R}$. 
If $U'^v$ is small enough, then
$$
i_w(P,P)=\langle \pi^{-1}\CP_R-e\CP',\ \CP'\rangle.
$$
Here the pairing denotes the intersection multiplicity on the special fiber of $\CX_{U', R}$. 
\end{lem}

In the lemma, the morphism $\pi$ is \'etale, so $\CP'$ must be a section of 
$\CX_{U', R}$ over $R$. The ramification index of $P$ is $e$. 
Then the multiplicity of $\CP'$ in $\pi^{-1}\CP$ is $e$ if $U'^v$ is small enough, 
so the intersection in the lemma is a proper intersection. 
The lemma can be viewed as a modified projection formula. 
We will prove it later, but let us first use it to finish the proof of the arithmetic adjunction formula. 

Recall that it is reduced to the local identity
$$
\dim_{k_w}(\CN_w/O_{H_w})=\frac{1}{e}  i_w(P,P).
$$
Here $\CN$ denotes the image of $\CM|_{\CP}$ under the residue map 
$$
\Res_P: M|_P \lra H
$$

As in the archimedean case, we will use have a different interpretation of the residue map. Let
$\pi: \CX_{U', R}\to \CX_{U, R}$ and $\CP'$ be as in the lemma. 
Denote by $P'$ the generic fiber of $\CP'$.
By the definition of the Hodge bundle, we have canonical isomorphisms
$$
\pi^*L_{U,H_w^\ur} \lra \omega_{X_{U',H_w^\ur}/H_w^\ur}, \quad
\pi^*\CL_{U,R} \lra \omega_{\CX_{U',R}/R}.
$$
Thus we have canonical isomorphisms
$$
(M|_P)\otimes_H H_w^\ur
\lra (\pi^*L_{U, H_w^\ur}\otimes \pi^*\CO(e^{-1}P))|_{P'}
\lra (\omega_{X_{U',H_w^\ur}/H_w^\ur} \otimes \CO(P'))|_{P'}
\lra H_w^\ur.
$$
Here the last map is a residue map again, and the whole composition is exactly the base change to $H_w^\ur$ of the original residue map $\Res_P:M|_P\to H$.

The computation is to track the change of integral structures of the composition. 
The composition has the integral version
$$
(\CM|_\CP)\otimes_{O_H} R
\lra (\pi^*\CL_{U, R}\otimes \pi^*\CO(e^{-1}\CP))|_{\CP'}
\dashedrightarrow (\omega_{\CX_{U',R}/R} \otimes \CO(\CP'))|_{\CP'}
\lra R.
$$
The first arrow is an isomorphism by definition, and the last arrow is an isomorphism by the adjunction formula on $\CX_{U',R}$. 
The dashed arrow in the middle may only be well-defined map after base change to $H_w^\ur$, but we write it this way to track the change of the integral structure.  
Thus $\dim_{k_w}(\CN_w/O_{H_w})$ is equal to the dimension of the quotient of two sides of the dashed arrow. 
Tensoring with $(\pi^*\CL_{U, R}|_{\CP'})^{\otimes(-1)}$, 
the dashed arrow becomes
$$
\pi^*\CO(e^{-1}\CP)|_{\CP'}
\dashedrightarrow  \CO(\CP')|_{\CP'}.
$$
Tensoring with $\pi^*\CO(-e^{-1}\CP)|_{\CP'}$, it further becomes
$$
\CO_{\CP'} \dashedrightarrow  \CO(\CP'-e^{-1}\pi^*\CP)|_{\CP'} .
$$
Note that $e^{-1}\pi^*\CP-\CP'$ is an effective divisor. The real map should be the inverse direction
$$
\CO(\CP'-e^{-1}\pi^*\CP)|_{\CP'} \lra \CO_{\CP'}.
$$
The image of the last map is the restriction of the ideal sheaf of $e^{-1}\pi^*\CP-\CP'$ to $\CP'$, so the cokernel of the map has dimension exactly equal to the intersection number
$$
\langle\CO(e^{-1}\pi^*\CP-\CP', \ \CP'\rangle.
$$
Hence,
$$
\dim_{k_w}(\CN_w/O_{H_w})
=\langle\CO(e^{-1}\pi^*\CP-\CP', \ \CP'\rangle.$$
By Lemma \ref{self-int}, it further equals
$$
\frac{1}{e}  i_w(P,P).
$$
This finishes the proof of the adjunction formula.

\subsubsection*{Proof of the lemma}

Here we prove Lemma \ref{self-int}. 
Let $U'=U_vU'^v$ be as in the lemma.
Recall that if $v$ is nonsplit in $E$,
\begin{eqnarray*}
i_{\bar v}([1]_{U'}, [1]_{U'})
= \sum _{\gamma \in \mu_{U'}\bs (B^{\times}-E^{\times}\cap  U')}
m(\gamma,1) 1_{U'^v}(\gamma).
\end{eqnarray*}
Here $B=B(v)$, and the multiplicity function $m:B^{\times}_v\times_{E^\times_v}\bb_v^\times/U_v\to \BQ$
takes the same form for $U$ and $U'$. 
The key is the following result. 

\begin{lem} \label{self-int2}
If $v$ is nonsplit in $E$, then $i_{\bar v}([1]_{U'}, [1]_{U'})=0$ if $U'^v$ is small enough. 
\end{lem}
\begin{proof}
Note that $m(\gamma,1)$, as a function in $\gamma$, is supported on an open compact subgroup $W_v$ of $B_v^\times$. 
In fact, by $q(\gamma)\in O_{F_v}^\times$, we can take $W_v=O_{B_v}^\times$ if $v$ is nonsplit in $B$, and $W_v$ still exists if $v$ is split in $B$ by Lemma \ref{superspecial multiplicity}. 
Then $\gamma$ contributes to the summation only if $\gamma\in B^{\times} \cap W$. 
Here we write $W=W_vU'^v$ as a open compact subgroup of $B^\times(\af)$. 
Since $B$ is totally definite, $\mu_{W}$ has finite index in $B^{\times} \cap W$. 
Let $S$ be set of representatives of the nontrivial cosets of $B^{\times} \cap W/\mu_W$.
Shrinking $U'^v$ if necessary, we can keep $\mu_{W}$ invariant, but make $S\cap U'^v$ empty. 
Hence, we end up with $B^{\times} \cap W=\mu_W$.  
It follows that $B^{\times} \cap W \subset E^{\times}\cap  U'$. 
Then the sum for $i_{\bar v}([1]_{U'}, [1]_{U'})$ has no nonzero terms. 
\end{proof}

Now we prove Lemma \ref{self-int}. 
By the right multiplication of $U$ on $X_{U'}$, it is easy to see that the Galois group of $X_{U'}\to X_U$ is isomorphic to $U/(U'\mu_U)$. 
It follows that 
$$
\pi^{-1}(P)=\pi^{-1}([1]_U)=\sum_{\beta\in U/(U'\mu_U)}[\beta]_{U'}
=\frac{1}{[\mu_U:\mu_{U'}]}\sum_{\beta\in U/U'}[\beta]_{U'}.
$$
Denote $P'=[1]_{U'}$, and we can assume that $\CP'$ is the Zariski closure of $P'$ since the intersection multiplicity in the lemma does not depend on the choice of $\CP'$ by the action of the Galois group of $X_{U'}\to X_U$.
Assume that $U'$ satisfies Lemma \ref{self-int2}; i.e.,  $i_{\bar v}(P',P')=0$.
Then 
$$
\langle \pi^{-1}\CP-e\CP',\ \CP'\rangle
=i_{\bar v}(\pi^{-1}P-eP',P')=i_{\bar v}(\pi^{-1}P,P')
$$
It is reduced to check 
$$
i_{\bar v}(\pi^{-1}P,P')=i_{\bar v}(P,P).
$$
Here both sides use our extended definitions. 
It is straightforward by the expression of $\pi^{-1}(P)$ above.

We first assume that $v$ is nonsplit in $E$.  
Recall that for any $\beta\in\bfcross$,
\begin{eqnarray*}
i_{\bar v}([\beta]_{U'}, [1]_{U'})
= \sum _{\gamma \in \mu_{U'}\bs (B^{\times}-E^{\times}\cap \beta_v U_v)}
m(\gamma, \beta_{v}^{-1}) 1_{U'^v}((\beta^v)^{-1}\gamma).
\end{eqnarray*}
Then 
\begin{eqnarray*}
i_{\bar v}(\pi^{-1}P,P')
&=&\frac{1}{[\mu_U:\mu_{U'}]} \sum_{\beta\in U/U'} i_{\bar v}([\beta]_{U'}, [1]_{U'}) \\
&=& \frac{1}{[\mu_U:\mu_{U'}]} \sum_{\beta\in U^v/U'^v}  
\sum _{\gamma \in \mu_{U'}\bs (B^{\times}-E^{\times}\cap  U_v)}
m(\gamma, 1) 1_{U'^v}(\beta^{-1}\gamma) \\
&=& \frac{1}{[\mu_U:\mu_{U'}]} 
\sum _{\gamma \in \mu_{U'}\bs (B^{\times}-E^{\times}\cap  U_v)}
m(\gamma, 1) 1_{U^v}(\gamma) \\
&=& i_{\bar v}(P,P).
\end{eqnarray*}
This finishes the nonsplit case. 

It remains to treat the case that $v$ is split in $E$. 
In this case, Lemma \ref{self-int2} is automatic, since  $i_{\bar v}(P',P')=0$ is actually true for any $U'$. 
The proof is similar to the nonsplit case by the formula
\begin{eqnarray*}
i_{\bar v}([\beta]_{U'}, [1]_{U'})
= \sum_{\gamma \in \mu_{U'}\bs (E\cross-\beta_v U_v)}
m_{\bar v}(\gamma^{-1}\beta) 1_{U^v}(\beta^{-1}\gamma). 
\end{eqnarray*}
It is also similar to the second half of the proof of Lemma \ref{ordinary multiplicity}.
An interesting consequence is that both sides of Lemma \ref{self-int} are 0.

\end{document}